\documentclass[11pt]{article}
\usepackage{enumitem}
\usepackage{amsmath,amsthm,amssymb,amsfonts}
\usepackage{amsmath}
\usepackage{multirow}
\usepackage{graphicx}
\usepackage{hyperref}
\usepackage[pdftex]{color}
\usepackage{cancel}
\usepackage[toc,page]{appendix}
\usepackage{fullpage}
\usepackage[margin=0.75in]{geometry}
\usepackage[utf8]{inputenc}

\theoremstyle{plain}
\newtheorem{theorem}{Theorem}[section]

\newtheorem{proposition}[theorem]{Proposition}
\newtheorem{lemma}[theorem]{Lemma}
\newtheorem{corollary}[theorem]{Corollary}
\newtheorem*{claim}{Claim}

\theoremstyle{definition}

\newtheorem{remark}[theorem]{Remark}
\newtheorem*{theorem*}{Theorem}

\numberwithin{equation}{section}

\newcommand{\AAA}{\color{black}}
\newcommand{\EEE}{\color{black}}

\def\R{\mathbb{R}}

\def\S{\mathbb{S}^{n-1}}

\def\N{\mathbb{N}}

\def\t{\nabla_{T}}

\def\Dt{\Delta_{\S}}

\def\L2{L^2(\S)}

\def\s{\mathbb{S}^2}
\def\ds{\dashint_{\mathbb{S}^{n-1}}}
\def\d2{\dashint_{\mathbb{S}^2}}

\def\Xint#1{\mathchoice
{\XXint\displaystyle\textstyle{#1}}%
{\XXint\textstyle\scriptstyle{#1}}%
{\XXint\scriptstyle\scriptscriptstyle{#1}}%
{\XXint\scriptscriptstyle\scriptscriptstyle{#1}}%
\!\int}
\def\XXint#1#2#3{{\setbox0=\hbox{$#1{#2#3}{\int}$ }
\vcenter{\hbox{$#2#3$ }}\kern-.6\wd0}}

\def\dashint{\Xint-}

\def\ds{\dashint_{\S}}

\begin{document}

\centerline{{\Large {\bf Rigidity estimates for isometric and conformal maps from $\S$ to $\R^n$}}}

\vspace*{0.5  cm}

\centerline{{\large{\bf Stephan Luckhaus$^1$, \ Konstantinos Zemas$^2$}}}
 
\vspace*{0.5  cm}

\begin{abstract}
\footnotesize We investigate both linear and nonlinear stability aspects of rigid motions (resp. Möbius transformations) of $\S$ among Sobolev maps from $\S$ into $\R^n$. Unlike similar in flavour results for maps defined on domains of $\R^n$ and mapping into $\R^n$, not only an isometric (resp. conformal) deficit is necessary in this more flexible setting, but also a deficit measuring the distortion of $\S$ under the maps in consideration. The latter is defined as an associated isoperimetric type of deficit. The focus is mostly on the case $n=3$ (where it is explained \EEE why the estimates are optimal in their corresponding settings), but we also address the necessary adaptations for the results in higher dimensions. 
\EEE We also obtain linear stability estimates for both cases in all dimensions. These can be regarded as Korn-type inequalities for the combination of the quadratic form associated with the isometric (resp. conformal) deficit on $\S$ and the isoperimetric one.\\[-10pt]
\end{abstract}



\ \ \ \ \ \  \ \ \ \ \ \ \  \ \ \ \  conformal-isoperimetric deficit, stability

\section{Introduction}\label{Introduction}

In this paper we examine stability issues of isometric and conformal maps from $\S$ into $\R^n$ of relatively low regularity, focusing mostly, but not solely, on the case $n=3$. Since the starting domain is of codimension 1 in $\R^n$, these maps exhibit of course more flexibility than their analogues from open subdomains of $\R^n$ into $\R^n$. On the one hand, isometric and conformal maps are actually \textit{rigid} \EEE when considered from $\S$ into itself, as the following version of the well known theorem by {\sc J. Liouville} asserts.

\begin{theorem}\label{spherical Liouville's Theorem}
\textbf{$\rm{(}$Liouville's Theorem on $\S$\rm{)}}\label{Liouville_on_the_sphere_isometric_conformal_case} \\
$(i)$ Let $n\geq 2$ and $p\in [1,+\infty]$. A generalized orientation-preserving $\rm{(}\backslash$-reversing$\rm{)}$ $u\in W^{1,p}(\S;\S)$ is isometric iff it is a rigid motion of $\S$, i.e., iff there exists $O\in O(n)$ so that for every $x\in \S$, 
\begin{equation}\label{isometric_group_of_S_representation}
u(x)=Ox\,. 
\end{equation}
$(ii)$ Let $n\geq 3$. A generalized orientation-preserving $\rm{(}\backslash$-reversing$\rm{)}$ $u\in W^{1,n-1}(\S;\S)$ of degree 1  $\rm{(}\backslash$-1$\rm{)}$ is conformal iff it is a Möbius transformation of $\S$, i.e., iff there exist $O\in O(n)$, $\xi\in\S$ and $\lambda>0$ so that for every $x\in \S$, 
\begin{equation}\label{conformal_group_of_S_representation}
u(x)=O\phi_{\xi,\lambda}(x)\,. 
\end{equation}
Here, $\phi_{\xi,\lambda}:=\sigma_\xi^{-1}\circ i_\lambda\circ\sigma_\xi$, where $\sigma_\xi$ is the stereographic projection of $\S$ onto ${T_\xi\S}\cup\{\infty\}$, and \linebreak $i_\lambda:T_\xi\S\mapsto T_\xi\S$ is the dilation in $T_ \xi\S$ by factor $\lambda>0$.
\end{theorem}

On the other hand however, there is a wide variety of such maps from $\S$ into $\R^n$. In contrast to the classical rigidity in the \textit{Weyl problem for isometric embeddings}, according to which the only $C^2$ (or even $C^{1,\alpha}$ for $\alpha>\frac{2}{3}$) isometric embedding of $\S$ into $\R^n$ is the standard one modulo rigid motions (cf.~\cite{borisov1959},\cite{cohn1927zwei}, \cite{conti2012h}, \cite{herglotz1943starrheit}), as a consequence of the celebrated \textit{Nash-Kuiper theorem} (cf.~\cite{kuiper1955c1},\cite{nash1954c1}), the following paradox happens for less regular, say $C^1$ isometric embeddings.

Given any $\delta\in(0,1)$, in an arbitrarily small $C^0$-neighbourhood of the short homothety $u_\delta:\S\mapsto\R^n$, $u_\delta(x):=\delta x$, there exist $C^1$ isometric embeddings, which can be visualized as wrinkling isometrically $\S$ inside the small ball $B_\delta(0)$ in a way that produces continuously changing tangent planes. For the more general case of conformal maps from $\S$ to $\R^n$, at least when $n=3$, other examples that are not Möbius transformations are provided by the Uniformization Theorem and some of them have often been used in cartography, for instance the inverse of Jacobi's conformal map projection that smoothly and conformally maps $\mathbb{S}^2$ onto the surface of an ellipsoid.

Therefore, \textit{Liouville's rigidity theorem on $\S$} on the one hand, and the aforementioned \textit{flexibility phenomena} on the other, indicate the following fact. When one seeks stability of the isometry (resp. the conformal) group of $\S$ among Sobolev maps $u\colon\S\mapsto\R^n$, apart from an \textit{isometric} (resp. \textit{conformal}) deficit, an extra deficit measuring the deviation of $u(\S)$ from being a round sphere is necessary. \EEE In this paper we make a connection between stability aspects for these two classes of mappings and the isoperimetric inequality, and this extra deficit should be interpreted in both cases as an \textit{isoperimetric type of deficit} produced by the maps in consideration.

With the notations that we adopt in Section \ref{Notation}, our main result in the \textit{isometric case} is the following.\EEE
\begin{theorem}\label{main_thm_isometric_case}
There exists $c_1>0$ so that for every $u\in W^{1,2}(\s;\R^3)$ there exists $O\in O(3)$ such that
\begin{equation}\label{main_estimate_isometric_case}
\d2 \big|\t u-OP_T\big|^2\leq c_1\Big(\big\|\big(\sigma_{2}-1\big)_+\big\|_{L^2(\s)}+\Big(1-\big|V_3(u)\big|\Big)_+\Big)\,,
\end{equation}
where $0\leq \sigma_1\leq \sigma_{2}$ are the principal stretches of u, i.e., the eigenvalues of $\sqrt{\t u^t \t u}$, and
\begin{equation}\label{signed_volume_3_dim}
V_3(u):=\d2\big\langle u, \partial_{\tau_1}u\wedge\partial_{\tau_2}u\big\rangle
\ \text{is the signed volume of } u\,.
\end{equation}
\end{theorem}
The first term on the right hand side of \eqref{main_estimate_isometric_case} is an \textit{$L^2$-isometric deficit of $u$ penalizing local stretches}, while the second term (in the definition of which in \eqref{signed_volume_3_dim} we use the identification between a $2$-simple vector and its Hodge dual) represents in this setting the \textit{isoperimetric deficit} of $u$. Since isometric maps preserve the surface area of $\s$, 
the latter reduces in this situtation to \textit{the positive part of the excess in the signed volume} produced by $u$. The exact analogue of Theorem \ref{main_thm_isometric_case} holds true also in dimension $n=2$ (see Proposition \ref{main_estimate_isometric_case_n=2} in Section \ref{sec: isom}) and, as long as $u$ satisfies an apriori bound on its homogeneous $W^{1,2(n-2)}$-seminorm, also in dimensions $n\geq 4$, as stated in the following.

\begin{theorem}\label{main_thm_isometric_case_n_greater_than_3}
Let $n\geq 4$ and $M>0$. There exists $c_{n,M}>0$ so that for every $u\in \dot{W}^{1,2(n-2)}(\S;\R^n)$ with $\|\t u\|_{L^{2(n-2)}(\S)}\leq M$, there exists $O\in O(n)$ such that
\begin{equation}\label{main_estimate_isometric_case_n_greater_than_3}
\ds \left|\t u-OP_T\right|^2\leq c_{n,M}\left(\big\|(\sigma_{n-1}-1)_+\big\|_{L^2(\S)}+\big(1-|V_n(u)|\big)_+\right),
\end{equation}
where $0\leq \sigma_1\leq\dots\leq \sigma_{n-1}$ are again the eigenvalues of $\sqrt{\t u^t \t u}$,
and the signed volume of $u$ is now 
\begin{equation}\label{definition_of_signed_volume}
V_n(u):=\ds \Big\langle u, \bigwedge_{i=1}^{n-1} \partial_{\tau_i}u\Big\rangle\,.
\end{equation} 
\end{theorem}
Let us clarify that here we are using the identification 
\begin{equation}\label{determinant_of_signed_volume}
\Big\langle u,\bigwedge_{i=1}^{n-1}\partial_{\tau_i}u\Big\rangle:=\Big\langle u,\ast\bigwedge_{i=1}^{n-1}\partial_{\tau_i}u\Big\rangle=\mathrm{det}\begin{pmatrix}
u^1 & \dots & u^n\\
\partial_{\tau_1}u^1 & \dots & \partial_{\tau_1}u^n\\
&\vdots&\\
\partial_{\tau_{n-1}}u^1 & \dots & \partial_{\tau_{n-1}}u^n
\end{pmatrix}\,.
\end{equation}  

The constant in \eqref{main_estimate_isometric_case_n_greater_than_3} depends in principle now both on the dimension and on the apriori bound in the $L^{2(n-2)}$-norm of the gradient. The reason why this particular condition is introduced will be explained in Subsection \ref{Subsection 3.4}. As we also justify by examples in Remark \ref{examples}, the estimate is optimal in this setting, in the sense that \textit{the exponents with which the two deficits appear cannot generically be improved.}

For the \textit{conformal case}, due to the scaling invariant nature of the problem, the correct notions for the average conformal deficit and the isoperimetric one can be combined together. The main result when $n=3$ in this case is the following.

\begin{theorem}\label{main_thm_conformal_case}
There exists a constant $c_2>0$ so that for every $u\in W^{1,2}(\s;\R^3)$ with $V_3(u)\neq 0$ there exist a Möbius transformation $\phi$ of $\s$  and $\lambda >0$ such that
\begin{equation}\label{main_estimate_conformal_case}
\d2 \Big|\frac{1}{\lambda}\t u-\t \phi\Big|^2\leq c_2\left(\frac{\Big[D_2(u)\Big]^{\frac{3}{2}}}{|V_3(u)|}-1\right)\,,
\end{equation}
where $D_2(u):=\frac{1}{2}\d2|\t u|^2$ is the Dirichlet energy of $u$, and $V_3(u)$ is again its signed volume, \AAA as in \eqref{signed_volume_3_dim}\EEE.
\end{theorem}		
Of course the question is void when $n=2$, since conformality is a trivial notion for maps from $\mathbb{S}^1$ to $\R^2$. One can directly check that the estimate \eqref{main_estimate_conformal_case} is again optimal in its setting, by considering the sequence of maps $u_{\sigma}(x):=A_\sigma x:\s\mapsto \R^3,\ \mathrm{where}\ A_\sigma:=\mathrm{diag}(1,1,1+\sigma)\in \R^{3\times 3}$ as $\sigma\to 0^+$.

The use of this combined conformal-isoperimetric deficit is very natural in this framework. Indeed, generalizing to any dimension $n\geq 3$ (for $n=3$ cf.~\cite[Theorem 2.4]{wente1969existence}), for $u\in W^{1,n-1}(\S;\R^n)$ the following inequalities, sometimes referred to as \textit{Wente's isoperimetric inequality for mappings}, are known to hold.
\begin{align}\label{AM-GM-isoperimetric-inequality}
\begin{split}
&\ds\left(\frac{|\t u|^2}{n-1}\right)^\frac{n-1}{2}\geq\ds\sqrt{\mathrm{det}\left(\t u^t\t u\right)}\geq \left|\ds \Big\langle u,\bigwedge_{i=1}^{n-1}\partial_{\tau_i}u\Big\rangle\right|^{\frac{n-1}{n}}\\[5pt]
\iff& \Big[D_{n-1}(u)\Big]^{\frac{n}{n-1}}\ \ \ \ \ \ \ \geq \Big[P_{n-1}(u)\Big]^{\frac{n}{n-1}} \ \ \ \ \ \ \ \ \ \ \ \ \   \geq \big|V_n(u)\big|\,,
\end{split}
\end{align}
where $D_{n-1}(u)$, $P_{n-1}(u)$ are the first two integral quantities in the first line of the above inequalities, i.e., the $(n-1)$-Dirichlet energy and the generalized area produced by $u$ respectively. 
The first inequality in \eqref{AM-GM-isoperimetric-inequality} follows from the arithmetic mean-geometric mean inequality for the eigenvalues of $\sqrt{\t u^t\t u}$ and equality is achieved iff these eigenvalues coincide for $\mathcal{H}^{n-1}$-a.e. $x\in \S$, or equivalently, iff
\begin{equation*}
(\t u)^t\t u=\frac{|\t u|^2}{n-1}I_x\quad \mathcal{H}^{n-1}\text{-a.e. on } \S,
\end{equation*}
i.e., iff $u$ is a generalized conformal map from $\S$ to $\R^n$. The second inequality in \eqref{AM-GM-isoperimetric-inequality} is \textit{the functional form of the isoperimetric inequality} (cf.~\cite[inequality (2)]{almgren1986optimal}), which can be proven first for smooth maps and can then be extended by density in $W^{1,n-1}(\S;\R^n)$. Equality is achieved \EEE iff the image of $u$ is another round sphere in the $\mathcal{H}^{n-1}$-a.e. sense. In the case of a $C^1$ embedding, the inequality reduces of course to the classical Euclidean isoperimetric inequality for the open bounded set in $\R^n$ whose boundary is $u(\S)$.

Based on these simple observations, \textit{the combined conformal-isoperimetric deficit}
\begin{equation}\label{def_of_combined_deficit}
\mathcal{E}_{n-1}(u):=\frac{\Big[D_{n-1}(u)\Big]^{\frac{n}{n-1}}}{\big|V_n(u)\big|}-1\,,
\end{equation}
considered among maps $u\in W^{1,n-1}(\S;\R^n)$ for which $V_n(u)\neq 0$, provides a correct notion of deficit when one seeks stability of the conformal group of $\S$ among maps from $\S$ into $\R^n$. Indeed, it is immediate that $\mathcal{E}_{n-1}$ is translation, rotation and scaling invariant, as well as invariant under precompositions with Möbius transformations of $\S$. Moreover, as we have discussed above, $\mathcal{E}_{n-1}$ is nonnegative and vanishes iff $u$ is a generalized conformal map from $\S$ onto another round sphere, which after translation and scaling can be taken to be $\S$ again. If $d\in \mathbb{Z}$ would denote the degree of $u\in W^{1,n-1}(\S;\S)$ (following the definitions in \cite{brezis1995degree}), then  
$$|d|=|V_n(u)|=[D_{n-1}(u)]^{\frac{n}{n-1}}\geq |V_n(u)|^{\frac{n}{n-1}}=|d|^{\frac{n}{n-1}}\,.$$
Since \EEE the degree (for maps from $\S$ to itself) takes integer values, we would have that either $d=0$ or $d=\pm 1$, with the first case being excluded automatically, since by assumption $V_n(u)\neq 0$. Hence, absolute minimizers of $\mathcal{E}_{n-1}$ are degree $\pm 1$ conformal maps from $\S$ into itself, up to a translation vector and a scaling factor, i.e., according to Theorem \ref{spherical Liouville's Theorem}, Möbius transformations of $\S$ up to translation and scaling.
In this respect, Theorem \ref{main_thm_conformal_case} can be thought of as a \textit{sharp quantitative version} of the previous statements for $n=3$. At the core of its proof lies the study of the linearized version of the problem, since by the use of a \textit{contradiction$\backslash$compactness argument} it is enough to show the theorem for maps that are sufficiently close to the $\mathrm{id}_{\s}$ in the $W^{1,2}$-topology. In this regime, and \textit{after a correct rescaling of $u$}, if $w:=u-\mathrm{id}_{\s}$ is the corresponding displacement field, one obtains the formal Taylor expansion
\begin{equation}\label{intro_formal_expansion_conf_3}
\mathcal{E}_2(u)=Q_3(w)+o\left(\d2 |\t w|^2\right)\,,
\end{equation}
where $Q_3(w)$ is the associated quadratic form, i.e., the second derivative of $\mathcal{E}_2$ at the $\mathrm{id}_{\s}$, defined explicitely later in \eqref{quadratic_forms_without_proofs_}. The next and main step of the proof is to examine the coercivity properties of the quadratic form $Q_3$. This is something that can actually be done in every dimension $n\geq 3$, the main ingredient for doing so being the fine interplay between the Fourier decomposition of a $W^{1,2}(\S;\R^n)$-vector field into $\R^n$-valued spherical harmonics and the properties of the linear first order differential operator 
associated to the second derivative of 
$V_n$ at the $\mathrm{id}_{\S}$.

To be more precise, as we thoroughly examine in Subsection \ref{Subsection 4.2} for the case $n=3$, and in Subsection \ref{Subsection 5.1.} for the higher dimensional case, if one rescales $u$ properly, sets $w:=u-\mathrm{id}_{\S}$ and expands $\mathcal{E}_{n-1}(u)$ \AAA in \eqref{def_of_combined_deficit} \EEE around the $\mathrm{id}_{\S}$, then the resulting quadratic form 
\begin{equation}\label{Q_n_definition}
Q_n(w):=\frac{n}{2(n-1)}\ds \Big(|\t w|^2+\frac{n-3}{n-1}(\mathrm{div}_{\S}w)^2\Big)-\frac{n}{2}\ds \Big\langle w, (\mathrm{div}_{\S}w)x- \sum_{j=1}^nx_j\t w^j\Big\rangle
\end{equation}
has finite-dimensional kernel and its dimension actually coincides with that of the conformal group of $\S$. Moreover, when considered in the correct space (see the definitions of the spaces $H_n$, $(H_{n,k,i})_{k\geq 1, i=1,2,3}$ in equation \eqref{H_space} and Theorem  \ref{eigenvalue_decomposition} in Subsection  \ref{Subsection 4.2}), the form $Q_n$ satisfies the following coercivity estimate.

\begin{theorem}\label{main_coercivity_estimate_conformal_general_dimension}
Let $n\geq 3$. There exists a constant $C_n>0$ such that for every $w\in H_n$,
\begin{equation}\label{Korn_type_inequality_for_Q_n}
Q_n(w)\geq C_n\ds \big|\t w-\t(\Pi_{n,0} w)\big|^2\,,
\end{equation}
where $H_{n,0}:=H_{n,1,2}\oplus H_{n,2,3}$ is the kernel of $Q_n$ in $H_n$, and $\Pi_{n,0}\colon H_n\mapsto H_{n,0}$ is the $W^{1,2}$-orthogonal projection of $H_n$ onto $H_{n,0}$.
\end{theorem}
When $n=3$, \textit{the optimal constant} in \eqref{Korn_type_inequality_for_Q_n} can actually be calculated explicitely. Since $H_{n,0}$ turns out to be \textit{isomorphic to the Lie algebra of infinitesimal Möbius transformations of $\S$}, an application of the Inverse Function Theorem together with a topological argument (given in Lemma  \ref{fixingMobius}) will finally allow us to infer the \textit{nonlinear estimate} \eqref{main_estimate_conformal_case} from the \textit{linear one} \eqref{Korn_type_inequality_for_Q_n} in the \textit{$W^{1,2}$-close to the $\mathrm{id}_{\s}$-regime}, and hence conclude with Theorem  \ref{main_thm_conformal_case}.

It is maybe worth remarking here that\AAA, in contrast to \eqref{intro_formal_expansion_conf_3}, \EEE in dimensions $n\geq 4$ a formal expansion of the combined conformal-isoperimetric deficit around the $\mathrm{id}_{\S}$ yields 
\begin{equation}\label{expansion_conf_higher_dim}
\mathcal{E}_{n-1}(u)=Q_n(w)+\mathcal{O}\left(\ds |\t w|^3\right)\,.
\end{equation} 
Since the higher order term is now cubic in $\t w$, the linear estimate \eqref{Korn_type_inequality_for_Q_n} alone would only imply the nonlinear one (following exactly the same steps of proof as those described in Subsections \ref{Subsection 4.1}, \ref{Sec: completion_of_proof_conformal_case_n_3} and \ref{subsec: 4.4} for the case $n=3$) only in the \textit{$W^{1,\infty}$-close to the $\mathrm{id}_{\S}$-regime} \AAA (see Remark \ref{from_linear_to_nonlinear_n_bigger_than_3})\EEE, as stated in the following.
	
\begin{corollary}\label{fake_nonlinear_estimate_higher_dim}
Let $n\geq 4$. There exist constants $\theta \in (0,1)$ 
$\rm{(}$sufficiently small$\rm{)}$ and $c_{n-1}>0$ such that the following statement holds. For every $u\in W^{1,\infty}(\S;\R^n)$ with $\|\t u-P_T\|_{L^{\infty}(\S)}\leq \theta\ll 1$, there exist a Möbius transformation $\phi$ of $\S$ and $\lambda >0$ such that
\begin{equation}\label{fake_nonlinear_estimate_conformal_case_high_dimensions}
\ds \Big|\frac{1}{\lambda}\t u-\t \phi\Big|^2\leq c_{n-1}\mathcal{E}_{n-1}(u)\,,
\end{equation}
where $\mathcal{E}_{n-1}$ is defined in \eqref{def_of_combined_deficit}. 
\end{corollary}
\begin{remark}
An interesting question would be if the local statement of the above Corollary can be improved to a global one, possibly via a PDE argument. However, in the case of maps $u:\mathbb{S}^{n-1}\mapsto \mathbb{R}^n$, one cannot perform something like an \textit{$n$-harmonic replacement trick}, as for instance in \cite{reshetnyak1970stability} (or $F$-harmonic, harmonic in the setting of \cite{faraco2005geometric},\cite{friesecke2002theorem} respectively), since $\mathbb{S}^{n-1}$ is boundaryless, and there are of course no boundary conditions to relate to the replacement map. It seems that a penalization argument in the spirit of the \textit{selection principle} devised in \cite{cicalese2012selection} (for the optimal quantitative isoperimetric inequality) could be more promising in that direction, which is an interesting question for future investigation. 
\end{remark}
As one can easily notice, and for convenience of the reader we provide the details in Appendix \ref{sec:B},
\begin{equation*}
Q_{n}(w)=Q_{n,\mathrm{conf}}(w)+Q_{n,\mathrm{isop}}(w)\,,
\end{equation*}
where 
\begin{equation*}
Q_{n,\mathrm{conf}}(w):=\frac{n}{n-1}\ds\left|(P_T^t\t w)_{\mathrm{sym}}-\frac{\mathrm{div}_{\S}w}{n-1}I_x \right|^2
\end{equation*}
is the quadratic form associated to the \textit{nonlinear conformal deficit} $\left[\frac{D_{n-1}(u)}{P_{n-1}(u)}\right]^{\frac{n}{n-1}}-1\geq 0$, and 
\begin{equation}\label{Q_isop_sphere_intro}
\hspace{-0.8em}Q_{n,\mathrm{isop}}(w):=\frac{n}{2(n-1)}\left[\ds|\t w|^2+(\mathrm{div}_{\S}w)^2-2\left|(P_T^t\t w)_{\mathrm{sym}}\right|^2\right]-Q_{V_n}(w)
\end{equation} 	
is the one associated to the \textit{nonlinear isoperimetric deficit} $\frac{\big[P_{n-1}(u)\big]^{\frac{n}{n-1}}}{|V_{n}(u)|}-1\geq 0$. Actually, an estimate like \eqref{Korn_type_inequality_for_Q_n} holds true for every positive combination of the two forms $Q_{n,\mathrm{conf}}$ and $Q_{n,\mathrm{isop}}$.

Finally, as we mention in Subsection \ref{Subsection 5.3.}, a similar \textit{linear stability phenomenon} holds true in the isometric case as well, namely one can prove the following. 
\begin{theorem}\label{coercivity_estimate_for_Q_isom_Q_isop}
Let $n\geq 2$. For every $\alpha>0$ there exists a constant $C_{n,\alpha}>0$ such that for every map \linebreak $w\in W^{1,2}(\S;\R^n)$,
\begin{equation}\label{alpha_Q_isom_Q_isop_coercivity_estimate}
\alpha Q_{n,\mathrm{isom}}(w)+Q_{n,\mathrm{isop}}(w)\geq C_{n,\alpha}\ds \big|\t w-[\nabla w_h(0)]_{\mathrm{skew}}P_T\big|^2\,,
\end{equation}
where,
\begin{equation*}
Q_{n,\mathrm{isom}}(w):=\ds\left|(P_T^t\t w)_{\mathrm{sym}}\right|^2
\end{equation*}
is the quadratic form associated to the full $L^2$-isometric deficit $\ds |\sqrt{\t u^t\t u}-I_x|^2$, $Q_{n,\mathrm{isop}}$ \AAA is as in \eqref{Q_isop_sphere_intro}\EEE, and $w_h:\overline{B_1}\mapsto \R^n$ denotes the (componentwise) \textit{harmonic continuation} of $w$ in the interior of $B_1$.
\end{theorem}

The structure of the paper is the following. In Section \ref{Notation} we introduce some notations that we are going to use in the subsequent sections. In Section \ref{sec: isom} we give in steps the proof of Theorem  \ref{main_thm_isometric_case} and remark on the adaptations needed to prove its generalization in higher dimensions, i.e., Theorem \ref{main_thm_isometric_case_n_greater_than_3}. In Section \ref{Section 4} we give again in steps the proof of Theorem \ref{main_thm_conformal_case}. Building upon the analysis that we perform in Subsection \ref{Subsection 4.2}, in Section \ref{sec: 5 linear_stability} we prove the linear stability estimates stated in Theorems \ref{main_coercivity_estimate_conformal_general_dimension} and \ref{coercivity_estimate_for_Q_isom_Q_isop} in all dimensions. In Appendix \ref{New_proof_of_Liouville_Appendix_A} we first exhibit \textit{a short, intrinsic and to our knowledge, new proof of Liouville's Theorem} \ref{spherical Liouville's Theorem}, as well as a related compactness result that can be proven by a slight perturbation of the idea. In Appendix \ref{sec:B} we include just for the convenience of the reader a detailed derivation of \AAA some integral identities for Jacobians, \EEE as well as the Taylor expansions of the geometric quantities that appear in the main body of the paper. \AAA Finally, in Appendix \ref{sec:C} we collect some basic facts from the theory of spherical harmonics that we are using. \EEE

\section{Notation}\label{Notation}

The following standard notation will be adopted throughout the paper.
\begin{tabbing} \hspace*{3.5cm}\=\kill 
$\{e_i\}_{i=1}^n,\  \langle \cdot, \cdot \rangle,\ |\cdot| $ \> the Euclidean orthonormal basis, inner product, norm in $\R^n$\\
$A^t$ \> the transpose of a matrix or the adjoint of the corresponding linear map \\
$Sym(n)$, $Skew(n)$ \> the space of $n\times n$ symmetric, skew-symmetric matrices respectively\\
$A_{\mathrm{sym}}$, $A_{\mathrm{skew}}$ \> the symmetric, skew-symmetric part of a matrix $A\in \R^{n\times n}$ respectively\\
$\{\tau_1,\dots,\tau_{n-1}\}$ \> a positively oriented local orthonormal frame for $T_x\S$, so that for every $x\in \S$\\
\> $\{\tau_1(x),\dots,\tau_{n-1}(x),x\}$ is a positively oriented orthonormal system of $n$ vectors in $\R^n$\\
$\omega_n$ \> the volume of the unit ball $B_1$ in $\R^n$ \EEE\\
$dv_g$ \> the standard $(n-1)$-volume form on $\S$\\
$\mathcal{H}^k$ \> the $k$-dimensional Hausdorff measure\\
$O(n),\ SO(n)$ \> the orthogonal, special orthogonal group of $\R^n$ respectively\\
$Isom_{(+)}(\S)$ \>  the group of rigid motions of $\S$ (the orientation-preserving ones respectively)\\
$Conf_{(+)}(\S)$ \>  the group of Möbius transformations of $\S$ (the orientation-preserving ones respectively)\\
$I_x$\> the identity transformation on $T_x\S$ \\
$\t u$ \> the tangential gradient of $u:\S\mapsto \R^n$, represented in local coordinates by the\\
\> $n\times(n-1)$ matrix with entries $(\t u)_{ij}=\langle \t u^i,\tau_j\rangle$\\
$P_T$ \> $\t\mathrm{id}_{\S}$\\
$d_x u$ \> the intrinsic gradient of a map $u:\S\mapsto\S$, viewed as a linear map\\
\> $d_x u:T_x\S\mapsto T_{u(x)}\S$ with respect to the frame $\{\tau_1,\dots,\tau_{n-1}\}$\\
$\partial_{\vec \nu}f$ \> the radial derivative of a map $f:\overline{B_1}\mapsto\R^m$ on $\S$\\
$\mathrm{div}_{\S}u$, $\Dt u$\> the tangential divergence, Laplace-Beltrami operator of a map $u:\S\to \R^n$\\
$C^k$ \> the space of $k$-times continuously differentiable  maps, $k\in \mathbb{N}$\\
$L^p, W^{1,p}$ \> the standard Lebesgue or Sobolev spaces (on $\S$) respectively, for $1\leq p<\infty$.\\
\> \textit{The norms are taken with respect to the normalized $\mathcal{H}^{n-1}$-measure}, to simplify some \\
\>dimensional constants appearing later in the content\\
$W^{1,\infty}(\S;\R^n)$ \> the space of Lipschitz maps from $\S$ to $\R^n$; $\|u\|_{W^{1,\infty}}:=\mathrm{max}\left\{\|u\|_{L^{\infty}},\|\t u\|_{L^{\infty}}\right\}$\\
$\sim_{M_1,M_2,\dots}, \lesssim_{M_1,M_2,\dots}$ \> the corresponding equality, inequality is valid up to a constant that is allowed\\ \> to vary from line to line but depends only on the parameters $M_1, M_2,\dots$, or only on the \\
\> dimension when the subscripts are absent.\\
$c,C>0$ \> universal constants whose value is allowed to vary from line to line and place to place\\
\> but depend in any case only on the dimension.
\end{tabbing}

\section{The isometric case: Proof of Theorem \ref{main_estimate_isometric_case}}\label{sec: isom}
In what follows, the \textit{$L^2$-isometric deficit} of a map $u\in W^{1,2}(\s;\R^3)$ that we are using is denoted by
\begin{equation}\label{def_isometric_deficit}
\delta(u):=\big\|\big(\sigma_{2}-1\big)_+\big\|_{L^2(\s)}\,,
\end{equation}
where $0\leq \sigma_1\leq \sigma_{2}$ are the principal stretches of $u$, i.e., the eigenvalues of $\sqrt{\t u^t \t u}$.\\[3pt]
Note that $\delta(u)=0$ whenever $u$ is a \textit{short map}, i.e., $u\in W^{1,\infty}(\s;\R^3)$ with
$\t u^t\t u\leq I_x\,,$ $\mathcal{H}^2$-a.e. on $\s$ in the sense of quadratic forms. In general,
\begin{equation}\label{comparison_of_two_deficits}
\delta(u)\leq \big\|\sigma_{2}-1\big\|_{L^2(\s)}\leq \Big\|\sqrt{\t u^t\t u}-I_x \Big\|_{L^2(\s)}\leq \sqrt{2} \big\|\sigma_{2}-1\big\|_{L^2(\s)}\,,
\end{equation}
so that (having in mind the {\sc Nash-Kuiper} Theorem, cf.~\cite{kuiper1955c1},\cite{nash1954c1}) the deficit $\delta(u)$ is sharper than the \textit{full $L^2$-isometric deficit} 
\begin{equation}\label{full_isometric_deficit}
\delta_{\mathrm{isom}}(u):=\Big\|\sqrt{\t u^t\t u}-I_x \Big\|_{L^2(\s)}\,,
\end{equation}
 since it only \textit{penalizes local stretches} under $u$. The isoperimetric deficit (or the positive part of the excess in volume) in this setting is denoted by
\begin{equation}\label{def_isoperimetric_deficit}
\varepsilon(u):=\Big(1-\big|V_3(u)\big|\Big)_+\,.
\end{equation}
Before presenting the proof of the result, let us make some preliminary remarks.

\begin{remark}
$(i)$ If $u\in W^{1,2}(\s;\R^3)$ is a \textit{globally short map}, then $\delta(u)=0$. Moreover, since in this case $|\partial_{\tau_1}u\wedge\partial_{\tau_2}u|\leq 1$ and $|\t u|\leq \sqrt2$\ $\mathcal{H}^2$-a.e. on $\s$, by the Cauchy-Schwarz inequality and the \textit{sharp Poincare inequality on $\s$} (equality in which is achieved for restrictions on $\s$ of affine maps of $\R^3$\AAA, see \eqref{Poincare} in Appendix \hyperref[sec:C]{C})\EEE,
\begin{equation}\label{isoperimetric_inequality_direct_for_isometries}
\hspace{-0.5em}|V_3(u)|=\left|\d2 \Big\langle u-\d2 u, \partial_{\tau_1}u\wedge\partial_{\tau_2}u\Big\rangle \right|\leq \d2 \Big|u-\d2 u\Big|
\leq \left(\d2 \Big|u-\d2 u\Big|^2\right)^{\frac{1}{2}}\leq \left(\frac{1}{2}\d2 |\t u|^2\right)^{\frac{1}{2}}\leq 1,
\end{equation}
that is, $\varepsilon(u)=1-|V_3(u)|$. This is \EEE something that could also be seen just by using the isoperimetric inequality in this case. Hence, for \textit{globally short maps} only the \textit{excess in volume} is present in the right hand side of the stability estimate \eqref{main_estimate_isometric_case}.

$(ii)$ On the other hand, if $u\in W^{1,2}(\s;\R^3)$ is \textit{volume-increasing} in the sense that $|V_3(u)|\geq 1$, then $\varepsilon(u)=0$, and only the isometric deficit $\delta(u)$ is present in the right hand side of \eqref{main_estimate_isometric_case}.

$(iii)$ In all other cases, i.e., \EEE if $u\in W^{1,2}(\s;\R^3)$ is not globally short and not volume-increasing, both deficits are present in the estimate. It is also immediate that \textit{one cannot have simultaneously a globally short map $u$ that is volume-increasing, unless $u$ is a rigid motion of $\s$}, something that can be directly verified by checking the equality cases in \eqref{isoperimetric_inequality_direct_for_isometries}.
\end{remark}

As we also mentioned in the Introduction, \textit{\eqref{main_estimate_isometric_case} is optimal in the norm appearing on the left hand side and the deficits on the right hand side, i.e., the exponent $1$ with which $\delta(u)$ and $\varepsilon(u)$ appear in the estimate cannot generically be improved}. Examples showing the optimality of the exponents can easily be constructed even in dimension $n=2$, where the exact analogue of Theorem  \ref{main_thm_isometric_case} becomes

\begin{proposition}\label{main_estimate_isometric_case_n=2}
There exists a constant $c_0>0$ so that for every $u\in W^{1,2}(\mathbb{S}^1;\R^2)$ there exists $O\in O(2)$ such that
\begin{equation}\label{main_estimate_isometric_case_n_2}
\dashint_{\mathbb{S}^1} \big|\partial_{\tau} (u-O\mathrm{id}_{\mathbb{S}^1})\big|^2\leq c_0\Big(\big\| \big(|\partial_{\tau}u|-1\big)_+\big\|_{L^2(\mathbb{S}^1)}+\Big(1-\Big|\dashint_{\mathbb{S}^1}\big\langle u,(\partial_\tau u)^{\bot}\big\rangle\Big|\Big)_+\Big)\,.
\end{equation}	
\end{proposition}	
Here, $\partial_\tau u$ denotes the tangential derivative of $u$ along $\mathbb{S}^1$. The previous proposition can be proven in exactly the same way as Theorem \ref{main_thm_isometric_case}, following the arguments of the next subsections. As the reader might observe later, the Lipschitz truncation argument of Subsection \ref{Reduction_to_Lipschitz_mappings} is even simpler in the case $n=2$, because the signed volume $V_2(u):=\dashint_{\mathbb{S}^1}\big\langle u,(\partial_\tau u)^{\bot}\big\rangle $ is of first order in $\partial_\tau u$. 

Keeping the notation $\delta(u)$ and $\varepsilon(u)$ for the isometric and the isoperimetric deficit also when $n=2$, two instructive examples for the optimality of the exponents are given in the next remark.

\begin{remark}\label{examples}
$(i)$ For $0<\sigma \ll 2\pi$, let $u_{\sigma}:\mathbb{S}^1\mapsto \R^2$ be defined in polar coordinates via
\begin{equation}\label{example_for_optimality_isometric_case}
u_{\sigma}(\theta):=\left\{
\begin{array}{lr}
(\cos\theta,\sin\theta);\quad 0\leq \theta<\frac{3\pi}{2}-\frac{\sigma}{2}\,,\\
\left(\cos\theta,2\sin\left(\frac{3\pi}{2}-\frac{\sigma}{2}\right)-\sin\theta\right);\quad \frac{3\pi}{2}-\frac{\sigma}{2}\leq \theta<\frac{3\pi}{2}+\frac{\sigma}{2}\ \\
(\cos\theta,\sin\theta);\quad \frac{3\pi}{2}+\frac{\sigma}{2}\leq \theta<2\pi
\end{array}\right\}\,.
\end{equation}
For each $\sigma\in [0,2\pi)$ the map $u_{\sigma}$ is isometric, being essentially the identity transformation, except for a small circular arc of angle $\sigma$, where it is a \textit{flip} with respect to the horizontal line at height $y_0=\sin\left(\frac{3\pi}{2}-\frac{\sigma}{2}\right)$. Hence, $\delta(u_\sigma)=0$ for every $\sigma\in [0,2\pi)$. Obviously, $\partial_\tau u_{\sigma}\rightarrow \partial_\tau \mathrm{id}_{\mathbb{S}^1}$  strongly in $L^2(\mathbb{S}^1;\R^2)$ as $\sigma\to 0^+$, and one can easily obtain that
\begin{align*}
\begin{split}
\dashint_{\mathbb{S}^1}\big|\partial_\tau u_{\sigma}-\partial_\tau \mathrm{id}_{\mathbb{S}^1}\big|^2&\sim \int_{0}^{2\pi}\big|\partial_\theta u_{\sigma}-\partial_\theta\mathrm{id}_{\mathbb{S}^1}\big|^2=\int_{\frac{3\pi}{2}-\frac{\sigma}{2}}^{\frac{3\pi}{2}+\frac{\sigma}{2}}\big|(-\sin\theta,-\cos\theta)-(-\sin\theta,\cos\theta)\big|^2\\
&=4\int_{\frac{3\pi}{2}-\frac{\sigma}{2}}^{\frac{3\pi}{2}+\frac{\sigma}{2}}\cos^2(\theta)=2(\sigma-\sin{\sigma})=\mathcal{O}(\sigma^3), \ \ \mathrm{for\ } 0<\sigma\ll 2\pi\,.
\end{split}
\end{align*}
On the other hand, using elementary plane-geometry formulas for the area of circular triangles, we can compute the area of the double arc-region of the unit disc missed by $u_{\sigma}$, so that also
\begin{equation*}
\varepsilon(u_{\sigma})=\frac{2}{\pi}\left(\pi\cdot\frac{\sigma}{2\pi}-\frac{1}{2}\sin{\sigma}\right)=\frac{1}{\pi}(\sigma-\sin{\sigma})=\mathcal{O}({\sigma}^3), \quad \mathrm{for\ } 0<\sigma\ll 2\pi\,,
\end{equation*}
which reveals the optimality of the exponent of $\varepsilon(u)$ in the estimate \AAA \eqref{main_estimate_isometric_case_n_2}\EEE. \\
$(ii)$ Identify now $\mathbb{S}^1$ with the interval $[0,1]$ by identifying the endpoints. For $0<\sigma\ll 1$, consider the maps $f_{\sigma}:[0,1]\mapsto[0,1]$, defined as follows.
\begin{equation}\label{example_for_optimality_isoperimetric_case}
f_{\sigma}(t):=\left\{
\begin{array}{lr}
t;\quad 0\leq t<\sigma\,,\\
2{\sigma}-t; \quad \sigma \leq t<2\sigma\,, \\
-\frac{2\sigma}{1-2\sigma}+\frac{1}{1-2\sigma}t;\quad 2\sigma\leq t<1
\end{array}\right\}\,,
\end{equation}
and let $u_{\sigma}:\mathbb{S}^1\mapsto\mathbb{S}^1$ be the corresponding maps defined on the unit circle. Obviously, $\varepsilon(u_{\sigma})=0$ for every $\sigma\in [0,2\pi)$. Geometrically, the maps $u_{\sigma}$ \textit{travel back and forth, and produce a triple cover of a small ${\sigma}$-arc, locally stretching $\mathbb{S}^1$}. With similar calculations as before,
\begin{align*}
\begin{split}
\dashint_{\mathbb{S}^1}\big|\partial_\tau u_{\sigma}-\partial_\tau \mathrm{id}_{\mathbb{S}^1}\big|^2&\sim \int_{0}^{1}\big|f'_{\sigma}(t)-1\big|^2=\int^{2\sigma}_{\sigma}(-2)^2+\int_{2{\sigma}}^{1}\left(\frac{1}{1-2\sigma}-1\right)^2\\
&\sim 4\sigma+\frac{4{\sigma}^2}{1-2\sigma}\sim\sigma+\mathcal{O}({\sigma}^2)=\mathcal{O}({\sigma}), \quad \mathrm{for\ } 0<\sigma\ll 1\,.
\end{split}
\end{align*}
Moreover,
\begin{align*}
\begin{split}
\delta^2(u_{\sigma})&\sim \int_{0}^{1}\Big|\big(|f'_{\sigma}(t)|-1\big)_+\Big|^2=\int_{2\sigma}^{1}\left(\frac{1}{1-2\sigma}-1\right)^2\sim\frac{4{\sigma}^2}{1-2{\sigma}}\\
&\sim {\sigma}^2\big(1+\mathcal{O}(\sigma)\big)=\mathcal{O}({\sigma}^2), \quad \mathrm{for\ } 0<\sigma\ll 1\,,
\end{split}
\end{align*}
which reveals the optimality of the exponent of $\delta(u)$ in the estimate  \AAA \eqref{main_estimate_isometric_case_n_2} \EEE in the generic setting. The geometric reason behind this, is the fact that the deficit $\delta(u)$ (as well as the full $L^2$-isometric deficit $\delta_{\mathrm{isom}}(u)$) \textit{does not penalize changes in the orientation neither extrinsically}, i.e., flips in ambient space, \textit{nor intrinsically}, when $u$ is seen as a map from the sphere onto its image.
\end{remark}
\vspace{-1em}
When $n=3$ (and also in higher dimensions) one can construct similar examples \AAA as in \eqref{example_for_optimality_isometric_case}, \eqref{example_for_optimality_isoperimetric_case}\EEE. For instance, in the first case one can consider maps that are the identity outside a small geodesic ball of $\S$ and inside being again flips in $\R^n$ with respect to the appropriate affine hyperplane. In the second case, one can rotate the previous one-dimensional example around a fixed axis.

We are now ready to present the proof of Theorem \ref{main_thm_isometric_case} in steps. For the most part, by straightforward modifications that mainly regard the change of some dimensional constants in the estimates and of some purely algebraic expressions, the arguments are valid in all dimensions, and can be used to prove Theorem  \ref{main_thm_isometric_case_n_greater_than_3} as well. We will come back to that issue in Subsection \ref{Subsection 3.4}.

\subsection{Reduction to Lipschitz mappings}\label{Reduction_to_Lipschitz_mappings}
As in the pioneering geometric rigidity result of {\sc G. Friesecke, R.D. James} and {\sc S. Müller} (cf.~\cite[Theorem 3.1]{friesecke2002theorem}), the first step is to justify why it suffices to work with maps with a universal upper bound on their Lipschitz constant. This is achieved through the use of the following standard \textit{truncation lemma}.

\begin{lemma}\label{Lipschitz_truncation_lemma}	
There exists $c>0$ so that for every $u\in W^{1,2}(\s;\R^3)$ and every $M>0$, there exists $u_M\in W^{1,\infty}(\s;\R^3)$ such that
\begin{itemize}	
\item[$\mathrm{(}i\mathrm{)}$] $\|\t u_M\|_{L^{\infty}}\leq cM$\,,
\item[$\mathrm{(}ii\mathrm{)}$] $\mathcal{H}^{2}\big(\{x\in \s\colon u(x)\neq u_M(x)\}\big)\leq \frac{c}{M^2}\int_{\{|\t u|>M\}}|\t u|^2$\,,
\item[$\mathrm{(}iii\mathrm{)}$] $\d2|\t u-\t u_M|^2\leq c\int_{\{|\t u|>M\}}|\t u|^2$\,.
\end{itemize}
\end{lemma}
The proof of this lemma can be performed as for the corresponding statement in the bulk (cf.~\cite[Proposition A.1]{friesecke2002theorem}), since it relies basically on a partition of unity argument. With the use of it we can now prove the following.
\begin{lemma}\label{good_behaviour_of_deficits_in_terms_of_Lipschitz_truncation}
For $u\in W^{1,2}(\s;\R^3)$ let $u_M$ be its Lipschitz truncation of Lemma \ref{Lipschitz_truncation_lemma} for $M:=2\sqrt 2$. Then,
\begin{equation}\label{deficits_behave_well_estimate}
\delta(u_M)\lesssim \delta(u)\quad \mathrm{and} \quad \varepsilon(u_M)\lesssim \delta(u)+\varepsilon(u)\,. 
\end{equation} 
\end{lemma}
\begin{proof}
If $\delta(u)>1$, \AAA recalling the definitions of the deficits in  \eqref{comparison_of_two_deficits}-\eqref{def_isoperimetric_deficit}, \EEE we trivially have
\begin{align*}
\delta^2(u_M)\leq \delta^2_{\mathrm{isom}}(u_M) \leq2\d2\Big(\Big|\sqrt{\t u_M^t\t u_M}\Big|^2+|I_x|^2\Big)\leq 2\big(c^2M^2+2\big)\lesssim \delta^2(u)\,,
\end{align*}
and
\begin{align*}
\varepsilon(u_M)\leq 1\leq \delta(u)\leq \delta (u)+\varepsilon(u)\,,
\end{align*}
so we may assume without loss of generality that $0\leq \delta(u)\leq 1$. With the notation we have employed \AAA in \eqref{def_isometric_deficit}\EEE, $|\t u|^2=\sigma_1^2+\sigma_2^2\leq 2\sigma_2^2$, and therefore in this case we also have the upper bound
\begin{align}\label{these_u_s_bounded_in_w_1_2}
\begin{split}
\d2 |\t u|^2&=\frac{1}{3\omega_3}\int_{\{0\leq\sigma_2\leq 1\}} \big|\t u\big|^2+\frac{1}{3\omega_3}\int_{\{\sigma_2>1\}}\big|\t u\big|^2\leq 2+\frac{2}{3\omega_3}\int_{\{\sigma_2>1\}}\sigma_2^2\\[3pt]
&\leq 2+\frac{4}{3\omega_3}\int_{\{\sigma_2>1\}}\big((\sigma_2-1)^2+1\big)\leq 6+4\delta^2(u)\leq 10\,.
\end{split}
\end{align} 
By a standard argument, \AAA using Lemma \ref{Lipschitz_truncation_lemma}, \EEE we also obtain 
\begin{align}\label{size_of_the_set_where they_differ}
\begin{split}
M^2\mathcal{H}^{2}\big(\{x\in \s: u(x)\neq u_M(x)\}\big)\leq c\int_{\{|\t u|>M\}}|\t u|^2\lesssim \delta^2(u)\,,\\
\d2|\t u-\t u_M|^2\leq c\int_{\{|\t u|>M\}}|\t u|^2\lesssim\delta^2(u)\,.
\end{split}
\end{align}
Indeed, in the set $\{x\in \mathbb{S}^2\colon |\t u(x)|>M:=2\sqrt{2}\}$ we have $\sigma_2\geq \frac{1}{\sqrt 2}|\t u|> 2$, so in this set we can estimate pointwise,
\begin{align*}
(\sigma_2-1)_+=\sigma_2-1\geq \frac{|\t u|}{\sqrt 2}-1> \frac{|\t u|}{\sqrt 2}-\frac{|\t u|}{2\sqrt 2}>\frac{|\t u|}{2\sqrt{2}}\,,
\end{align*}
and then the final estimates in \eqref{size_of_the_set_where they_differ} follow immediately.

For the first estimate in \eqref{deficits_behave_well_estimate} the argument is now elementary. Labelling $0\leq \sigma_{M,1}\leq\sigma_{M,2}$ the eigenvalues of $\sqrt{\t u_M^t\t u_M}$, using \eqref{size_of_the_set_where they_differ} and the fact that $\{u_M=u\}\subseteq\{\t u_M=\t u\}$ in the $\mathcal{H}^2$-a.e. sense, we can estimate
\begin{align}\label{isometric_deficits_well_controlled}
\hspace{-1em}\begin{split}
\delta^2(u_M)&\sim\int_{\{\sigma_{M,2}>1\}}(\sigma_{M,2}-1)^2=\int_{\{u=u_M\}\cap\{\sigma_{M,2}>1\}}(\sigma_{M,2}-1)^2+\int_{\{u\neq u_M\}\cap\{\sigma_{M,2}>1\}}(\sigma_{M,2}-1)^2\\
&\leq\int_{\{u=u_M\}\cap\{\sigma_{M,2}>1\}}(\sigma_2-1)^2+\int_{\{u\neq u_M\}\cap\{\sigma_{M,2}>1\}}(\sigma_{M,2}^2+1)\\
&\leq \int_{\{\sigma_2>1\}}(\sigma_2-1)^2+(c^2M^2+1)\mathcal{H}^{2}\big(\{u_M\neq u\}\big)\lesssim \delta^2(u)\,.
\end{split}
\end{align}

For the second desired estimate in \eqref{deficits_behave_well_estimate} we observe that if $\big|V_3(u_M)\big|>1$ then $\varepsilon(u_M)=0\leq \varepsilon(u)$, so we may assume without loss of generality that $\big|V_3(u_M)\big|\leq 1$. Then,  
\begin{align}\label{1st_easy_estimate_for_difference_of_isoperimetric_deficits}
\varepsilon(u_M)&=1-\big|V_3(u_M)\big|\leq \varepsilon(u)+\big|V_3(u)\big|-\big|V_3(u_M)\big|\leq\varepsilon(u)+\big|V_3(u)-V_3(u_M)\big|\,, 
\end{align}
i.e., it suffices to control the absolute value of the difference between the corresponding signed volumes. Towards this end, denoting by 
\begin{equation}\label{notation_for_mean}
\overline{v}:=\d2 v, \quad \text{for } v\in W^{1,2}(\mathbb{S}^2;\R^3)\,,
\end{equation}
one can easily verify that
\begin{equation}\label{identity_difference_in_volumes}
V_3(u)-V_3(u_M)=V_3(u-u_M)+R_1(u,u_M)+R_2(u,u_M)+R_3(u,u_M)+R_4(u,u_M)\,,
\end{equation}
where 
\begin{align}\label{remainders_in_difference_in_volumes}
\begin{split}
R_1(u,u_M)&:=\d2\left\langle (u-u_M)-\overline{(u-u_M)},\partial_{\tau_1}u_M\wedge\partial_{\tau_2}u\right\rangle\,,\\
R_2(u,u_M)&:=\d2 \left\langle (u-u_M)-\overline{(u-u_M)},\partial_{\tau_1}(u-u_M)\wedge\partial_{\tau_2}u_M\right\rangle\,,\\
R_3(u,u_M)&:=\d2\Big\langle u_M-\overline{u_M},\partial_{\tau_1}u_M\wedge\partial_{\tau_2}(u-u_M)\Big\rangle\,,\\
R_4(u,u_M)&:=\d2\Big\langle u_M-\overline{u_M},\partial_{\tau_1}(u-u_M)\wedge\partial_{\tau_2}u\Big\rangle\,.
\end{split}
\end{align}
We can now estimate each term on the right hand side of \eqref{identity_difference_in_volumes} separately. For the first one, by the isoperimetric inequality \AAA (see \eqref{AM-GM-isoperimetric-inequality} for $n=3$) \EEE and the second estimate in \eqref{size_of_the_set_where they_differ}, we obtain
\begin{align}\label{bound_of_first_summand}
\big|V_3(u-u_M)\big|\leq\left(\frac{1}{2}\d2|\t u-\t u_M|^2 \right)^{\frac{3}{2}} \lesssim \delta^3(u)\,.
\end{align}
To estimate the terms $(R_i(u,u_M))_{i=1,\dots,4}$ we can now use the properties of the Lipschitz truncation $u_M$ provided by Lemma \ref{Lipschitz_truncation_lemma}, the Cauchy-Schwarz inequality and the sharp Poincare inequality on $\s$ \AAA(see \eqref{Poincare} in Appendix \ref{sec:C})\EEE, as well as the estimates \eqref{these_u_s_bounded_in_w_1_2} and \eqref{size_of_the_set_where they_differ}, in order to estimate each of the remaining terms  in \eqref{identity_difference_in_volumes} as follows.
\begin{align}\label{bound_of_R_1_2_3_4}
\begin{split}
\big|R_1(u,u_M)\big|&\lesssim M\big\|(u-u_M)-\overline{(u-u_M)}\big\|_{L^2}\left\|\t u\right\|_{L^2}\lesssim \left\|\t u-\t u_M\right\|_{L^2}\lesssim \delta(u)\,,\\
\big|R_2(u,u_M)\big|&\lesssim M\big\|(u-u_M)-\overline{(u-u_M)}\big\|_{L^2}\left\|\t u-\t u_M\right\|_{L^2}\lesssim \left\|\t u-\t u_M\right\|^2_{L^2}\lesssim \delta^2(u)\,,\\
\big|R_3(u,u_M)\big|&\lesssim M^2\int_{\{u\neq u_M\}}|\t u-\t u_M|\lesssim \sqrt{\mathcal{H}^2\big(\{u\neq u_M\}\big)}\left\|\t u-\t u_M\right\|_{L^2}\lesssim \delta^2(u)\,,\\
\big|R_4(u,u_M)\big|&\lesssim \|u_M-\overline{u_M}\|_{L^\infty}\left\|\t u\right\|_{L^2}\left\|\t u-\t u_M\right\|_{L^2}\lesssim M\delta(u)\lesssim \delta(u)\,.
\end{split}
\end{align}
By \eqref{remainders_in_difference_in_volumes}-\eqref{bound_of_R_1_2_3_4}, and since we have assumed without loss of generality that $0\leq \delta(u)\leq 1$, the expansion \eqref{identity_difference_in_volumes} implies that 
\begin{equation}\label{difference_in_signed_volumes_behaves_well_in_dimension_3}
\big|V_3(u)-V_3(u_M)\big|\AAA\lesssim \|\t u-\t u_M\|_{L^2}\EEE\lesssim \delta(u)\,,
\end{equation} 
and then \eqref{1st_easy_estimate_for_difference_of_isoperimetric_deficits} yields the desired estimate \AAA \eqref{deficits_behave_well_estimate} \EEE for the isoperimetric deficit.
\end{proof}
In view of Lemma \ref{good_behaviour_of_deficits_in_terms_of_Lipschitz_truncation}\EEE, fixing from now on $M:=2\sqrt{2}$, we easily see that if the estimate \eqref{main_estimate_isometric_case} holds true for the Lipschitz map $u_M$ for some $O\in O(3)$, then it also holds true for $u$ with the same $O$, up to changing the constant in its right hand side. It therefore suffices to prove Theorem  \ref{main_thm_isometric_case} for maps $u\in W^{1,\infty}(\s;\R^3)$ whose Lipschitz constant is apriori bounded from above by $cM$, where $c>0$ is the constant of Lemma \ref{Lipschitz_truncation_lemma}.

\subsection{Further reduction to maps $W^{1,2}$-close to the $\mathrm{id}_{\s}$}\label{Further reduction to maps W1,2-close to the identity}
Having reduced our attention to maps that enjoy an apriori Lipschitz bound, we show in this subsection that for our purposes, we can further assume without loss of generality that the maps in consideration are sufficiently close to the $\mathrm{id}_{\s}$ in the $W^{1,2}$-topology. To do so, we first prove a qualitative analogue of Theorem \ref{main_thm_isometric_case}. \AAA Recalling the notations introduced in \eqref{def_isometric_deficit}, \eqref{def_isoperimetric_deficit} and \eqref{notation_for_mean}, we have. \EEE

\begin{lemma}\label{isometries_compactness_general}
Let $(u_k)_{k\in \mathbb{N}}\subset W^{1,\infty}(\s;\R^3)$ be such that $\underset{k\in \mathbb{N}}{\mathrm{sup}}\ \|\t u_k\|_{L^{\infty}}\leq cM$, and suppose that 
\begin{equation}\label{both deficits go to 0}
\lim_{k\to \infty}\big(\delta(u_k)+\varepsilon(u_k)\big)=0\,.
\end{equation}
Then, there exists $O\in O(3)$ so that up to a non-relabeled subsequence, 
\begin{equation}
u_k-\overline{u_k}\rightarrow O\mathrm{id}_{\s} \mathrm{\ \ strongly \  in \ } W^{1,2}(\s;\R^3)\,.
\end{equation}
\end{lemma}
\begin{proof}
We can obviously assume without loss of generality that $\overline{u_k}:=\d2 u_k=0$ for all $k\in \mathbb{N}$. Hence, the sequence $(u_k)_{k\in \mathbb{N}}$ is uniformly bounded in $W^{1,2}(\s;\R^3)$, and up to passing to a non-relabeled subsequence, converges weakly in $W^{1,2}(\s;\R^3)$ and also pointwise $\mathcal{H}^{2}$-a.e. to a map $u\in W^{1,2}(\s;\R^3)$ with $\overline{u}=0$. By lower semicontinuity of the Dirichlet energy under weak $W^{1,2}$-convergence, we further have that 
\begin{equation}\label{lower_semicontinuity_inequality_for_L_2_gradient_of_u}
\frac{1}{2}\d2 |\t u|^2\leq\liminf_{k\to\infty}\frac{1}{2}\d2 |\t u_k|^2\leq 1\,.
\end{equation}
The last inequality \AAA in \eqref{lower_semicontinuity_inequality_for_L_2_gradient_of_u} \EEE is justified by the following estimates.
\begin{align}\label{norm_of_grad_u_k_estimate_from_isometric_deficit}
\begin{split}
\frac{1}{2}\d2|\t u_k|^2&=\frac{\int_{\{0\leq \sigma_{k,2}\leq 1\}}|\t u_k|^2+\int_{\{\sigma_{k,2}> 1\}}|\t u_k|^2}{6\omega_3}\leq \frac{\mathcal{H}^2\big(\{0\leq \sigma_{k,2}\leq 1\}\big)+\int_{\{\sigma_{k,2}>1\}}\sigma_{k,2}^2}{3\omega_3}\\
&=1+\frac{1}{3\omega_3}\int_{\{\sigma_{k,2}>1\}}(\sigma_{k,2}^2-1)\leq 1+\frac{\big(\|\sigma_{k,2}\|_{L^\infty}+1\big)}{3\omega_3}\int_{\{\sigma_{k,2}>1\}}(\sigma_{k,2}-1)_+\\
&\leq 1+(cM+1)\left(\d2 (\sigma_{k,2}-1)_+^2\right)^{\frac{1}{2}}\leq 1+c\delta(u_k)\,.
\end{split}
\end{align}
	
In a similar manner, we can use again the assumption that $\underset{k\in \mathbb{N}}{\mathrm{sup}}\ \|\t u_k\|_{L^{\infty}}\leq cM$, and the fact that the determinant is a Lipschitz function, to estimate also
\begin{equation}\label{determinant_estimate_by_isometric_deficit}
\d2 \big|\partial_{\tau_1}u_k\wedge\partial_{\tau_2}u_k\big|^2\leq 1+ c\delta(u_k)\,.
\end{equation}

Since $0\leq 1-\varepsilon(u_k)\leq\big|V_3(u_k)\big|$ and $\overline{u_k}=0$,
by the sharp Poincare inequality on $\s$ \AAA(see again \eqref{Poincare} in Appendix \ref{sec:C}) \EEE and the estimates \eqref{norm_of_grad_u_k_estimate_from_isometric_deficit}, \eqref{determinant_estimate_by_isometric_deficit}, we obtain
\begin{align}\label{basic_trick_isometric_estimate_qualitative}
\begin{split}
0\leq 1-\varepsilon(u_k)&\leq\left|\d2\Big\langle u_k,\partial_{\tau_1}u_k\wedge\partial_{\tau_2}u_k\Big\rangle\right|\leq \left(\d2 |u_k|^2\right)^{\frac{1}{2}}\cdot\left(\d2 \big|\partial_{\tau_1}u_k\wedge\partial_{\tau_2}u_k\big|^2\right)^{\frac{1}{2}}\\
&\leq \left(\d2 |u_k|^2\right)^{\frac{1}{2}}\cdot\Big(1+c\delta(u_k)\Big)^{\frac{1}{2}}\leq \left(\frac{1}{2}\d2 |\t u_k|^2\right)^{\frac{1}{2}}\cdot\Big(1+c\delta(u_k)\Big)^{\frac{1}{2}}\\
&\leq 1+c\delta(u_k)\,.
\end{split}
\end{align}
By the assumption \eqref{both deficits go to 0}, and since $u_k\to u$ strongly in $L^2(\s;\R^3)$,  we can let $k\to\infty$ in \eqref{basic_trick_isometric_estimate_qualitative}, to obtain 
\begin{equation}\label{norm_L2_of_limit_is_1}
\d2 |u|^2=\lim_{k\to\infty}\d2 |u_k|^2=1\,.
\end{equation} 
Hence, the limiting map $u$ is such that $\d2 u=0$, and \AAA by \eqref{lower_semicontinuity_inequality_for_L_2_gradient_of_u} and \eqref{norm_L2_of_limit_is_1} it also satisfies \EEE
\begin{equation}\label{inequality_chain_for_characterization_of_u}
1\geq \frac{1}{2}\d2|\t u|^2 \geq \d2 |u|^2=1\,.
\end{equation}
By the equality case in the sharp Poincare inequality on $\s$ (since the first nontrivial eigenfunctions of $-\Delta_{\s}$ are the coordinate functions\AAA, cf.~Appendix \ref{sec:C}\EEE), we deduce \AAA from \eqref{inequality_chain_for_characterization_of_u} \EEE that $u(x)=Ax$ for some $A\in \R^{3\times 3}$ with $|A|^2=3$. In particular, equalities are achieved in \eqref{lower_semicontinuity_inequality_for_L_2_gradient_of_u}, and therefore $u_k\to u:=A\mathrm{id}_{\s}$ in the strong $W^{1,2}$-topology.
	
To show that $A\in O(3)$, we argue as follows. Having established the strong $W^{1,2}$-convergence of $(u_k)_{k\in\mathbb{N}}$ towards $u$, up to a further non-relabeled subsequence we can assume now that $\t u_k\to \t u$ also pointwise $\mathcal{H}^{2}$-a.e. on $\s$ and therefore, using the assumption that $\|\t u_k\|_{L^\infty}\leq cM$ for every $k\in \mathbb{N}$, 
\begin{align}\label{pointwise_convergence}
\begin{split}
\big\langle u_k, \partial_{\tau_1}u_k\wedge\partial_{\tau_2}u_k \big\rangle&\to \big\langle u, \partial_{\tau_1}u\wedge\partial_{\tau_2}u \big\rangle \ \mathrm{\ pointwise \ } \mathcal{H}^{2}\text{-a.e.}\,,\\
\big|\big\langle u_k, \partial_{\tau_1}u_k\wedge\partial_{\tau_2}u_k \big\rangle\big|&\leq \big|\partial_{\tau_1}u_k\wedge\partial_{\tau_2}u_k\big||u_k|\leq \frac{1}{2}|\t u_k|^2|u_k|\lesssim|u_k|\quad \forall k\in \mathbb{N}\,,\\
|u_k|\to  |u|& \quad \mathrm{\ pointwise \ } \mathcal{H}^{2}\text{-a.e.}\,,\\
\underset{k\in \mathbb{N}}{\mathrm{sup}}\ \|u_k\|_{L^1}&\leq\underset{k\in \mathbb{N}}{\mathrm{sup}}\ \|u_k\|_{L^2}\ll +\infty,\quad \mathrm{\ since\ } \|u_k\|_{L^2}\to 1\,.
\end{split}
\end{align}
Using a variant of Lebesgue's Dominated Convergence Theorem in the assumption that $\lim_{k\to\infty}\varepsilon(u_k)=0$ \AAA and \eqref{pointwise_convergence}\EEE, allows us to conclude. Indeed,
\begin{align*}
0&=\lim_{k\to \infty} \varepsilon(u_k)=\lim_{k\to \infty} \left(1-\left|\d2 \big\langle u_k,\partial_{\tau_1}u_k\wedge\partial_{\tau_2}u_k\big\rangle\right|\right)_+=\left(1- \left|\d2 \lim_{k\to \infty}\big\langle u_k,\partial_{\tau_1}u_k\wedge\partial_{\tau_2}u_k\big\rangle\right|\right)_+\\
&=\left(1-\left|\d2 \big\langle u,\partial_{\tau_1}u\wedge\partial_{\tau_2}u
\big\rangle\right|\right)_+=\left(1-\left|\d2 \big\langle Ax,\partial_{\tau_1}(Ax)\wedge\partial_{\tau_2}(Ax)
\big\rangle\right|\right)_+=\big(1-\big|\mathrm{det}A\big|\big)_+\,,
\end{align*}	
i.e., $|\mathrm{det}A|\geq 1$. If we now perform the polar decomposition $A=O\sqrt{A^tA}$, where $O\in O(3)$, and label $0\leq\alpha_1\leq\alpha_2\leq \alpha_3$ the eigenvalues of $\sqrt{A^tA}$, by the arithmetic mean-geometric mean inequality we get 
\begin{equation*}
1=\left(\frac{|A|^2}{3}\right)^\frac{3}{2}\geq |\mathrm{det}A|\geq 1\,,
\end{equation*}
and equality in this algebraic inequality implies that $\alpha_1=\alpha_2=\alpha_3=1$, i.e., $O:=A\in O(3)\EEE$.
\end{proof}	
As an immediate consequence of the \AAA Lemmata \ref{good_behaviour_of_deficits_in_terms_of_Lipschitz_truncation} and \ref{isometries_compactness_general}\EEE, we obtain the following.
\begin{corollary}\label{it_suffices_to_prove_the_local_isometric_case}
It suffices to prove Theorem \ref{main_thm_isometric_case} for maps 
\begin{equation}\label{local_family_of_maps_sufficient_for_isometric_case}
u\in \mathcal{A}_{M,\theta}:=\left\{u \in W^{1,\infty}(\s;\R^3)\colon
\begin{array}{lr}
(i)\ \ \ \d2 u=0\,\\
(ii)\ \ \big\|\t u\big\|_{L^\infty}\leq cM\,\\
(iii)\ \big\| \t u-P_T\big\|_{L^2}\leq \theta
\end{array}\right\}\,,
\end{equation}
where $c>0$ is the constant of Lemma \ref{Lipschitz_truncation_lemma}, $M:=2\sqrt{2}$ and $0<\theta\ll 1$ is a sufficiently small constant that will be suitably chosen later.
\end{corollary}
\begin{proof}
The proof is a standard \textit{contradiction argument}. Indeed, suppose that we have proven Theorem \ref{main_thm_isometric_case} for maps in $\mathcal{A}_{M,\theta}$ for some $\theta\in (0,1)$ sufficiently small. According to the \textit{Lipschitz truncation argument} provided by Lemma \ref{good_behaviour_of_deficits_in_terms_of_Lipschitz_truncation}, for the general case it suffices to prove that
\begin{align*}
\mathrm{sup} \left\{\frac{\underset{O\in O(3)}{\mathrm{min}}\d2\big|\t u-OP_T\big|^2}{\delta(u)+\varepsilon(u)},\quad \text{where } u\in W^{1,\infty}(\s;\R^3)\colon \d2 u=0,\ \|\t u\|_{L^\infty}\leq cM\right\}\ll +\infty\,,
\end{align*}
whenever the denominator above is non-zero. Arguing by contradiction, suppose that the latter is false. Then, for every $k\in \mathbb{N}$ there exist $u_k$ with mean value 0, Lipschitz norm bounded by $cM$, $\delta(u_k)+\varepsilon(u_k)>0$, and $O_k\in \underset{O\in O(3)}{\mathrm{Argmin}}\ \d2\big|\t u_k-OP_T\big|^2$, such that
\begin{equation}\label{contradictory_estimate_1_isometric}
\d2\big|\t u_k-O_kP_T\big|^2 \geq k\big(\delta(u_k)+\varepsilon(u_k)\big)\,.
\end{equation}
In particular,
\begin{equation*}
\delta(u_k)+\varepsilon(u_k)\leq \frac{1}{k}\d2\big|\t u_k-O_kP_T\big|^2\leq \frac{2\big(c^2M^2+2\big)}{k}\,,
\end{equation*}
and letting $k\to\infty$ we see that along this sequence, $\lim_{k\to\infty}(\delta(u_k)+\varepsilon(u_k))=0$. By Lemma \ref{isometries_compactness_general} and up to passing to a subsequence, we can find $O_0\in O(3)$ so that $u_k\to O_0\mathrm{id}_{\s}$ strongly in $W^{1,2}(\s;\R^3)$. Without loss of generality (up to considering $O_0^tu_k$ instead of $u_k$ if necessary) we can also suppose that $O_0=I_3$, so there exists $k_0:=k_0(\theta)\in \mathbb{N}$ such that 
\begin{equation*}
\big\|\t u_k-P_T\big\|_{L^2}\leq \theta\quad \forall  k\geq k_0\,,
\end{equation*}
i.e., $u_k\in \mathcal{A}_{M,\theta}$ for all $k\geq k_0$. Therefore, by assumption, there should exist $(R_k)_{k\geq k_0}\subset O(3)$ (and actually in $SO(3)$) such that
\begin{equation*}
\d2 \big|\t u_k-O_kP_T\big|^2\leq \d2 \big|\t u_k-R_kP_T\big|^2\lesssim \delta(u_k)+\varepsilon(u_k) \quad \forall k\geq k_0\,,
\end{equation*}
which contradicts \eqref{contradictory_estimate_1_isometric}. \qedhere 
\end{proof}  
  
\subsection{Proof of the local version of Theorem \ref{main_thm_isometric_case}}\label{proof_of_local_isometric_case}
By the reductions we have performed in the previous two subsections, we are left with proving a \textit{local version} of Theorem  \ref{main_thm_isometric_case}. This will be done by perturbing quantitatively the idea of proof of Lemma  \ref{isometries_compactness_general}.
\begin{proposition}\label{local_version_main_theorem_isometric_case}
There exists a constant $\theta\in (0,1)$ so that for every $u\in \mathcal{A}_{M,\theta}$ \AAA $\rm{(}$defined in \eqref{local_family_of_maps_sufficient_for_isometric_case}$\rm{)}$, \EEE there exists $R\in SO(3)$ such that
\begin{equation}\label{local_version_of_main_estimate_isometric}
\d2 \big|\t u-RP_T\big|^2\lesssim \delta(u)+\varepsilon(u)\,.
\end{equation}
\end{proposition}
\begin{proof}  First of all, it obviously suffices to prove \AAA \eqref{local_version_of_main_estimate_isometric} \EEE in the regime where both deficits are sufficiently small, say 
\begin{equation}\label{isometric_isoperimetric_deficits_small}	
\AAA 0\leq \delta(u)\leq\delta_0\ll 1, \quad 0\leq \varepsilon(u)\leq\varepsilon_0\ll 1\,,
\end{equation} 
for some absolute constants $\delta_0, \varepsilon_0>0$ which will also be chosen sufficiently small later. By using \eqref{norm_of_grad_u_k_estimate_from_isometric_deficit} with $u$ instead of $u_k$, we have
\begin{equation*}\label{auxest_general}
\d2 |u|^2=\frac{1}{2}\d2|\t u|^2-\left(\frac{1}{2}\d2|\t u|^2-\d2|u|^2\right)\leq 1+c\delta(u)-\left(\frac{1}{2}\d2|\t u|^2-\d2|u|^2\right)\,,
\end{equation*}
and therefore \eqref{basic_trick_isometric_estimate_qualitative}, with $u$ instead of $u_k$, would now give us
\begin{equation*}\label{basic_estimates}
\Big(1-\varepsilon(u)\Big)^2\leq \left(\d2 |u|^2\right)\Big(1+c\delta(u)\Big)\leq \left(1+c\delta(u)-\left(\frac{1}{2}\d2|\t u|^2-\d2|u|^2\right)\right)\Big(1+c\delta(u)\Big)\,.
\end{equation*}
Since by \AAA \eqref{Poincare} we have \EEE $\frac{1}{2}\d2|\t u|^2-\d2|u|^2\geq 0$, we can rearrange the terms 
and use \eqref{isometric_isoperimetric_deficits_small} 
to arrive at the estimate \EEE
\begin{equation}\label{prefinal_estimate_general}
\d2 \big|\t u-\nabla u_h(0)P_T\big|^2\lesssim\ \frac{1}{2}\d2|\t u|^2-\d2 |u|^2 \lesssim \ \delta(u)+\varepsilon(u)\,,
\end{equation}
the first inequality in which, is justified as follows. Let 
$u_h:\overline{B_1}\mapsto \R^3$ be the harmonic continuation of $u$ in the interior of $B_1$, being taken componentwise. The quantity in the middle \AAA of \eqref{prefinal_estimate_general} \EEE is the deficit of $u$ in the $L^2$-Poincare inequality for maps with zero average on $\s$. For every $k\in\mathbb{N}$, let $H_{k}$ be the subspace of $W^{1,2}(\s;\R^3)$ consisting of vector fields whose components are all $k$-th order spherical harmonics \AAA(see also Appendix \ref{sec:C})\EEE, so that one has the orthogonal (with respect to the $W^{1,2}$-inner product) decomposition $W^{1,2}(\s;\R^3)=\bigoplus_{k=0}^\infty H_{k}$. Let also $\Pi_{k}$ be the corresponding orthogonal projection.
In our case of consideration, $\Pi_{0} u=\d2 u =0$, and it is straightforward to check that $\Pi_{1} u=\nabla u_h(0)x$. Since the first non-trivial eigenvalue of the Laplace-Beltrami operator on $\s$ is $\lambda_{1}=2$ and the second one is $\lambda_{2}=6$ \AAA(see \eqref{properties})\EEE, by orthogonally decomposing $u=\Pi_1u+(u-\Pi_1u)$, we have
\begin{align}\label{maps_w_o_linear_part_estimate}
\begin{split}
\frac{1}{2}\d2|\t u|^2-\d2|u|^2&= \frac{1}{2}\d2\big|\t u-\nabla u_h(0)P_T\big|^2-\d2\big|u-\nabla u_h(0)x\big|^2\\
&\geq\frac{1}{2}\d2\big|\t u-\nabla u_h(0)P_T\big|^2-\frac{1}{6}\d2\big|\t u-\nabla u_h(0)P_T\big|^2\\
&=\frac{1}{3} \d2\big|\t u-\nabla u_h(0)P_T\big|^2\,.
\end{split}
\end{align}
Hence, the only thing that is left to be justified \AAA in order to prove \eqref{local_version_of_main_estimate_isometric}, \EEE is why in \eqref{prefinal_estimate_general} the matrix 
\begin{equation}\label{def_of_A}
\AAA A:=\nabla u_h(0) 
\end{equation}
can be replaced by a matrix $R\in SO(3)$. In that respect, \AAA observe that by the mean-value property of harmonic functions, the basic $L^2$-estimate \ref{basic_harmonic_estimates} (whose simple proof is given at the end of Appendix \hyperref[sec:C]{C}) applied to the function $u-\mathrm{id}_{\s}$, and \eqref{local_family_of_maps_sufficient_for_isometric_case}, we obtain \EEE 
\begin{equation}\label{linear_part_theta_close_to_I}
\big|A-I_3\big|^2=\left|\dashint_{B_1}\nabla u_h-I_3\right|^2\leq \dashint_{B_1} \big|\nabla u_h-I_3\big|^2\leq\frac{3}{2}\d2\big|\t u-P_T\big|^2\leq \frac{3}{2}\theta^2\,.
\end{equation}
In particular, if $\theta\in (0,1)$ is sufficiently small,  \eqref{linear_part_theta_close_to_I} directly implies that
\begin{equation}\label{A_close_to_I_in_all_respects}
2\leq|A|^2\leq 4,\quad 2\leq|A^{-1}|^2\leq 4,\quad \frac{1}{2}\leq\mathrm{det}A\leq \frac{3}{2}\,.
\end{equation}
\EEE Using the polar decomposition $A=R_0\sqrt{A^tA}$ for some $R_0\in SO(3)$, \AAA the last inequality in \eqref{A_close_to_I_in_all_respects} (in particular the fact that $\mathrm{det}A>0$) and \eqref{linear_part_theta_close_to_I} yield \EEE
\begin{equation}\label{dist_linear_part_to_SO(3)_apriori_small}
\mathrm{dist}^2\big(A,SO(3)\big)=\big|A-R_0\big|^2\leq \big|A-I_3\big|^2\leq \frac{3}{2}\theta^2\,.
\end{equation}  
Labelling $0<\alpha_1\leq \alpha_2\leq \alpha_3$ the eigenvalues of $\sqrt{A^tA}$, and setting 
\begin{equation}\label{Lambda_definitions}
\lambda_i:=\alpha_i-1, \quad \lambda:=\sum_{i=1}^3\lambda_i, \quad \Lambda:=\left(\sum_{i=1}^3\lambda_i^2\right)^{\frac{1}{2}}\, \ \mathrm{for}\ i=1,2,3\,,
\end{equation} 
the inequality \eqref{dist_linear_part_to_SO(3)_apriori_small} \EEE can be rewritten as
\begin{equation}\label{square_L_inequality_isometric_case}
\Lambda^2=\AAA\mathrm{dist}^2(A,SO(3))=\EEE\sum_{i=1}^3(\alpha_i-1)^2=\left|\sqrt{A^tA}-I_3\right|^2\leq \frac{3}{2}\theta^2\ll 1\,.
\end{equation}
The key observation now is that when $\theta\in (0,1)$ is sufficiently small, a map $u\in \mathcal{A}_{M,\theta}$ satisfies the estimate
\begin{equation}\label{determinant_estimate_general}
|\mathrm{det}A-1|\lesssim \delta(u)+\varepsilon(u)\,.
\end{equation}
\AAA The proof of \eqref{determinant_estimate_general} is a bit more involved, and is therefore presented separately in Lemma \ref{separate_determinant_estimate}. \EEE Let us assume for the moment its validity, and see how to finish the proof of \AAA \eqref{local_version_of_main_estimate_isometric}. With the notations introduced in \eqref{def_of_A} and \eqref{Lambda_definitions}, \EEE we can write $$\mathrm{det}A-1=\prod_{i=1}^3\alpha_i-1=\prod_{i=1}^3 (\lambda_i+1)-1\,,$$
\EEE expand the polynomial in the eigenvalues \AAA and use \eqref{square_L_inequality_isometric_case}, \EEE to obtain
\begin{equation}\label{determinant_polynomial_isometric}
\mathrm{det}A-1=\lambda+\frac{1}{2}\big(\lambda^2-\Lambda^2\big)+\lambda_1\lambda_2\lambda_3\leq \lambda+\frac{1}{2}\big(\lambda^2-\Lambda^2\big)+3^{-3/2}\Lambda^3\leq \lambda+\frac{1}{2}\big(\lambda^2-\Lambda^2\big)+\tfrac{\theta}{3\sqrt{2}}\Lambda^2\,.
\end{equation}
\AAA Since $0<\tfrac{\theta}{3\sqrt2}<\tfrac{1}{3\sqrt{2}}<\tfrac{1}{4}$, after  rearranging terms in \eqref{determinant_polynomial_isometric} and using \eqref{determinant_estimate_general}, we get \EEE
\begin{equation}\label{L2_first_estimate_isometric}
\frac{\Lambda^2}{4}\leq \left(\lambda+\frac{\lambda^2}{2}\right)+|\mathrm{det}A-1|\implies \frac{\Lambda^2}{4}\leq \left(\lambda+\frac{\lambda^2}{2}\right)+c\Big(\delta(u)+\varepsilon(u)\Big)\,.
\end{equation}
In order to handle the term $\left(\lambda+\frac{\lambda^2}{2}\right)$, we proceed as follows. \AAA Using again the mean value property of harmonic functions, \eqref{basic_harmonic_estimates} applied to $u$ now, \EEE and the outcome of \eqref{norm_of_grad_u_k_estimate_from_isometric_deficit} with $u$ instead of $u_k$ here, we can estimate 
\begin{equation}\label{small_lambda_estimate_general}
\frac{|A|^2}{3}\AAA=\frac{1}{3}\left|\dashint_{B_1}\nabla u_h\right|^2\leq \frac{1}{3}\dashint_{B_1}|\nabla u_h|^2 \EEE\leq \frac{1}{2}\d2 |\t u|^2 \leq 1+c\delta(u)\,.
\end{equation}
With the notations introduced in \eqref{def_of_A} and \eqref{Lambda_definitions} we have
$$|A|^2=\sum_{i=1}^3\alpha_i^2=\sum_{i=1}^3(1+\lambda_i)^2=3+2\lambda+\Lambda^2\,,$$
and the last identity, together with \eqref{small_lambda_estimate_general}, implies that
\begin{equation}\label{small_lambda_estimate}
\lambda\leq -\frac{\Lambda^2}{2}+c\delta(u)\leq c\delta(u)\,.
\end{equation}
\EEE
Since $\lambda$ does not necessarily have a sign, we distinguish two cases:
\vspace{-0.5em}
\begin{itemize}
\item [$\mathrm{(}i\mathrm{)}$]  In the case $\lambda\leq0$, and since \AAA by \eqref{square_L_inequality_isometric_case} \EEE $|\lambda|\leq \sqrt{3}{\Lambda}\leq \frac{3}{\sqrt{2}} \theta\ll 1$, the term in the first parenthesis on the right hand side of \AAA \eqref{L2_first_estimate_isometric} \EEE is estimated by
\begin{equation*}
\lambda+\frac{\lambda^2}{2}\leq \lambda+\frac{3\theta}{2\sqrt{2}}|\lambda|=\left(1-\frac{3\theta}{2\sqrt{2}}\right)\lambda\leq 0\,,
\end{equation*}
since by choosing $\theta\in (0,1)$ even smaller if necessary, we can also achieve $1-\frac{3\theta}{2\sqrt{2}}>0$. The term $(\lambda+\frac{\lambda^2}{2})$ is therefore nonpositive in this case, and \eqref{L2_first_estimate_isometric} gives
\begin{equation*}\label{linear_part_is_close_to_so3_isometric}
\mathrm{dist}^2\big(A,SO(3)\big)=\Lambda^2\leq 4c\Big(\delta(u)+\varepsilon(u)\Big).
\end{equation*}	
\vspace{-2.8em}
\item [$\mathrm{(}ii\mathrm{)}$] In the case $\lambda> 0$, by \eqref{small_lambda_estimate} we have  $0<\lambda\leq c\delta(u)$, so again \AAA \eqref{L2_first_estimate_isometric} together with \eqref{isometric_isoperimetric_deficits_small} imply that \EEE
\begin{equation*}\label{case1}
\mathrm{dist}^2\big(A,SO(3)\big)=\Lambda^2 \lesssim \delta(u) +\delta^2(u)+c\big(\delta(u)+\varepsilon(u)\big)\lesssim \delta(u)+\varepsilon(u)\,.
\end{equation*}
\end{itemize}
\vspace{-1em}
\AAA In both cases, we obtain 
\begin{equation}\label{A_from_SO3}
|A-R_0|^2=\mathrm{dist}^2\big(A,SO(3)\big)\lesssim \delta(u)+\varepsilon(u)\,,
\end{equation}
and combining \eqref{A_from_SO3} with \eqref{prefinal_estimate_general} allows us to deduce \eqref{local_version_of_main_estimate_isometric} with $R:=R_0\in SO(3)$, and conclude.
\end{proof}

\AAA To complete the arguments, we finally give the proof of the estimate \eqref{determinant_estimate_general}, which for convenience of the reader we recall in the next lemma.

\begin{lemma}\label{separate_determinant_estimate}
\EEE Let $u\in \mathcal{A}_{M,\theta}$, defined in \eqref{local_family_of_maps_sufficient_for_isometric_case}. Then, the matrix $A:=\nabla u_h(0)$ $\rm{(}$see \eqref{def_of_A}$\rm{)}$ satisfies \begin{equation}\label{det_estmate_restated}
|\mathrm{det}A-1|\lesssim \delta(u)+\varepsilon(u)\,,
\end{equation}
where $\delta(u)$ and $\varepsilon(u)$ are as always defined by \eqref{def_isometric_deficit} and \eqref{def_isoperimetric_deficit} respectively, and are here supposed to further satisfy \eqref{isometric_isoperimetric_deficits_small}\,. 
\end{lemma}
\EEE\vspace{-1em}
\begin{proof}
The main trick is to  write the signed-volume in the isoperimetric deficit $\varepsilon(u)$ as \AAA the corresponding \textit{bulk integral} in $B_1$. In particular, using the identity \eqref{determinant_bulk_surface} (which we prove in Appendix \ref{sec:B}), we have \EEE
\begin{equation}\label{volume_term_1st_expansion_isometry}
V_3(u)=\dashint_{B_1} \mathrm{det}\nabla u_h=\mathrm{det}A\dashint_{B_1}  \mathrm{det}\big(I_3+\nabla w_h\big)\,,
\end{equation}
where 
\begin{equation}\label{defintition_of_w}
w(x):=A^{-1}\big(u(x)-Ax\big)\,.
\end{equation}
\AAA By the fact that $u\in \mathcal{A}_{M,\theta}$ and \eqref{A_close_to_I_in_all_respects}, \EEE the map $w$ also satisfies
\begin{equation}\label{w_average_0_tilde_w_Lipschitz_isometric}
\overline{w}:=\d2 w=0,\quad \|\t w\|_{L^\infty}=\|A^{-1}\t u-P_T\|_{L^{\infty}}\leq \tilde c:=2cM+\sqrt{2}\,,
\end{equation}
and we will not distinguish further between the universal constants $c$ and $\tilde c$.
Because of \eqref{prefinal_estimate_general}, \AAA \eqref{def_of_A} and \eqref{A_close_to_I_in_all_respects}\EEE, we actually get
\begin{equation}\label{prefinal_estimate_general_for_w}
\d2 |\t w|^2\leq |A^{-1}|^2\d2 \big|\t u-AP_T\big|^2\lesssim \delta(u)+\varepsilon(u)\,.
\end{equation}
Now, in the rightmost hand side of \eqref{volume_term_1st_expansion_isometry} we can use the expansion of the determinant around $I_3$, \AAA i.e., the identity \eqref{determinant_expansion_around_I} which is proved in Appendix \ref{sec:B}, \EEE according to which,
\begin{equation}\label{expansion_of_volume}
\dashint_{B_1} \mathrm{det}(I_3+\nabla w_h)=1+3\d2 \langle w,x\rangle+Q_{V_3}(w)+\d2 \big\langle w,\partial_{\tau_1}w\wedge\partial_{\tau_2}w\big\rangle\,, 
\end{equation}
where \AAA the quadratic form $Q_{V_3}(w)$ is explicitly given in \eqref{quadratic_form_Q_V_3}. \EEE 
Notice that the linear term is vanishing, because \AAA the definitions of $A:=\nabla u_h(0)$ and $w$ in \eqref{defintition_of_w}, \EEE together with the mean value property of harmonic functions, imply that \EEE
\begin{equation}\label{linear_term_vanishing_isometry}
3\d2 \langle w,x\rangle=\dashint_{B_1} \mathrm{div} w_h=\dashint_{B_1} \mathrm{Tr}(\nabla w_h)=\mathrm{Tr}\dashint_{B_1} \nabla w_h
=\mathrm{Tr}\left[A^{-1}\left(\dashint_{B_1} \nabla u_h(x)-A\right)\right]=0\,.
\end{equation}
For the higher order terms one can argue as follows. \AAA Recalling the notation $\Pi_k$ for the projections onto the subspaces $H_k$ of the $k$-th order spherical harmonics (see the comments after \eqref{prefinal_estimate_general}), we note that \eqref{defintition_of_w} and \eqref{w_average_0_tilde_w_Lipschitz_isometric} directly imply that $\Pi_0w=\Pi_1w=0$. \EEE Hence, by the Cauchy-Schwarz inequality and the sharp Poincare inequality on $\s$ for $w$ \AAA (see the comment just below \eqref{Poincare})\EEE, we obtain 
\begin{align}\label{Q_vol3_estimated_Dirichlet_energy}
\big|Q_{V_3}(w)\big|&=\frac{3}{2}\Bigg|\d2\Big\langle w,(\mathrm{div}_{\s}w)x-\sum_{j=1}^3x_j\t w^j\Big\rangle\Bigg|
\leq\frac{3}{2}\Big(\d2 |w|^2\Big)^{\frac{1}{2}}\Bigg(\d2 \big|(\mathrm{div}_{\s}w)x\big|^2+\Big|\sum_{j=1}^3x_j\t w^j\Big|^2\Bigg)^{\frac{1}{2}}\notag\\
&\leq \frac{3}{2}\Big(\frac{1}{6}\d2|\t w|^2\Big)^{\frac{1}{2}}\Bigg(\d2 |\t w:P_T|^2+\Big(\sum_{j=1}^3x_j^2\Big)\Big(\sum_{j=1}^3|\t w^j|^2 \Big)\Bigg)^{\frac{1}{2}}\leq \frac{3}{2\sqrt{2}}\d2 |\t w|^2\,,
\end{align} 
and by \textit{Wente's isoperimetric inequality} \AAA(see \eqref{AM-GM-isoperimetric-inequality} for $n=3$)\EEE,
\begin{equation}\label{highest_order_remainder_estimated_Dirichlet_energy}
\left|\d2 \big\langle w,\partial_{\tau_1}w\wedge\partial_{\tau_2}w\big\rangle\right|\leq\left(\frac{1}{2}\d2 |\t w|^2\right)^{\frac{3}{2}}.
\end{equation} 
Therefore, by \AAA \eqref{Q_vol3_estimated_Dirichlet_energy}, \eqref{highest_order_remainder_estimated_Dirichlet_energy} and \eqref{prefinal_estimate_general_for_w}, together with the assumption \eqref{isometric_isoperimetric_deficits_small}, \EEE we estimate
\begin{align}\label{quadratic_of_volume_isometric}
\begin{split}
\Big|Q_{V_3}(w)+\d2 \big\langle w,\partial_{\tau_1}w\wedge\partial_{\tau_2}w\big\rangle\Big|&\lesssim \d2 |\t w|^2+ \left(\d2 |\t w|^2\right)^{\frac{3}{2}}\\
&\lesssim \big(\delta(u)+\varepsilon(u)\big)+\big(\delta(u)+\varepsilon(u)\big)^{\frac{3}{2}}\\
&\lesssim \delta(u)+\varepsilon(u)\ll 1\,.
\end{split}
\end{align}

In particular, since $\mathrm{det}A>0$ \AAA(see \eqref{A_close_to_I_in_all_respects}), by \eqref{volume_term_1st_expansion_isometry}, \EEE the expansion in \eqref{expansion_of_volume}, \AAA \eqref{linear_term_vanishing_isometry} \EEE and \eqref{quadratic_of_volume_isometric}, we deduce that $V_3(u)>0$, and we can finally consider two cases:
\begin{itemize}
\item[$(i)$] If $V_3(u)>1\implies \varepsilon(u)=0$, then by \AAA combining \eqref{volume_term_1st_expansion_isometry} with \eqref{expansion_of_volume} and \eqref{linear_term_vanishing_isometry}, \EEE rearranging terms, and then using again \AAA \eqref{A_close_to_I_in_all_respects}, \eqref{quadratic_of_volume_isometric}, the isoperimetric inequality and \eqref{norm_of_grad_u_k_estimate_from_isometric_deficit} with $u$ instead of $u_k$ here, \EEE we obtain the estimate
\begin{align*}\label{detgradu_at_0-1_first}
\left|\mathrm{det}A-1\right|&\leq \big(V_3(u)-1\big)+\mathrm{det}A\Big|Q_{V_3}(w)+\d2 \big\langle w,\partial_{\tau_1}w\wedge \partial_{\tau_2}w\big\rangle\Big|\\
&{\leq}  \left(\frac{1}{2}\d2 |\t u|^2\right)^{\frac{3}{2}}-1+c\Big(\delta(u)+\varepsilon(u)\Big)\\
&\leq \big(1+c\delta(u)\big)^{\frac{3}{2}}-1+c\delta(u)\lesssim \delta(u)\,.\\[-30pt]
\end{align*}
\item[$(ii)$] If $0\leq V_3(u)\leq 1$, then we can again similarly estimate,
\begin{align*}
|1-\mathrm{det}A|&\leq \big|1-V_3(u)\big|+\mathrm{det}A\Big|Q_{V_3}(w)+\d2 \big\langle w,\partial_{\tau_1}w\wedge \partial_{\tau_2}w\big\rangle\Big|\\
&\leq \big(1-V_3(u)\big)_++c\big(\delta(u)+\varepsilon(u)\big)\\
&\lesssim \delta(u)+\varepsilon(u)\,.
\end{align*}
This finishes the proof of \eqref{det_estmate_restated} in both cases, and allows us to conclude. \qedhere
\end{itemize}
\end{proof}
\vspace{-1em}
\subsection{The generalization to dimensions $n\geq 4$: Proof of Theorem \ref{main_thm_isometric_case_n_greater_than_3}}\label{Subsection 3.4}

By following closely the steps of proof of Theorem \ref{main_thm_isometric_case}, one can also prove its generalization in dimensions $n\geq 4$, i.e., Theorem \ref{main_thm_isometric_case_n_greater_than_3}. Regarding \textit{the extra assumption on an apriori bound in the $L^{2(n-2)}$-norm of $\t u$ in the latter}, let us first make the following short remark. When $n=3$, $n-1=2(n-2)=2$ and the assumption that $\t u$ is apriori bounded in $L^2$ is obsolete in this case, since we have anyway seen that it suffices to prove Theorem \ref{main_thm_isometric_case} for maps $u\in W^{1,2}(\s;\R^3)$ for which $0<\delta(u)\ll 1$, which trivially implies the bound $\|\t u\|_{L^2}\leq \sqrt {10}$ (recall \eqref{these_u_s_bounded_in_w_1_2}). \textit{In higher dimensions, the assumption is imposed by the growth behaviour of the signed-volume term with respect to $\t u$}. 

Indeed,  as we will see next, apart from the obvious differences in the proof due to the change in dimension, the only essential difference appears when we are trying to implement the Lipschitz truncation argument of Subsection  \ref{Reduction_to_Lipschitz_mappings}, in order to control both the isometric and the isoperimetric deficit of the Lipschitz truncated map in terms of the ones of the original map $u$.

\begin{proof}[\AAA Proof of Theorem \ref{main_thm_isometric_case_n_greater_than_3}]

Let $n\geq 4$, $M>0$ and $u\in \dot{W}^{1,2(n-2)}(\S;\R^n)$ with $\|\t u\|_{L^{2(n-2)}}\leq M$. Applying the analogue of Lemma  \ref{Lipschitz_truncation_lemma} in $W^{1,2}(\S;\R^n)$ for $M_n:=2\sqrt{n-1}$, we obtain again $u_{M_n}\in W^{1,\infty}(\S;\R^n)$ with $\|\t u_{M_n}\|_{L^\infty}\lesssim M_n$ 
\EEE and for which, exactly as in the estimates \eqref{isometric_deficits_well_controlled} (with $\sigma_{M_n,n-1}$ in the place of $\sigma_{M,2}$ now) and \eqref{1st_easy_estimate_for_difference_of_isoperimetric_deficits} of Lemma  \ref{good_behaviour_of_deficits_in_terms_of_Lipschitz_truncation},
\begin{equation}\label{relation_of_deficits_general_n}
\delta(u_{M_n})\lesssim \delta(u) \ \ \mathrm{and}\ \ \varepsilon(u_{M_n})\leq \varepsilon(u)+\big|V_n(u)-V_n(u_{M_n})\big|\,.
\end{equation}

Since $V_n$ is now of order $n-1>2$ in $\t u$, it is of course not expected that one can have an estimate of the form of \eqref{difference_in_signed_volumes_behaves_well_in_dimension_3} without any further assumption, since $V_n(u)$ is not even finite if $u$ does not belong to $W^{1,n-1}(\S;\R^n)$ at least. Nevertheless, under the imposed assumption \AAA that $\|\t u\|_{L^{2(n-2)}}\leq M$, \EEE the difference of the corresponding signed volumes \AAA in \eqref{relation_of_deficits_general_n} \EEE can be controlled as follows. Assuming again without loss of generality that $0<\delta(u)\leq 1$, \AAA adopting the notation in \eqref{notation_for_mean} \EEE and using \AAA the fact that $\|\t u_{M_n}\|_{L^\infty}\lesssim M_n:=2\sqrt{n-1}$, \EEE as well as \eqref{size_of_the_set_where they_differ} (in dimension $n\geq 4$ now), \EEE  we can estimate 
\begin{align}\label{difference_in_volume_higher_d_first_estimates}
\Big|V_n(u)-V_n(u_{M_n})\Big|&=\left|\ds\Big\langle u-\overline{u},\bigwedge_{i=1}^{n-1}\partial_{\tau_i}u\Big\rangle-\ds\Big\langle u_{M_n}-\overline{u_{M_n}} ,\bigwedge_{i=1}^{n-1}\partial_{\tau_i}u_{M_n}\Big\rangle\right|  \nonumber \\
&\leq\left|\ds\Big\langle (u_{M_n}-u)-\overline{(u_{M_n}-u)} ,\bigwedge_{i=1}^{n-1}\partial_{\tau_i}u_{M_n}\Big\rangle\right|+\left|\ds\Big\langle u-\overline{u} ,\bigwedge_{i=1}^{n-1}\partial_{\tau_i}u_{M_n}-\bigwedge_{i=1}^{n-1}\partial_{\tau_i}u\Big\rangle\right| \nonumber \\
&\lesssim \ds\Big|(u_{M_n}-u)-\overline{(u_{M_n}-u)}\Big|+ R_n(u) 
\nonumber\\
&\lesssim\left(\ds\Big|\t u_{M_n}-\t u\Big|^2\right)^{\frac{1}{2}} + R_n(u)
\nonumber\\
&\lesssim\delta(u)+ R_n(u)\EEE\,,
\end{align}
where 
\begin{equation}\label{remainder_term_in_volume_higher_dim}
R_n(u):= \left|\ds\Big\langle u-\overline{u},\bigwedge_{i=1}^{n-1}\partial_{\tau_i}u_{M_n}-\bigwedge_{i=1}^{n-1}\partial_{\tau_i}u\Big\rangle\right|\,. 
\end{equation}
\EEE
By the Sobolev embedding and our assumption, we further have 
\begin{equation}\label{Linfty_estimate}
\AAA \|u-\overline{u} \|_{L^\infty}\lesssim\left\|\t u\right\|_{L^{2(n-2)}}\lesssim M\,.
\end{equation}
Therefore, \AAA by the fact that $\{u_{M_n}=u\}\subseteq\{\t u_{M_n}=\t u\}$ $\mathcal{H}^{n-1}$-a.e., \AAA \eqref{Linfty_estimate}, \EEE the first inequality in \eqref{size_of_the_set_where they_differ} (in dimension $n\geq 4$), \AAA and the assumption that $\|\t u\|_{L^{2(n-2)}}\leq M$,  the remainder term in \AAA \eqref{remainder_term_in_volume_higher_dim} \EEE can be estimated further by
\begin{align}\label{R_n_estimate}
\begin{split}
R_n(u)&\sim\left|\int_{\{u\neq u_{M_n}\}}\Big\langle u-\overline{ u},\bigwedge_{i=1}^{n-1}\partial_{\tau_i}u_{M_n}-\bigwedge_{i=1}^{n-1}\partial_{\tau_i}u\Big\rangle\right|\lesssim M\int_{\{u\neq u_{M_n}\}}\left|\bigwedge_{i=1}^{n-1}\partial_{\tau_i}u_{M_n}-\bigwedge_{i=1}^{n-1}\partial_{\tau_i}u\right| \\
&\lesssim M\left(\mathcal{H}^{n-1}\Big(\{u\neq u_{M_n}\}\Big)+\int_{\{u\neq u_{M_n}\}}\big|\t u\big|^{n-1}\right)  \\
&\lesssim_M\left(\delta^2(u)+\left\|\t u\right\|_{L^{2(n-2)}}^{n-2}\left(\int_{\{u\neq u_{M_n}\}}|\t u|^2\right)^{\frac{1}{2}}\right) \\
&\lesssim_{M}\delta^2(u)+\left(M_n^2\mathcal{H}^{n-1}\big(\{u\neq u_{M_n}\}\big)+\int_{\{|\t u|> {M_n}\}}|\t u|^2\right)^{\frac{1}{2}}\lesssim_{M}\delta^2(u)+\delta(u)\ \lesssim_{M}\delta(u)\,.
\end{split}
\end{align} 
Therefore, under this extra assumption for $n\geq 4$, \AAA in view of \eqref{difference_in_volume_higher_d_first_estimates} and \eqref{R_n_estimate}, \eqref{relation_of_deficits_general_n} \EEE implies again that
\begin{equation}\label{final_volume_estimate}
\big|V_n(u)-V_n(u_{M_n})\big|\lesssim_{M} \delta(u)\implies \varepsilon(u_{M_n})\lesssim_{M} \varepsilon(u)+\delta(u)\,. 
\end{equation} 
Hence, under the assumption that $\|\t u\|_{L^{2(n-2)}}\leq M$ when $n\geq 4$, we can again reduce to proving Theorem \ref{main_thm_isometric_case_n_greater_than_3} for Lipschitz maps that enjoy an apriori dimensional upper bound on their Lipschitz constant. The proof can then be continued exactly as in Subsections \ref{Further reduction to maps W1,2-close to the identity} and \ref{proof_of_local_isometric_case}, with the obvious modifications in the dimensional constants and the algebraic expressions involved.

For instance, and just for the sake of clarity, we note that in this higher dimensional setting, $\frac{1}{2}\d2 |\t u|^2$ should be replaced in the corresponding arguments by  $\frac{1}{n-1}\ds |\t u|^2$, but the arguments go through exactly in the same way, since \textit{$n-1$ is both the norm of the gradient of isometric maps from $\S$ to $\R^n$ and also the first nontrivial eigenvalue of $-\Delta_{\S}$} (the second being $2n$, \AAA see Appendix \ref{sec:C} and \eqref{maps_w_o_linear_part_estimate}\EEE). In this respect, the estimate  \eqref{prefinal_estimate_general} should of course be replaced by 
\begin{equation*}
\frac{n+1}{2n(n-1)}\ds \big|\t u-\nabla u_h(0)P_T\big|^2\lesssim \frac{1}{n-1}\ds|\t u|^2-\ds |u|^2 \lesssim_{M} \delta(u)+\varepsilon(u)\,,
\end{equation*}
and analogously to \eqref{expansion_of_volume}, the expansion of the signed-volume around the identity is now
\begin{equation*}
\dashint_{B_1} \mathrm{det}(I_n+\nabla w_h)=1+n\ds \langle w,x\rangle+Q_{V_n}(w)+o\left(\Big(\ds |\t w|^2\Big)\right).\\[3pt] 
\end{equation*}
Modulo these changes, the proof remains essentially unchanged, which is left to the reader to verify. \qedhere
\end{proof}

\begin{remark}
An interesting question would be whether for $n\geq 4$ the apriori bound on the $L^{2(n-2)}$-norm of $\t u$, imposed as an assumption in Theorem  \ref{main_thm_isometric_case_n_greater_than_3}, can be replaced by one in an $L^p$-norm for some $p\in (n-1,2(n-2))$. \textit{The previous approach indicates however that the exponent $2(n-2)=(n-1)+(n-3)$ in the assumption is the sharpest one}.

Indeed, let us assume that $\left\| \t u \right\|_{L^{n-1+\gamma}}\leq M$ for some $\gamma\in (0,n-3)$ and $M>0$. Then, for $\alpha \in (0,n-1)$ and $p>1$, we can apply Hölder's inequality and use again the analogues of the estimates in \eqref{size_of_the_set_where they_differ} for $n\geq 4$, to deduce as before that  
\begin{align*}\label{alternative_est_leading_nowhere}
\begin{split}
\int_{\{u\neq u_{M_n}\}}\big|\t u\big|^{n-1}&\leq \left(\int_{\{u\neq u_{M_n}\}}|\t u|^{\frac{p(n-1-\alpha)}{p-1}}\right)^{\frac{p-1}{p}}\left(\int_{\{u\neq u_{M_n}\}}|\t u|^{\alpha p}\right)^{\frac{1}{p}}\\
&\lesssim_p\big\|\t u\big\|_{L^{\frac{p(n-1-\alpha)}{p-1}}}^{n-1-\alpha}\left(M_n^{\alpha p}\mathcal{H}^{n-1}\big(\{u\neq u_{M_n}\}\big)+\int_{\{u\neq u_{M_n}\}\cap\{|\t u|>M_n\}}|\t u|^{\alpha p}\right)^{\frac{1}{p}}\,.
\end{split}
\end{align*}
As long as $\alpha p\leq 2$, by Hölder's inequality again, 
\begin{align*}
\int_{\{u\neq u_{M_n}\}\cap\{|\t u|>M_n\}}|\t u|^{\alpha p}\leq \Big(\mathcal{H}^{n-1}\big(\{u\neq u_{M_n}\}\big)\Big)^{\frac{2-\alpha p}{2}}\left(\int_{\{|\t u|>M_n\}}|\t u|^2\right)^{\frac{\alpha p}{2}}\,,
\end{align*}
and \eqref{size_of_the_set_where they_differ} (for $n\geq 4$) would finally give us 
\begin{equation}\label{end_exponent}
\int_{\{u\neq u_{M_n}\}}\big|\t u\big|^{n-1}\lesssim_{\alpha,p}\big\|\t u\big\|_{L^{\frac{p(n-1-\alpha)}{p-1}}}^{n-1-\alpha}\cdot\delta^{\frac{2}{p}}(u)\lesssim_{\alpha,p}M^{n-1-\alpha}\delta^{\frac{2}{p}}(u)\,,
\end{equation}
as long as $\AAA 1\leq \EEE\frac{p(n-1-\alpha)}{p-1}\leq n-1+\gamma$. Therefore, \AAA by \eqref{end_exponent} and following closely the estimates used to arrive at \eqref{difference_in_volume_higher_d_first_estimates} and \eqref{R_n_estimate}, \EEE we deduce that the optimal exponent with which $\delta(u)$ can appear \AAA in \eqref{final_volume_estimate} \EEE through these estimates is exactly
\begin{equation*}
\frac{2}{p'}=\frac{2(\gamma+\alpha)}{n-1+\gamma},\quad \mathrm{\ achieved\ for\ } p':=\frac{n-1+\gamma}{\gamma+\alpha}. 
\end{equation*}
But this value should also satisfy the inequality $\alpha p'\leq 2$, which implies that
\begin{equation*}
\alpha\leq \frac{2\gamma}{n-3+\gamma}\implies
\frac{2}{p'}\leq \frac{2\gamma}{n-3+\gamma}<1\,,
\end{equation*}
for $0<\gamma <n-3$, i.e., the exponent of $\delta(u)$ \AAA in \eqref{final_volume_estimate} \EEE would become suboptimal.
\end{remark}
\vspace{-0.5em}

\section{Proof of Theorem \ref{main_thm_conformal_case}}\label{Section 4}
\subsection{Reduction to maps $W^{1,2}$-close to the $\mathrm{id}_{\s}$ and linearization of the problem}\label{Subsection 4.1}
We recall that for a map $u\in W^{1,2}(\s;\R^3)$ with $V_3(u)\neq 0$, its \textit{combined conformal-isoperimetric deficit} is denoted by
\begin{equation}\label{notation_for_conformal_isoperimetric_deficit}
\mathcal{E}_2(u):=\frac{\Big[D_2(u)\Big]^{\frac{3}{2}}}{|V_3(u)|}-1\geq 0\,,
\end{equation}
where $D_2(u)$, $V_3(u)$ are as in the statement of Theorem \ref{main_thm_conformal_case}, and that $\mathcal{E}_2(u)=0$ iff $u$ is a Möbius transformation of $\s$, up to a translation vector and a dilation factor.
 
To pass from the nonlinear deficit $\mathcal{E}_2$ to its linearized version, we make use of the following compactness result, whose proof can be found for instance in \cite{brezis1985convergence} or \cite[Lemma 2.1]{caldiroli2006dirichlet} (stated on $\R^2$ rather than $\s$ therein).

\begin{lemma}\label{compactness_lemma_for_conformal_n_3}
Let $(u_k)_{k\in \mathbb{N}}\subset W^{1,2}(\s;\R^3)$ be such that $V_3(u_k)\neq 0$, and for which
\begin{equation*}
\mathcal{E}_2(u_k)\to 0 \mathrm{\ \ as\ } k\to\infty\,.
\end{equation*}
Then, there exist Möbius transformations $(\phi_k)_{k\in \mathbb{N}}$ of $\s$, $(\lambda_k)_{k\in \mathbb{N}}\subset \R_+$ and $O\in O(3)$, such that
\begin{equation*}
\frac{1}{\lambda_k}\left(u_k\circ\phi_k-\d2 u_k\circ\phi_k\right)\to O\mathrm{id}_{\s} \ \ \mathrm{strongly\ in \ } W^{1,2}(\s;\R^3)\,.
\end{equation*}
\end{lemma}
Using this compactness lemma and the invariances of the combined conformal-isoperimetric deficit, with a contradiction argument as the one we used in the proof of Corollary \ref{it_suffices_to_prove_the_local_isometric_case}, one can now prove the following.

\begin{corollary}\label{it_suffices_to_prove_the_local_coformal_case}
It suffices to prove the $W^{1,2}$-local version of Theorem \ref{main_thm_conformal_case}, i.e., to prove it for maps 
\begin{equation}\label{local_family_of_maps_sufficient_for_conformal_case}
u\in \mathcal{B}_{\theta,\varepsilon_0}:=\left\{u \in W^{1,2}(\s;\R^3):
\begin{array}{lr}
(i)\ \ \ \d2 u=0\\
(ii)\ \big\| \t u-P_T\big\|_{L^2}\leq \theta\\
(iii)\ \ \mathcal{E}_2(u)\leq \varepsilon_0
\end{array}\right\},
\end{equation}
where $\theta, \varepsilon_0\in (0,1)$ are sufficiently small constants that will be suitably chosen later.
\end{corollary}
\begin{proof}
The fact that without loss of generality we can assume $(i)$ is obvious because \eqref{main_estimate_conformal_case} is translation invariant. That we can assume property $(iii)$ is also immediate, because for every $u\in W^{1,2}(\s;\R^3)$ with $V_3(u)\neq 0$, choosing $\lambda:=\|\t u\|_{L^2}>0$ and $\phi:=\mathrm{id}_{\s}$ we have that  $\d2\big|\frac{1}{\|\t u\|_{L^2}}\t u-P_T\big|^2\leq 6$. Therefore, it suffices to prove the desired estimate in the \textit{small-deficit regime}.\\[-15pt]

Suppose now that we have proven Theorem \ref{main_thm_conformal_case} for maps in $\mathcal{B}_{\theta,\varepsilon_0}$, but for the sake of contradiction the theorem fails to hold globally. Then, for every $k\in \mathbb{N}$ there exists $u_k\in W^{1,2}(\s;\R^3)$ with $V_3(u_k)\neq 0$ such that for every pair $(\lambda,\phi)\in \R_+\times Conf(\s)$;
\begin{equation*}
\d2\left|\frac{1}{\lambda}\t u_k-\t \phi\right|^2\geq k\mathcal{E}_2(u_k)\,.
\end{equation*}
Choosing $(\lambda,\phi):=(\|\t u_k\|_{L^2},\mathrm{id}_{\s})$ we obtain
$\mathcal{E}_2(u_k)\leq \frac{6}{k}\to 0$, as $k\to \infty$. We can then use Lemma \ref{compactness_lemma_for_conformal_n_3}, and argue as in the end of the proof of Corollary \ref{it_suffices_to_prove_the_local_isometric_case}, to arrive at a contradiction for the corresponding maps \linebreak $\tilde u_k:=\frac{1}{\lambda_k}O^t\left(u_k\circ\phi_k-\d2 u_k\circ\phi_k\right)$.
\end{proof}

\begin{remark}\label{we_can_assume_volume_is_positive}
Note that $V_3(u)>0$ whenever $u\in \mathcal{B}_{\theta,\varepsilon_0}$ with $\theta\in (0,1)$ sufficiently small. Indeed, recalling the expansions and the estimates \eqref{identity_difference_in_volumes}-\eqref{bound_of_R_1_2_3_4}, with the $\mathrm{id}_{\s}$ in place of $u_M$ here, we can arrive exactly as in the second half of the proof of Lemma \ref{Lipschitz_truncation_lemma} \AAA(recall \eqref{difference_in_signed_volumes_behaves_well_in_dimension_3}) \EEE at the estimate
\begin{equation*}
|V_3(u)-1|=|V_3(u)-V_3(\mathrm{id}_{\s})|\lesssim \|\t u-P_T\|_{L^2}\leq \theta\ll 1\,.
\end{equation*}
\end{remark}

Having now reduced to showing Theorem \ref{main_thm_conformal_case} for mappings in $ \mathcal{B}_{\theta,\varepsilon_0}$, where $V_3(u)>0$, we can \textit{linearize the initial problem}, by making use of the following two lemmata.

\begin{lemma}\label{fixing_center_scale}
Given $\theta,\varepsilon_0 \in (0,1)$ sufficiently small, there exists $\tilde\theta\in (0,1)$ sufficiently small accordingly, so that after possibly replacing $\theta$ with $\tilde\theta$, we can assume that every $u\in \mathcal{B}_{\theta,\varepsilon_0}$ has the additional property that
\begin{equation}\label{fixing_scale}
\d2 \langle u,x\rangle= 1\,.
\end{equation}  
\end{lemma}
\begin{proof}
Let $u\in \mathcal{B}_{\theta,\varepsilon_0}$ \AAA(defined in \eqref{local_family_of_maps_sufficient_for_conformal_case}) \EEE with $0<\theta,\varepsilon_0\ll 1$, and set 
\begin{equation}\label{choice_of_scale}
\lambda_u:=\d2 \langle u,x\rangle\,.
\end{equation} By the cancellation property $\d2 x=0$ and the sharp Poincare inequality on $\s$ \AAA(see \eqref{Poincare}) \EEE we have,
\begin{align*}
\begin{split}
\big|\lambda_u-1\big|&=\left|\d2 \big\langle (u-x)-\overline{(u-x)},x\big\rangle\right|\leq \d2 \Big|(u-x)-\overline{(u-x)}\Big|\\[3pt]
&\leq\left(\d2\Big|(u-x)-\overline{(u-x)}\Big|^2\right)^{\frac{1}{2}}\leq \left(\frac{1}{2}\d2 |\t u-P_T|^2\right)^{\frac{1}{2}}\leq \frac{\theta}{\sqrt 2}\,,
\end{split}
\end{align*}
i.e.,
\begin{equation}\label{lambda_u_estimate}
0<1-\frac{\theta}{\sqrt{2}}\leq\lambda_u\leq1+\frac{\theta}{\sqrt{2}}\,.
\end{equation}
Hence, setting $\tilde{u}:=\frac{u}{\lambda_u}$, \AAA by \eqref{local_family_of_maps_sufficient_for_conformal_case} and \eqref{choice_of_scale} \EEE we have $$\d2 \tilde{u}=0,\quad \d2 \langle \tilde{u},x\rangle =1,\quad \mathcal{E}_2(\tilde{u})\leq \varepsilon_0\,,$$ and by using \eqref{lambda_u_estimate},
\begin{align*}
\big\|\t \tilde{u}-P_T\big\|_{L^2}\leq \frac{1}{\lambda_u}\big\|\t u-P_T\big\|_{L^2}+\left|\frac{1}{\lambda_u}-1\right|\big\|P_T\big\|_{L^2}\leq\tilde\theta:=\left(\frac{2}{1-\frac{\theta}{\sqrt 2}}\right)\theta\,.
\end{align*} 
Although the precise value of the new constant $\tilde \theta>0$ is not of major importance, what is more important is that $\lim_{\theta\to 0^+}\tilde \theta =0$, so that when we will finally choose $\theta>0$ sufficiently small, $\tilde \theta>0$ will be sufficiently small accordingly.\qedhere 
\end{proof}
\begin{lemma}\label{quadratization_of_conformal_deficit}
There exists a constant $\beta:=\beta(\theta,\varepsilon_0)>0$ that tends to $0$ as  $(\theta,\varepsilon_0)\to (0,0)$, such that the following holds. If $u\in \mathcal{B}_{\theta,\varepsilon_0}$ satisfies \eqref{fixing_scale} and one sets $w:=u-\mathrm{id}_{\s}$, then 
\begin{equation}\label{reduction_to_quadratic_estimate}
Q_3(w)\leq \mathcal{E}_2(u)+\beta\d2 |\t w|^2\,,
\end{equation} 
where 
\begin{equation}\label{quadratic_forms_without_proofs_}
Q_3(w):=\frac{3}{4}\d2 |\t w|^2-\frac{3}{2}\d2 \Big\langle w, (\mathrm{div}_{\s}w)x-\sum_{j=1}^3 x_j\t w^j\Big\rangle\,. 
\end{equation}
\end{lemma}
\begin{proof}
For $u$ as in the statement of the lemma, property \eqref{fixing_scale} can be rewritten as
\begin{equation}\label{linear_term_in_the_expansions_is_vanishing}
\d2 \mathrm{div}_{\s}w=2\d2 \langle w,x\rangle= 2\left(\d2 \langle u,x\rangle -1\right)=0\,.
\end{equation}
Then,
\begin{align}\label{D_3_expansion}
\begin{split}
\big[D_2(u)\big]^{\frac{3}{2}}&=\left(1+\d2 \mathrm{div}_{\s}w+\frac{1}{2}\d2 |\t w|^2\right)^{\frac{3}{2}}=\left(1+\frac{1}{2}\d2|\t w|^2\right)^{\frac{3}{2}}\\
&=1+\frac{3}{4}\d2 |\t w|^2+\mathcal{O}\Big(\Big(\d2 |\t w|^2\Big)^2\Big)\,.
\end{split}
\end{align}
Since $\frac{d^2}{dt^2}\Big|_{t=0}(1+t)^{\frac{3}{2}}=\frac{3}{4}$, we can take $\theta\in(0,1)$ small enough so that \AAA (by \eqref{local_family_of_maps_sufficient_for_conformal_case}$(ii)$) \EEE the higher-order term in the expansion \AAA\eqref{D_3_expansion} \EEE is estimated by
\begin{equation}\label{D_3_remainder}
\left|\mathcal{O}\Big(\Big(\d2 |\t w|^2\Big)^2\Big)\right|\leq \frac{1}{2}\left(\d2 |\t w|^2\right)^2\leq \frac{\theta^2}{2}\d2 |\t w|^2\,.
\end{equation}
\AAA Regarding the expansion of the signed volume term $V_3(u)$, as we calculate in detail in Lemma \ref{integral_identities} and by using \eqref{linear_term_in_the_expansions_is_vanishing}, we have \EEE
\begin{equation}\label{surface_volume_term_for_conformal}
V_3(u)=V_3(\mathrm{id}_{\s}+w)=1+Q_{V_3}(w)+V_3(w)\,,
\end{equation}
where 
\begin{equation}\label{Q_v_3_again}
Q_{V_3}(w):=\frac{3}{2}\d2 \Big\langle w,(\mathrm{div}_{\s}w)x-\sum_{j=1}^3x_j\nabla_Tw^j\Big\rangle\,.
\end{equation}
\AAA Hence, by using the expansions \eqref{D_3_expansion} and \eqref{surface_volume_term_for_conformal} in the definition \eqref{notation_for_conformal_isoperimetric_deficit} of the deficit, we obtain \EEE
\begin{equation*}
1+\frac{3}{4}\d2 |\t w|^2+\mathcal{O}\Big(\Big(\d2 |\t w|^2\Big)^2\Big)= \Big(1+\mathcal{E}_2(u)\Big)\Big(1+Q_{V_3}(w)+\d2 \big\langle w,\partial_{\tau_1}w\wedge\partial_{\tau_2}w\big\rangle\Big)\,,
\end{equation*}
\AAA and after rearranging terms, \EEE
\begin{equation}\label{expansion_rewritten}
Q_3(w)= \mathcal{E}_2(u)+\mathcal{E}_2(u)\cdot Q_{V_3}(w)+\big(1+\mathcal{E}_2(u)\big)\d2 \big\langle w,\partial_{\tau_1}w\wedge\partial_{\tau_2}w\big\rangle-\mathcal{O}\Big(\Big(\d2 |\t w|^2\Big)^2\Big)\,.
\end{equation}
Arguing exactly as in \eqref{Q_vol3_estimated_Dirichlet_energy} (with the Poincare inequality being applied with constant $\frac{1}{2}$ \AAA instead of $\frac{1}{6}$ \EEE in this case) and \eqref{highest_order_remainder_estimated_Dirichlet_energy}, \AAA and using again \eqref{local_family_of_maps_sufficient_for_conformal_case}$(ii)$, \EEE we have 
\begin{equation}\label{remainder_terms_estimate}
\big|Q_{V_3}(w)\big|\leq \sqrt{\frac{27}{8}}\d2 |\t w|^2\,, \quad \left|\d2 \big\langle w,\partial_{\tau_1}w\wedge\partial_{\tau_2}w\big\rangle\right|\leq\left(\frac{1}{2}\d2 |\t w|^2\right)^{\frac{3}{2}}\leq 2^{-3/2}\theta\d2 |\t w|^2\,.
\end{equation} 
\AAA Therefore, \eqref{expansion_rewritten}, the estimates \eqref{D_3_remainder} and \eqref{remainder_terms_estimate} for the remainder terms, and \eqref{local_family_of_maps_sufficient_for_conformal_case}$(iii)$ imply that\EEE
\begin{align*}
\begin{split}
Q_3(w)&\leq \mathcal{E}_2(u)+\mathcal{E}_2(u)| Q_{V_3}(w)|+\big(1+\mathcal{E}_2(u)\big)\Big|\d2 \big\langle w,\partial_{\tau_1}w\wedge\partial_{\tau_2}w\big\rangle\Big|+\Big|\mathcal{O}\Big(\Big(\d2 |\t w|^2\Big)^2\Big)\Big|\\
&\leq \mathcal{E}_2(u)+\beta\d2|\t w|^2\,,
\end{split}
\end{align*}
where the precise value of the constant is $\beta:=\sqrt{\frac{27}{8}}\varepsilon_0+2^{-3/2}(1+\varepsilon_0)\theta+ \frac{\theta^2}{2}$.
\end{proof}
\AAA In view of Lemma \ref{quadratization_of_conformal_deficit}, \EEE if we thus choose $\varepsilon_0\in (0,1)$ sufficiently small and then $\theta\in (0,1)$ sufficiently small accordingly, the last term on the right hand side of \eqref{reduction_to_quadratic_estimate} can be set to be a sufficiently small multiple of the Dirichlet energy of $w$. Therefore, we can move our focus of attention on the coercivity properties of the resulting quadratic form $Q_3$ \AAA defined in \eqref{quadratic_forms_without_proofs_}\EEE, which is just the second derivative of the nonlinear combined conformal-isoperimetric deficit $\mathcal{E}_2(u)$ at the $\mathrm{id}_{\s}$. This will be the content of the next subsection. 
\subsection{On the coercivity of the quadratic form $Q_3$}\label{Subsection 4.2}
\textit{For the most part of this subsection the results hold true in every dimension $n\geq 3$}. Since we will use them also in Section \ref{sec: 5 linear_stability}, where we prove linear stability estimates in all dimensions, \textit{\AAA we also denote here the ambient dimension $3$ with the general letter $n$\EEE} (in order to avoid the repetition of the arguments in \AAA Section \ref{sec: 5 linear_stability}\EEE), and hope that no confusion will be caused to the reader. Our goal is to examine the coercivity properties of the quadratic form $Q_n$ \AAA in \eqref{quadratic_forms_without_proofs_}\EEE. By the reductions we have performed \AAA(see \eqref{local_family_of_maps_sufficient_for_conformal_case} and \eqref{fixing_scale})\EEE, this can be considered in the space 
\begin{equation}\label{H_space}
H_n:=\left\{ w\in W^{1,2}(\S;\mathbb{R}^n)\colon  \ds w = 0,\  \ds \langle w,x\rangle= 0\ \right\}\,.
\end{equation}
Similarly to \AAA the notation introduced in the proof of Proposition \ref{local_version_main_theorem_isometric_case} \EEE in Subsection \ref{proof_of_local_isometric_case}, for every $k\geq 1$ we define  \textit{$H_{n,k}$ to be the linear subspace of $H_n$ consisting of those maps in $H_n$, all the components of which are $k$-th order spherical harmonics} \AAA(cf.~also Appendix \ref{sec:C})\EEE, 
and also define
\begin{equation}\label{Hksol_def}
\tilde H_{n,k}:=\left\{ w_h:\overline {B_1}\mapsto \R^n :
\begin{array}{lr}
\ \Delta w_h\ =\ 0 \mathrm{\ in \ } B_1 \\
\ w_h|_{\S} \in H_{n,k}
\end{array}\right\},
\end{equation}
so that $\bigoplus\limits_{k=1}^{\infty} \tilde H_{n,k}$ is a $W^{1,2}$-orthogonal decomposition of the vector space of harmonic maps $w_h:\overline {B_1}\mapsto\R^n$ for which $w_h(0)=0$ and $\mathrm{Tr}\nabla w_h(0)=0$, \AAA the last identities following immediately from their equivalent formulation on $\S$ in \eqref{H_space}\EEE. Actually, for every $k\geq 1$ we can further consider the $W^{1,2}(B_1)$-Helmholtz decomposition
\begin{equation}\label{Hktilde_decomposition}
\tilde H_{n,k}=\tilde H_{n,k,\mathrm{sol}}\oplus \tilde H_{n,k,\mathrm{sol}}^{\bot}\,,
\end{equation}
where 
\begin{equation}\label{tildeHksol_definition}
\tilde H_{n,k,\mathrm{sol}}:=\left\{w_h\in \tilde H_{n,k}\colon \mathrm{div}w_h\equiv 0 \mathrm{\ in \ } B_1 \right\}\,,
\end{equation}
and $\tilde {H}_{n,k,\mathrm{sol}}^{\bot}$ is its orthogonal complement in $W^{1,2}(B_1;\R^n)$. In view of the $k$-homogeneity of the maps in $\tilde{H}_{n,k}$ \AAA in \eqref{Hksol_def}\EEE, we can write \AAA the equivalent to \eqref{Hktilde_decomposition} \EEE $W^{1,2}$-decomposition also on $\S$, namely
\begin{equation}\label{Hk_decomposition}
H_{n,k}=H_{n,k,\mathrm{sol}}\oplus H_{n,k,\mathrm{sol}}^{\bot}\,,
\end{equation} 
where
\begin{equation}\label{Hksol_definition}
H_{n,k,\mathrm{sol}}:=\left\{ w\in H_{n,k}: w_h\in \tilde H_{n,k,\mathrm{sol}}\right\},
\end{equation}
and $H_{n,k,\mathrm{sol}}^{\bot}$ is its $W^{1,2}(\S;\R^n)$-orthogonal complement. \AAA Hence, adopting from now on all these notations introduced in \eqref{H_space}-\eqref{Hksol_definition}, \EEE let us also denote 
\begin{equation}\label{def_dimensions}
N_{n,k}:=\mathrm{\ dim}H_{n,k}<\infty,\ N_{1,n,k}:=\mathrm{\ dim}H_{n,k,\mathrm{sol}},\ N_{2,n,k}:=\mathrm{\ dim}H_{n,k,\mathrm{sol}}^{\bot}\,,
\end{equation}
so that $N_{n,k}\ = N_{1,n,k}+N_{2,n,k}$.

Recall also that the second derivative of the signed-volume term $V_n$ at the $\mathrm{id}_{\S}$ corresponds to the bilinear form
\begin{equation}\label{quadratic_form_of_volume}
Q_{V_n}(v,w):=\frac{n}{2}\ds \big\langle v,A(w)\big\rangle \mathrm{\ \ \ for\ }v,w\in H_n\,,
\end{equation}
where the associated linear first-order differential operator $A$ is defined as
\begin{equation}\label{Aoperator}
A(w):=(\mathrm{div}_{\S}w)x-\sum_{j=1}^n x_j\nabla_Tw^j \mathrm{\ \ \ for\ } w\in H_n\,.
\end{equation}
\AAA For $n=3$, the expression \eqref{quadratic_form_of_volume} for the bilinear form $Q_{V_n}(\cdot,\cdot):H_n\times H_n\mapsto \R$ is essentially derived in the proof of Lemma \ref{integral_identities}, see \eqref{determinant_expansion_around_I},\eqref{quadratic_form_Q_V_3} and \eqref{quadratic_term_V3} therein.  Another \textit{intrinsic calculation} for its computation in any dimension is also given in Lemma \ref{expansions_n} at the end of Appendix \ref{sec:B}. \EEE 

\textit{The main feature that we are going to use in this subsection is the fine interplay between the operator $A$ and the above defined spaces, as it is properly described in the following.}
\begin{lemma}\label{A_invariances}
For every $k\geq 1$, the operator $A$ \AAA defined in \eqref{Aoperator} \EEE is a linear self-adjoint isomorphism of the spaces $H_{n,k,\mathrm{sol}}$ and $H_{n,k,\mathrm{sol}}^{\bot}$ \AAA defined in and after \eqref{Hksol_definition}, \EEE with respect to the $W^{1,2}$-inner product.
\end{lemma}
\vspace{-1em}
\begin{proof}
First of all, it is immediate that $A$ is self-adjoint with respect to the $L^2$-inner product in $H_n$, \AAA since it arises as the second derivative of $V_n$ at the $\mathrm{id}_{\S}$, but it is also easy to verify directly after integrating by parts  that for any $v,w\in H_n$,
\begin{equation*}
\ds \langle v,A(w)\rangle=\ds \langle A(v),w\rangle\,.
\end{equation*} 
Note also that, since $(H_{n,k,\mathrm{sol}})_{k\geq 1}$ and $(H_{n,k,\mathrm{sol}}^{\bot})_{k\geq 1}$ are subspaces of the $k$-th order spherical harmonics, the $W^{1,2}$- and the $L^2$-inner products restricted on these subspaces are equivalent (see \eqref{properties}).
\EEE It is also easy to check that for every $k\geq 1$,
\begin{equation}\label{AleavesH_nk_invariant}
w\in H_{n,k}\implies A(w)\in H_{n,k}\,.
\end{equation}
Indeed, \AAA as we mention in the beginning of Appendix \ref{sec:C}, \EEE for $k\geq 1$ fixed and $w \in H_{n,k}$, its harmonic extension  $w_h$ in $\overline {B_1}$ is an $\R^n$-valued homogeneous harmonic polynomial of degree $k$, and $ \forall j=1,2,\dots,n$, 
\begin{equation*}
\nabla w_h^j\ =\ \t w^j+kw^jx,  \ \ \ \mathrm{div}w_h\ =\mathrm{div}_{\S} w+k\langle w,x\rangle \ \mathrm{\ on\ } \S\,.
\end{equation*}
Therefore, the operator $A$ can alternatively be rewritten as
\begin{align}\label{Equation_for_full_gradient_general}
A(w)&=(\mathrm{div}_{\S}w)x-\sum_{j=1}^n x_j\nabla_Tw^j =(\mathrm{div}w_h)x-\sum_{j=1}^n x_j\nabla w_h^j \ \ \mathrm{on} \ \S\,.
\end{align}
Writing $A$ in terms of the full gradient and divergence operators on $\S$ \AAA as in \eqref{Equation_for_full_gradient_general}, we see that \EEE 
\begin{align}\label{mean_value_A_is_0}
\ds A(w)&=\ds \Big((\mathrm{div}w_h)x-\sum_{j=1}^n x_j\nabla w_h^j\Big)=\frac{1}{n}\dashint_{B_1} \Big(\nabla\mathrm{div}w_h-\sum_{j=1}^{n}\partial_j\nabla w_h^j\Big)=0\,,
\end{align}
and \AAA by \eqref{H_space}, \EEE
\begin{align}\label{mean_value_normal_part_of_A_is_0}
\ds \big\langle A(w),x\big\rangle&=\ds \mathrm{div}_{\S}w=(n-1)\ds \langle w,x\rangle=0\,.
\end{align}
It is also straightforward to verify that 
\begin{equation}\label{A(w)_h_formula}
[A(w)]_h=(\mathrm{div}w_h)x-\sum_{j=1}^n x_j\nabla w_h^j \ \ \ \mathrm{in\ } \overline{B_1}\,,
\end{equation} 
so $[A(w)]_h$ is also an $\R^n$-valued homogeneous harmonic polynomial of degree $k$, and therefore its restriction on $\S$ is an $\R^n$-valued $k$-th order spherical harmonic. \AAA In total, \eqref{mean_value_A_is_0}-\eqref{A(w)_h_formula} yield the implication in \eqref{AleavesH_nk_invariant}. \EEE
Directly from \eqref{A(w)_h_formula} one can also verify that $A$ leaves $H_{n,k,\mathrm{sol}}$ invariant, i.e.,
\begin{equation*}
w\in H_{n,k,\mathrm{sol}}\implies \mathrm{div}[A(w)]_h\equiv 0 \ \ \mathrm{in\ }B_1\,, 
\end{equation*}
as well. It remains to be checked that 
\begin{equation}\label{zero_kernel_in_H_n_sol}
\mathrm{ker}A=\{0\}\ \ \mathrm{in\ } H_{n,k,\mathrm{sol}}\,. 
\end{equation}
Indeed, let $w\in H_{n,k,\mathrm{sol}}$ be such that 
\begin{equation*}\label{w_in_kernel_of_A}
\hspace{-0.3em}A(w):=(\mathrm{div}_{\S}w)x-\sum_{j=1}^n x_j\nabla_Tw^j=0\AAA\iff |A(w)|^2=(\mathrm{div}_{\S}w)^2+\Big|\sum_{j=1}^n x_j\nabla_Tw^j\Big|^2\equiv 0\EEE \ \ \mathrm{on} \ \S\,.
\end{equation*}
\AAA Note that since $w\in H_{n,k,\mathrm{sol}}$ is the restriction on $\S$ of an $\R^n$-valued homogeneous harmonic polynomial of degree $k$ (hence smooth up to the boundary), the above equation holds true in the classical sense. Hence, by orthogonality, \EEE both the normal and the tangential part of $A(w)$ would have to vanish identically, namely,
\begin{equation}\label{normal_tangential_parts_vanishing}
\mathrm{div}_{\S}w=0 \ \ \mathrm{and} \ \ \sum_{j=1}^nx_j\t w^j=0 \ \ \mathrm{on}\ \S\,.
\end{equation}
By the definition of $H_{n,k,\mathrm{sol}}$ \AAA in \eqref{Hksol_definition} \EEE we have that $\mathrm{div}w_h\equiv0$ in $B_1$, and therefore \AAA \eqref{normal_tangential_parts_vanishing} \EEE implies that
\begin{equation}\label{normal_part_vanishing_identically}
\langle w,x\rangle=\frac{1}{k}(\mathrm{div}w_h-\mathrm{div}_{\S}w)=0 \ \ \mathrm{on} \ \S\,.
\end{equation}
Testing now the second one of the equations in \eqref{normal_tangential_parts_vanishing} with the vector field $w$ itself, integrating  by parts on $\S$ and using \eqref{normal_part_vanishing_identically}, we obtain 
\begin{align}\label{very_same_calculation}
\begin{split}
0 &\ =\ - \ds\Big\langle \sum_{j=1}^nx_j\nabla_Tw^j,w\Big\rangle=-\sum_{j=1}^n \ds\Big\langle \nabla_Tw^j,x_jw^T\Big\rangle\\
&\ =\ds |w|^2-n\ds \langle w,x\rangle^2=\ds |w|^2\,,
\end{split}
\end{align}
i.e., $w\equiv 0$ on $\S$. This concludes the proof \AAA of \eqref{zero_kernel_in_H_n_sol}, \EEE and thus the proof of the fact that $A$ is a self-adjoint linear isomorphism of $H_{n,k,\mathrm{sol}}$. Hence, $A$ leaves $H_{n,k,\mathrm{sol}}^{\bot}$ invariant as well, and is actually also an isomorphism of it, as we will see next. \qedhere
\end{proof}
As a consequence of Lemma \ref{A_invariances}, each one of the finite-dimensional  subspaces $(H_{n,k,\mathrm{sol}})_{k\geq 1}$ and $(H_{n,k,\mathrm{sol}}^{\bot})_{k\geq 1}$ admit an eigenvalue decomposition with respect to $A$ \AAA(cf.~\cite[Chapter 8, Theorem 4.3]{lang1987linear}). \EEE

\begin{theorem}\label{eigenvalue_decomposition} 
The following statements hold.
\vspace{-0.75em}
\begin{itemize}	
\item[$\rm{(}$i$\rm{)}$] For every $k\geq 1$, the subspace $H_{n,k,\mathrm{sol}}$ \AAA in \eqref{Hksol_definition}, \EEE has an eigenvalue decomposition with respect to the operator A, \AAA defined in \eqref{Aoperator}, \EEE  as
\begin{equation}\label{H_n_k_sol_eigenvalue_decomposition}
H_{n,k,\mathrm{sol}}=H_{n,k,1}\oplus H_{n,k,2}\,,
\end{equation}
where $H_{n,k,1}$ is the eigenspace of $A$ corresponding to the eigenvalue $\sigma_{n,k,1}:=-k$ and $H_{n,k,2}$ is the one corresponding to the eigenvalue $\sigma_{n,k,2}:=1$.
\vspace{-0.75em}
\item[$\rm{(}$ii$\rm{)}$]  For every $k\geq 1$, the subspace $H_{n,k,3}:=H_{n,k,\mathrm{sol}}^{\bot}$ is an eigenspace with respect to A corresponding to the eigenvalue $\sigma_{n,k,3}:=k+n-2$\,.
\end{itemize}
\end{theorem}
\vspace{-1.5em}
\begin{proof} As we have just remarked \AAA before the statement of Theorem \ref{eigenvalue_decomposition}\EEE, for every $k\geq 1$ there exists a $W^{1,2}$-orthonormal basis of eigenfunctions $\{w_{n,k,1},\dots,w_{n,k,N_{1,n,k}}\}$ for the subspace $H_{n,k,\mathrm{sol}}$ \AAA(see also \eqref{def_dimensions}) \EEE and similarly, $\{w_{n,k,N_{1,n,k}+1},\dots,w_{n,k,N_{n,k}}\}$ for $H_{n,k,\mathrm{sol}}^{\bot}$, i.e., for $i=1,\dots,N_{n,k}$, the map $w_{n,k,i}$ satisfies the eigenvalue equation
\begin{equation}\label{Eigenvalue_equation_for_A}
A(w_{n,k,i}):=(\mathrm{div}_{\S}w_{n,k,i})x-\sum_{j=1}^n x_j\nabla_Tw_{n,k,i}^j=\sigma_{n,k,i}w_{n,k,i} \mathrm{\ \ on \ } \S\,. 
\end{equation}
For each such eigenvalue $\sigma_{n,k,i}$ we denote its corresponding eigenspace by $H_{n,k,i}$. If in \AAA\eqref{Eigenvalue_equation_for_A} \EEE  we take the inner product with the unit normal vector field on $\S$, we obtain further that each eigenfunction $w_{n,k,i}$ satisfies the equation
\begin{equation}\label{Equation_for_divs}
\mathrm{div}_{\S}w_{n,k,i}=\sigma_{n,k,i}\langle w_{n,k,i},x\rangle \mathrm{ \ \ on\ } \S\,,
\end{equation} 
which in terms of the full divergence can be rewritten as
\begin{equation}\label{Equation_for_full_div}
\mathrm{div}(w_{n,k,i})_h\ =\ \mathrm{div}_{\S} w_{n,k,i}+\langle \partial_{\vec\nu}(w_{n,k,i})_h,x\rangle\ =\ (\sigma_{n,k,i}+k)\ \langle w_{n,k,i},x\rangle\ \mathrm{\ on\ }  \S\,.
\end{equation}

We now fix the index $k\geq 1$ and consider all the different possible cases that will allow us to find the eigenvalues of $A$ in the invariant subspaces $H_{n,k,\mathrm{sol}}$ and $H_{n,k,\mathrm{sol}}^{\bot}$ respectively.
\vspace{-0.5em}
\begin{itemize}
\item[($a_1$)] Let $w$ be a non-trivial eigenfunction of $A$ in $H_{n,k,\mathrm{sol}}$, so
\begin{equation}\label{divequal0}
\mathrm{div}w_h= 0 \mathrm{\ in\ } \overline {B_1} \ \ \iff \ \ \mathrm{div}w_h = 0 \ \ \mathrm{on} \  \ \S\,,
\end{equation}
due to \AAA \eqref{Hksol_def}, \eqref{tildeHksol_definition} \EEE and the $(k-1)$-homogeneity of $\mathrm{div}w_h$ in this case. By \eqref{Equation_for_full_div} we see that one possibility for \AAA \eqref{divequal0} \EEE to hold, is for the eigenvalue $\sigma=-k$. We thus set $\sigma_{n,k,1}:=-k$ and label its corresponding eigenspace as
$$H_{n,k,1}:=\mathrm{span}\{w_{n,k,1},\dots,w_{n,k,p_{n,k}}\}\,,$$  where $p_{n,k}:=\mathrm{dim}H_{n,k,1}$.
\item[($a_2$)] Let now $w$ be a non-trivial eigenfunction of $A$ in $H_{n,k,\mathrm{sol}}$, with $w\in H_{n,k,1}^{\bot}$. Then, \AAA in view of \eqref{Equation_for_full_div}, \EEE the only possibility for \AAA\eqref{divequal0} \EEE to hold is iff
\begin{equation}\label{normal_part_0}
\langle w, x \rangle \equiv 0 \ \ \mathrm{on\ } \S.
\end{equation}
In this case, $w$ is a tangential vector field and by \eqref{Equation_for_divs} and \AAA \eqref{normal_part_0}, \EEE we have $\mathrm{div}_{\S}w\equiv 0$ on $\S$, as well. The eigenvalue equation \eqref {Eigenvalue_equation_for_A} reduces then to
\begin{equation}\label{reduced_equation_on_S}
\sigma w=-\sum_{j=1}^n x_j\nabla_Tw^j \mathrm{\ \ on \ } \S\,.
\end{equation} 
With the very same calculations that we performed in the proof of Lemma \ref{A_invariances} \AAA(see \eqref{very_same_calculation})\EEE, we can test \AAA\eqref{reduced_equation_on_S} \EEE with $w$, integrate by parts \AAA and use \eqref{normal_part_0}\EEE, to obtain
\begin{equation*}
\sigma\ \ds |w|^2 = - \ds\Big\langle w, \sum_{j=1}^nx_j\nabla_Tw^j\Big\rangle\ =\ds |w|^2-n\ds\langle w,x\rangle^2 =\ds |w|^2\,.
\end{equation*}
We label this eigenvalue $\sigma_{n,k,2}:= 1$, and its corresponding eigenspace as $$H_{n,k,2}:=\mathrm{span}\{w_{n,k,p_{n,k}+1},\dots,w_{n,k,N_{1,n,k}}\}\,,$$ 
\AAA and in this way we are led to the decomposition \eqref{H_n_k_sol_eigenvalue_decomposition}. \EEE 
\vspace{-0.25em}
\item[($b$)] Let us now look at eigenfunctions $w$ of $A$ in $H_{n,k,\mathrm{sol}}^{\bot}$, where the divergence of $w_h\in \tilde{H}_{n,k}$ \AAA in \eqref{Hksol_def} \EEE does not vanish identically in $\overline {B_1}$. Since $w_h$ is an $\R^n$-valued $k$-homogeneous harmonic polynomial, we have that $\mathrm{div}w_h$ is a scalar $(k-1)$-homogeneous harmonic polynomial, and therefore its restriction on $\S$ is a scalar $(k-1)$-spherical harmonic. We can then apply the Laplace-Beltrami operator \AAA(see \eqref{laplace_Beltrami_definition}) \EEE on both sides of \eqref{Equation_for_full_div} \AAA and use again \eqref{Equation_for_divs}\EEE, to obtain
\begin{align*}
(k-1)(k+n-3)\mathrm{div}w_h\ &=\ -\Delta_{\S}(\mathrm{div}w_h)\ = \ -(\sigma+k)\Delta_{\S}\big(\langle w,x\rangle\big)\\
& =\ (\sigma+k)\Big(\langle -\Delta_{\S} w,x\rangle-2\t w:P_T+\langle w, -\Delta_{\S} x\rangle\Big)\\
& =\ \Big(k(k+n-2)-2\sigma+(n-1)\Big)(\sigma+k)\langle w, x\rangle\\
& =\ \Big(k(k+n-2)-2\sigma+(n-1)\Big)\mathrm{div}w_h \ \mathrm{\ on \ } \S\,. 
\end{align*}
Since in this case $\mathrm{div}w_h$ does not vanish identically, we conclude that 
\begin{equation*}
k(k+n-2)-2\sigma+(n-1)\ =\ (k-1)(k+n-3)\iff \sigma=  k+n-2\,.
\end{equation*}
We label this eigenvalue as $\sigma_{n,k,3}:= k+n-2$ and its corresponding eigenspace as $H_{n,k,3}$. In particular, we have found that $H_{n,k,\mathrm{sol}}^{\bot}=H_{n,k,3}$. \qedhere
\end{itemize}
\end{proof}
\vspace{-1em}
\begin{remark}\label{triviality_of_H_1_3}
We have obtained in total the $W^{1,2}$-orthogonal decomposition of our space of interest into eigenspaces of $A$ as
\begin{equation}\label{A_eigenspace_decomposition_of_H}
H_n:=\bigoplus_{k=1}^{\infty}\left(H_{n,k,1}\oplus H_{n,k,2}\oplus H_{n,k,3}\right)\,.
\end{equation}
It is easy to construct examples showing that except for $H_{n,1,3}$, none of these eigenspaces are apriori trivial. The triviality of $H_{n,1,3}$ is a consequence of the fact that we had already scaled properly our initial maps $u$, so that the corresponding maps $w$ satisfy $\ds \langle w,x\rangle=0$ \AAA(recall \eqref{H_space})\EEE. Indeed, let $w(x):=\Lambda x \in H_{n,1,3}$ for some $\Lambda\in\R^{n\times n}$. By assumption,
\begin{equation*}
0=\ds \langle w,x\rangle=\ds \langle \Lambda x,x\rangle=\frac{1}{n}\mathrm{Tr}\Lambda\,.
\end{equation*}
Therefore, $\mathrm{div}w_h\equiv\mathrm{Tr}\Lambda\equiv 0$ in $\overline{B_1}$, i.e., $w\in H_{n,1,\mathrm{sol}}=H_{n,1,3}^{\bot}$, forcing $w\equiv 0$, \AAA and thus, $H_{n,1,3}=\{0\}$.
\end{remark}

\textit{The eigenvalue decomposition \AAA\eqref{A_eigenspace_decomposition_of_H} \EEE of $H_n$ into eigenspaces of $A$ is valid for every $n\geq 3$}. In the case of interest of this subsection, i.e., in dimension $n=3$, it immediately gives the desired coercivity estimate for the quadratic form $Q_3$ \AAA defined in \eqref{quadratic_forms_without_proofs_}, \EEE \textit{with optimal constant}. 
\textit{For the rest of this subsection we switch back to denoting the ambient dimension by the number $3$.} As a consequence of Theorem \ref{eigenvalue_decomposition}, we have
\begin{lemma}\label{Qdiag_in_dim_3}
The following statements hold true.
\begin{itemize}	
\item[$(i)$] The forms $Q_{V_3}$ and $Q_3$\AAA, defined in \eqref{Q_v_3_again} and \eqref{quadratic_forms_without_proofs_} respectively, \EEE diagonalize on each one of the subspaces $(H_{3,k,i})_{k\geq 1,i=1,2,3}$, i.e., there exist explicit constants $(c_{3,k,i})_{i=1,2,3}$ and $(C_{3,k,i})_{i=1,2,3}$ so that for  every $w\in H_{3,k,i}$, 
\begin{equation}\label{Qvol3_Q3}
Q_{V_3}(w) = c_{3,k,i}\d2 |\t w|^2 \quad \mathrm{and\quad }Q_3(w)= C_{3,k,i}\d2 |\t w|^2\,. 
\end{equation}
\item[$(ii)$] For every $k,l\geq 1$ and $i,j=1,2,3$ with $(k,i)\neq (l,j)$, the subspaces $H_{3,k,i}$ and $H_{3,l,j}$ are also $Q_{V_3}$- and $Q_3$-orthogonal, i.e., for every $w_{k,i}\in H_{3,k,i}$ and $w_{l,j}\in H_{3,l,j}$,
\begin{equation}\label{H_k_is_Q_vol3_Q_3_orth}
Q_{V_3}(w_{k,i},w_{l,j})=0 \quad\mathrm{and} \quad Q_3(w_{k,i},w_{l,j})=0\,.
\end{equation} 
\end{itemize} 
\end{lemma}

\begin{proof}
The proof is immediate. For part $(i)$ of the lemma, if we denote by $\lambda_{3,k}:=k(k+1)$ the eigenvalues of $-\Delta_{\s}$ \AAA (see \eqref{properties})\EEE, then
\begin{equation*}
\d2 |w|^2=\frac{1}{\lambda_{3,k}}\ \d2|\t w|^2\quad \forall w\in H_{3,k}\,.
\end{equation*}
In particular, for $i=1,2,3$, if $w\in H_{3,k,i}$, \AAA by the definition \eqref{quadratic_form_of_volume} (for $n=3$) and Theorem \ref{eigenvalue_decomposition}, \EEE we have 
\begin{align*}
Q_{V_3}(w)=\frac{3}{2}\d2\big\langle w,A(w)\big\rangle=\frac{3\sigma_{3,k,i}}{2}\d2|w|^2=\frac{3\sigma_{3,k,i}}{2\lambda_{3,k}}\d2|\t w|^2,
\end{align*}
which is precisely the first identity in \eqref{Qvol3_Q3} for $c_{3,k,i}:=\frac{3\sigma_{3,k,i}}{2\lambda_{3,k}}$, and then
\begin{equation*}
Q_3(w)=C_{3,k,i}\d2|\t w|^2\,,
\end{equation*}
where $C_{3,k,i}:=\tfrac{3}{4}-c_{3,k,i}$. We list below the precise values of the constants, which are important in this case, since we will need to sum up the identities for $Q_3$ in the subspaces $(H_{3,k,i})_{k\geq 1, i=1,2,3}$, in order to obtain an estimate on the full space $H_3$.  
\begin{align}\label{constants1}
\begin{split}
&c_{3,k,1}=\frac{-3}{2(k+1)},\quad  
c_{3,k,2}=\frac{3}{2k(k+1)},\quad
c_{3,k,3}=\frac{3}{2k}\,,\\
&C_{3,k,1}=\frac{3(k+3)}{4(k+1)},\quad
C_{3,k,2}=\frac{3(k-1)(k+2)}{4k(k+1)},\quad 
C_{3,k,3}=\frac{3(k-2)}{4k}\,.
\end{split}
\end{align}
Part $(ii)$ of the lemma is immediate by the mutual $W^{1,2}$-orthogonality of $(H_{3,k,i})_{k\geq 1, i=1,2,3}$.\qedhere
\end{proof}
As an immediate consequence of Lemma \ref{Qdiag_in_dim_3} we obtain the desired estimate for the quadratic form $Q_3$ defined in \eqref{quadratic_forms_without_proofs_}.
\begin{theorem}\label{main_coercivity_estimate_3_d_conformal}
For every $w\in H_3$ \AAA\rm{(}given in \eqref{H_space}\rm{)}, \EEE the following \textit{coercivity estimate} holds.
\begin{equation}\label{Korn_type_inequality_for_Q_3}
Q_3(w)\geq \frac{1}{4}\d2 \big|\t w-\t(\Pi_{3,0} w)\big|^2,
\end{equation}
where $H_{3,0}:=H_{3,1,2}\oplus H_{3,2,3}$ is the kernel of $Q_3$ in $H_3$ and $\Pi_{3,0}$ is the $W^{1,2}$-orthogonal projection of $H_3$ onto $H_{3,0}$. The constant $\frac{1}{4}$ in the previous estimate is sharp.
\end{theorem}
\vspace{-1em} 
\begin{proof}
Having the precise values of the constants $(C_{3,k,i})_{k\geq 1,i=1,2,3}$ \AAA in \eqref{constants1}\EEE, we see that $C_{3,1,2}=C_{3,2,3}=0$, but otherwise it is easy to verify that 
\begin{equation}\label{sharp_coercivity_constant}
\tilde C:=\underset{\underset{(k,i)\neq (1,2),(1,3),(2,3)}{k\geq1,\ i\in\{1,2,3\}}}{\mathrm{min}}C_{3,k,i}=C_{3,3,3}=\frac{1}{4}\,.
\end{equation}
Now we can express any $w\in H_3$ as a Fourier series in terms of the eigenspace decomposition \AAA \eqref{A_eigenspace_decomposition_of_H}\EEE, i.e.,
$$w=\sum_{k,i} w_{3,k,i},\ \text{where } w_{3,k,i}\in H_{3,k,i},\ \forall k\geq 1, i=1,2,3\,.$$ 
Note that, as we have justified in Remark \AAA\eqref{triviality_of_H_1_3}\EEE, $w_{3,1,3}=0$. Expanding the quadratic form $Q_3$, \AAA and using \eqref{Qvol3_Q3}, \eqref{H_k_is_Q_vol3_Q_3_orth} and \eqref{sharp_coercivity_constant}, \EEE we indeed obtain
\begin{align}\label{Q_3_eigenspace_decomposition}
\begin{split}
Q_3(w)&=\sum_{(k,i) \in \mathbb{N}^{*}\times\{1,2,3\}}Q_3(w_{3,k,i})=\sum_{(k,i)\neq (1,2),(1,3),(2,3)}C_{3,k,i}\d2 \big|\t w_{3,k,i}\big|^2\\
&\geq \frac{1}{4}\sum_{(k,i) \neq (1,2),(1,3),(2,3)}\d2 \big|\t w_{3,k,i}\big|^2=\frac{1}{4}\d2 \big|\t w-\t(w_{3,1,2}+w_{3,2,3})\big|^2\,,
\end{split}
\end{align}
\AAA which finishes the proof of \eqref{Korn_type_inequality_for_Q_3}. \qedhere\EEE
\end{proof}

\subsection{Proof of the local version of Theorem \ref{main_thm_conformal_case}}\label{Sec: completion_of_proof_conformal_case_n_3}
The presence of the $Q_3$-degenerate space $H_{3,0}$ \AAA in the coercivity estimate \eqref{Korn_type_inequality_for_Q_3} \EEE is a small but natural obstacle to overcome in order to complete the proof of Theorem \ref{main_thm_conformal_case}. As we mentioned \AAA after the statement of Theorem \ref{main_coercivity_estimate_conformal_general_dimension} \EEE in the Introduction, since \textit{$H_{3,0}$ will eventually turn out to be isomorphic to the Lie algebra of infinitesimal Möbius transformations of $\s$}, at an infinitesimal level this basically means the following. Although, by the reductions we have made, the map $u$ is apriori supposed to be $\theta$-close to the $\mathrm{id}_{\s}$ in the $W^{1,2}$-topology \AAA (recall \eqref{local_family_of_maps_sufficient_for_conformal_case}$(ii)$)\EEE, there might be another Möbius transformation of $\s$ that is also $\theta$-close to the $\mathrm{id}_{\s}$ and is a better candidate for the nearest Möbius map to $u$ in terms of its combined conformal-isoperimetric deficit $\mathcal{E}_2(u)$ in  \AAA \eqref{notation_for_conformal_isoperimetric_deficit}\EEE. Similarly to \cite{faraco2005geometric} and \cite{reshetnyak1970stability}, an application of the Inverse Function Theorem and a topological argument will allow us to identify this more suitable candidate\AAA, see the details in the subsequent Lemma \ref{fixingMobius} and its proof. \EEE \textit{For this purpose, we will need the following characterization of the $Q_3$-degenerate subspace $H_{3,0}$, which is valid  in every dimension $n\geq 3$, and that is why we now switch back to denoting the ambient dimension by $n$. }

\begin{lemma}\label{Projections} The following statements hold true.
\vspace{-0.5em}	
\begin{itemize}
\setlength\itemsep{-0.5em}
\item[$\rm{(}$i$\rm{)}$] The subspace $H_{n,1,2}$ \AAA in \eqref{H_n_k_sol_eigenvalue_decomposition} \EEE can be characterized as
\begin{equation}\label{H_n,1,2}
H_{n,1,2}=\big\{w\in H_n: w(x)=\Lambda x,\ \mathrm{where\ } \Lambda\in Skew(n)\big\}\,,
\end{equation}  
and its dimension is $\mathrm{dim}H_{n,1,2}=\frac{n(n-1)}{2}$. The projection on $H_{n,1,2}$ is therefore characterized by
\begin{equation}\label{Projection_on_H_n,1,2_characterization_full}
\Pi_{H_{n,1,2}}w=0\iff \nabla w_h(0)=\nabla w_h(0)^t\,.
\end{equation}
\item[$\rm{(}$ii$\rm{)}$] The subspace $H_{n,2,\mathrm{sol}}$  \AAA in \eqref{Hksol_definition} \EEE can be characterized as
\begin{equation}\label{H_n,sol}
H_{n,2,\mathrm{sol}}=\left\{w\in H_n:
\begin{array}{lr}
\forall k=1,\dots,n:\ \ w^k(x)=\langle \Lambda^kx,x\rangle\,,\\
\Lambda^k\in Sym(n): \mathrm{Tr}\Lambda^k=0,\  \sum_{l=1}^n\Lambda_{lk}^l=0
\end{array}\right\}\,,
\end{equation}
and thus
\begin{equation*}
\mathrm{dim}H_{n,2,3}=\mathrm{dim}H_{n,2}-\mathrm{dim}H_{n,2,\mathrm{sol}}=n.
\end{equation*}
The projection on $H_{n,2,3}$ is therefore characterized by
\begin{equation}\label{Projection_on_H_n,1,2_characterization}
\Pi_{H_{n,2,3}}w=0\iff \ds \big(\mathrm{div}w_h(x)\big)x=0\,.
\end{equation}
\end{itemize}
\end{lemma}
\vspace{-1.5em}
\begin{proof}
For part $(i)$ of the lemma, if $w\in H_{n,1,2}$ we can write it as $w(x)=\Lambda x$ for some $\Lambda\in \R^{n\times n}$. In this space, \AAA recalling  \eqref{normal_part_0}, \EEE 
\begin{align*}
\langle w,x\rangle \equiv0 \mathrm{\ \ on\ } \S\iff \sum_{1\leq i\leq j\leq n}(\Lambda_{ij}+\Lambda_{ji})x_ix_j\equiv 0 \mathrm{\ \ on\ } \S\iff \Lambda^t=-\Lambda\,.
\end{align*}
The characterization \AAA \eqref{Projection_on_H_n,1,2_characterization_full} \EEE of the projection $\Pi_{H_{n,1,2}}$ is then immediate. For part $(ii)$, let $w\in H_{n,2,\mathrm{sol}}$. \AAA By \eqref{tildeHksol_definition} and \eqref{Hksol_definition}, \EEE its harmonic extension is a homogeneous solenoidal harmonic polynomial of degree 2, so for each $k=1,\dots,n$, there exists $\Lambda^k\in Sym(n)$ such that
\begin{equation*}
w^k_h(x)=\langle \Lambda^kx,x\rangle \ \ \mathrm{in} \ \overline{B_1}\,.
\end{equation*} 
In particular, for each $k,l=1,\dots,n$, we can compute
\begin{equation}\label{computation_for_characterization_of_H_2_sol}
\partial_lw_h^k(x)=2\sum_{i=1}^n\Lambda^k_{li}x_i\implies \left\{\begin{array}{lr} 0=\frac{1}{2}\Delta w^k_h= \mathrm{Tr}\Lambda^k\,,\\
0=\frac{1}{2}\mathrm{div}w_h= \sum_{k=1}^n\left(\sum_{l=1}^n\Lambda_{lk}^l\right)x_k \iff \sum_{l=1}^n\Lambda_{lk}^l=0 \end{array}\right\}\,,
\end{equation}
\AAA so \eqref{H_n,sol} follows directly from \eqref{computation_for_characterization_of_H_2_sol}. \EEE For the last characterization, \AAA i.e., \eqref{Projection_on_H_n,1,2_characterization}, \EEE we have 
\begin{equation*}
\Pi_{H_{n,2,3}}w=0\iff \Pi_{H_{n,2}}w\in H_{n,2,\mathrm{sol}}\iff \sum_{l=1}^n\left(\nabla^2w_h^l(0)\right)_{lk}=0 \ \mathrm{for\ every\ }k=1,\dots,n\,,
\end{equation*}
and by the mean value property of harmonic functions again, we can indeed equivalently write
\begin{align}
n\ds \big(\mathrm{div}w_h(x)\big)x_k=\dashint_{B_1}\partial_k\big(\mathrm{div}w_h(x)\big)=\sum_{l=1}^n\dashint_{B_1}\partial_{lk}w_h^l(x)=\sum_{l=1}^n\left(\nabla^2w_h^l(0)\right)_{lk}=0\,. \quad \quad \quad  \qedhere
\end{align}  
\end{proof}

\begin{remark}
It is worth noticing here that simply by counting dimensions,
\begin{equation*}
\mathrm{dim}H_{3,0}=\mathrm{dim}H_{3,1,2}+\mathrm{dim}H_{3,2,3}=6\,,
\end{equation*} 
which is also the dimension of $Conf(\s)$ \AAA(recall the notation in Section \ref{Notation} and see also Remark \ref{tangent_space_of_conformal_group_to_identity} for some more details on this Lie group and its corresponding Lie algebra). \EEE We therefore only need to verify that $H_{3,0}$ \AAA(introduced after \eqref{Korn_type_inequality_for_Q_3}) \EEE is actually isomorphic to the Lie algebra of infinitesimal Möbius transformations of $\s$, and then the Inverse Function Theorem can be applied. This is the context of the following.
\end{remark}
\vspace{-0.5em}
\begin{lemma}\label{fixingMobius}
Given $\theta, \varepsilon_0 \in (0,1)$ sufficiently small, there exists $\tilde \theta \in (0,1)$ that depends only on $\theta$ and is sufficiently small as well, so that for every $u\in \mathcal{B}_{\theta,\varepsilon_0}$ \AAA $\rm{(}$as in \eqref{local_family_of_maps_sufficient_for_conformal_case}$\rm{)}$ \EEE there exists $\phi\in Conf_+(\s)$ such that
\begin{equation}\label{identifying_the_correct_Moebius}
\left(u\circ\phi-\d2 u\circ\phi\right) \in \mathcal{B}_{\tilde\theta,\varepsilon_0} \ \ \mathrm{and\ also} \ \Pi_{3,0}(u\circ\phi)=0\,.
\end{equation} 	
\end{lemma}

\begin{proof} Given $u\in\mathcal{B}_{\theta,\varepsilon_0}$, \AAA in view of the characterizations \eqref{Projection_on_H_n,1,2_characterization_full} and \eqref{Projection_on_H_n,1,2_characterization}, \EEE let us introduce the map $\Psi_u:Conf_+(\s)\mapsto\R^{6}$, via
\begin{flalign}\label{Psi_u_def}
\begin{split}
\Psi_u(\phi):\ &=\ \left(  \frac{1}{3}\Big(\partial_j(u\circ\phi)^i_h(0)-\partial_i(u\circ\phi)^j_h(0)\Big)_{1\leq i<j\leq 3},\ \d2 \Big(\mathrm{div}(u\circ\phi)_h(x)\Big)x\right)\\
\ &=\ \left(\left(\d2 \left((u\circ\phi)^ix_j-(u\circ\phi)^jx_i\right)\right)_{1\leq i<j\leq 3},\ \d2 \Big(\mathrm{div}(u\circ\phi)_h(x)\Big)x\right)\,.
\end{split}
\end{flalign}
Our goal \AAA for \eqref{identifying_the_correct_Moebius} \EEE is essentially to show that $0\in \mathrm{Im}(\Psi_u)$. To simplify notation, let us also set 
\begin{equation}\label{Psi_identity}
\Psi:=\Psi|_{\mathrm{id}_{\s}}\,.
\end{equation}
Clearly, $\Psi(\mathrm{id}_{\s})=0$. In order to apply the Inverse Function Theorem, we look at the differential
\begin{equation*}
d\Psi|_{\mathrm{id}_{\s}}:T_{\mathrm{id}_{\s}}Conf_+(\s)\mapsto \R^{6}
\end{equation*}
and prove that it is a non-degenerate linear map. The differential of $\Psi$ at the $\mathrm{id}_{\s}$ is easy to compute. Indeed, by the linearity of all the operations involved, for every $Y \in T_{\mathrm{id}_{\s}}Conf_+(\s)$, defined via
\begin{equation}\label{Lie_algebra_representation}
Y(x):=Sx+\mu\big(\langle x,\xi\rangle x-\xi\big):\s\mapsto \R^3,\ \mathrm{where }\ S^t=-S, \ \xi\in\s,\ \mu\in \R\,,
\end{equation}
(see \AAA \eqref{tangent_space_of_the_conformal_group_at_the_identity_1} in Remark \ref{tangent_space_of_conformal_group_to_identity} for the derivation of this representation),
and with a slight abuse of notation in the domain of definition of $\Psi$ \AAA in \eqref{Psi_identity}\EEE, we can calculate
\begin{align}\label{dPsi_map}
\begin{split}
d\Psi|_{\mathrm{id}_{\s}}(Y):\ &=\frac{d}{dt}\Big|_{t=0}\Psi\left(\mathrm{exp}_{\mathrm{id}_{\s}}(tY)\right)=\Psi\left(\frac{d}{dt}\Big|_{t=0}\mathrm{exp}_{\mathrm{id}_{\s}}(tY)\right)=\Psi\left(Y\right)\\[3pt]
\ &=\left(\left(\d2 \big(Y^i(x)x_j- Y^j(x)x_i\big) \right)_{1\leq i<j\leq 3}, \d2\mathrm{div}Y_h(x)x \right)\,.
\end{split}
\end{align}
It is clear that the harmonic extension in $B_1$ of $Y$ \AAA as in \eqref{Lie_algebra_representation} \EEE is given by the vector field 
\begin{equation}\label{harmonic_extension_of_Y}
Y_h(x)=Sx+\mu\left(\langle x,\xi\rangle x-\left(\frac{|x|^2+2}{3}\right)\xi\right).
\end{equation}
In particular, \AAA \eqref{harmonic_extension_of_Y} directly implies that \EEE
\begin{equation}\label{Y_h_on_S_identity}
\d2 \big(Y^i(x)x_j- Y^j(x)x_i\big) =\frac{2}{3}S_{ij}\ \forall 1\leq i<j\leq 3\,,\quad \d2\mathrm{div}Y_h(x)x=\frac{10}{9}\mu\xi\,.
\end{equation}
Therefore, \AAA in view of \eqref{dPsi_map} and \eqref{Y_h_on_S_identity}, we indeed obtain \EEE $\mathrm{ker}(d\Psi|_{\mathrm{id}_{\s}})=\{0\}$, i.e., 
$$
d\Psi|_{\mathrm{id}_{\s}} \text{ is a linear isomorphism between } T_{\mathrm{id}_{\s}}Conf_+(\s) \text{ and }\R^{6}\,.
$$
\AAA From now on, we denote by $D_\sigma$ the open ball in the 6-dimensional vector space $T_{\mathrm{id}_{\s}}Conf_+(\s)$, centered at $0$ (or better said, at $\mathrm{id}_{\s}$) and of radius $\sigma>0$. \EEE
Since the exponential mapping $\mathrm{exp}_{\mathrm{id}_{\s}}(\cdot)$ is a local diffeomorphism between a neighbourhood of $0$ in $T_{\mathrm{id}_{\s}}Conf_+(\s)$ and a neighbourhood of $\mathrm{id}_{\s}$ in $Conf_+(\s)$, we can use the Inverse Function Theorem to find a sufficiently small \AAA $\sigma_0\in (0,1)$ such that for the \EEE open neighbourhood 
\begin{equation}\label{U_0_nbhd}
\mathcal U_0:=\mathrm{exp}_{\mathrm{id}_{\s}}(D_{\sigma_0})
\end{equation}
of\ the $\mathrm{id}_{\s}$ in $Conf_+(\s)$, the map 
\begin{equation}\label{Psi_local_diffeomorphism}
\Psi|_{\mathcal{U}_0}:\mathcal{U}_0\subseteq Conf_+(\s)\mapsto \Psi(\mathcal{U}_0)\subseteq \R^6 \mathrm{\  is\ a\ } C^1\text{-}\mathrm{diffeomorphism}\implies \mathrm{deg}(\Psi,0;\mathcal{U}_0)=1\,.
\end{equation}
As a next step,  \textit{we justify that $\Psi$ is homotopic to $\Psi_u$ in $\mathcal{U}_0$}. Indeed, for every $\phi\in \mathcal{U}_0$, we can estimate
\begin{align}\label{Psi_u_Psi_homotopic}
\begin{split}
\big|\Psi_u(\phi)-\Psi(\phi)\big|^2 \leq &\ \sum_{i\neq j}\d2 \left(\big[\left(u-\mathrm{id}_{\s}\right)\circ\phi\big]^i\right)^2x_j^2+\d2 \Big(\mathrm{div}\big[\left(u-\mathrm{id}_{\s}\right)\circ\phi\big]_h\Big)^2\\
\leq &\ \d2\big|(u-\mathrm{id}_{\s})\circ\phi\big|^2+6 \d2\big|\t\left[\left(u-\mathrm{id}_{\s}\right)\circ\phi\right]\big|^2\,.
\end{split}
\end{align}
In the last step \AAA of \eqref{Psi_u_Psi_homotopic} \EEE we used the general estimate
\begin{equation*}
\d2 (\mathrm{div}v_h)^2\leq 3\d2 |\nabla v_h|^2\leq 6\d2 |\t v|^2\ \ \ \mathrm{for}\ v\in W^{1,2}(\s;\R^3)\,,
\end{equation*}
\AAA the last inequality following from the second estimate in \hyperref[basic_harmonic_estimates]{(C.7)}, which is proved in Lemma \ref{harmest}.  
\EEE Note now that \AAA by the conformal invariance of the Dirichlet energy in two dimensions and \eqref{local_family_of_maps_sufficient_for_conformal_case}$(ii)$, \EEE 
\begin{equation}\label{grad_estimate_for_homotopies} 
\d2\big|\t\big[\left(u-\mathrm{id}_{\s}\right)\circ\phi\big]\big|^2= \d2\big| \t u-P_T\big|^2\leq \theta^2\,.
\end{equation}
By the change of variables formula, \AAA \eqref{local_family_of_maps_sufficient_for_conformal_case} again, \EEE and the Poincare inequality \AAA \eqref{Poincare} \EEE(since $\d2 (u-\mathrm{id}_{\s})=0$), we also have
\begin{equation}\label{function_estimate_for_homotopies} 
\d2\big|(u-\mathrm{id}_{\s})\circ\phi\big|^2\leq C_1(\mathcal{U}_0)\ \d2 \big|\t u-P_T \big|^2\leq  C_1(\mathcal{U}_0)\theta^2\,,
\end{equation}
where
\begin{equation}\label{positive_homotopy_constant}
C_1(\mathcal{U}_0)\sim\underset{\phi\in\mathcal{U}_0}{\mathrm{sup}}\ \underset{x\in\s}{\mathrm{inf}}\big|\t \phi(x)\big|^{-2}>0\,.
\end{equation}
The strict positivity of the constant $C_1(\mathcal{U}_0)$ \AAA in \eqref{positive_homotopy_constant} \EEE is ensured by the fact that we can take the neighbourhood $\mathcal{U}_0$ to be sufficiently small around $\mathrm{id}_{\s}$\AAA, which amounts to choosing $\sigma_0\in (0,1)$ sufficiently small (see \eqref{U_0_nbhd} and \eqref{Psi_local_diffeomorphism})\EEE. Hence, \AAA \eqref{Psi_u_Psi_homotopic}-\eqref{positive_homotopy_constant} imply that \EEE 
\begin{equation}\label{Psi_u_Psi_L_infty}
\|\Psi_u-\Psi\|_{L^{\infty}(\mathcal{U}_0)}\leq C(\mathcal{U}_0) \theta, \ \ \mathrm{where}\  C(\mathcal{U}_0):=\max\left\{\sqrt{6},\sqrt{C_1(\mathcal{U}_0)}\right\}>0\,.
\end{equation}
We can now continue as in \cite[Proposition 4.7]{faraco2005geometric}. For the sake of making the proof self-contained, we present the argument here, adapted to our setting.
	
\AAA Recalling the notation $D_\sigma$ introduced before \eqref{U_0_nbhd}, let us consider the family $(\Gamma_s)_{s\in[0,1]}$ of closed hypersurfaces in $Conf_+(\s)$, defined by
\begin{equation}\label{foliation}
 \Gamma_s:=\mathrm{exp}_{\mathrm{id}_{\s}}(\partial D_{s\sigma_0})\,,
\end{equation} 
\EEE so that $\Gamma_0=\{\mathrm{id}_{\s}\}$, and $\Gamma_1$ is the topological boundary of $\mathcal{U}_0$ inside the manifold $Conf_+(\s)$. Let us define
\begin{equation}\label{def_m_s}
m(s):=\mathrm{min}_{\phi\in\Gamma_s}|\Psi(\phi)| \ \ \mathrm{for\ every}\ s\in[0,1]\,.
\end{equation}
This is obviously a continuous function of $s$. Since $\Psi|_{\Gamma_0}\equiv 0$ and \AAA (by \eqref{Psi_local_diffeomorphism}) \EEE $\Psi|_{\mathcal{U}_0}$ is a homeomorphism onto its image, \AAA by \eqref{def_m_s} we deduce that \EEE
\begin{equation*}
m(s)>0\ \ \mathrm{\forall } s\in(0,1]  \mathrm{\ \ and\ } \lim_{s\to 0^+}m(s)=0\,.
\end{equation*}
Notice that $m(1)>0$ depends only on $\mathcal{U}_0$, \AAA or equivalently only on $\sigma_0$\EEE. We can thus choose $\theta\in (0,1)$ sufficiently small \AAA(depending on $\sigma_0$)\EEE, so that \AAA for the constant $C(\mathcal{U}_0)$ of \eqref{Psi_u_Psi_L_infty}, there holds \EEE
\begin{equation}\label{theta_small}
\big(C(\mathcal{U}_0)+1\big)\theta\leq \frac{m(1)}{2}\,,
\end{equation}
 and then define 
\begin{equation}\label{s_theta}
s_{\theta}:=\mathrm{inf}\left\{s\in[0,1):m(s)\geq\big(C(\mathcal{U}_0)+1\big)\theta\right\}.
\end{equation}
\AAA By \eqref{def_m_s} and \eqref{s_theta} \EEE it is clear that $\lim_{\theta\to 0^+}s_{\theta}=0$. Let us now consider \textit{the linear homotopy between $\Psi$ and $\Psi_u$}. For every $t\in [0,1]$ and $\phi\in \Gamma_{s_\theta}\subseteq\mathcal{U}_0\subseteq Conf_+(\s)$, \AAA by \eqref{def_m_s}, \eqref{s_theta} and \eqref{Psi_u_Psi_L_infty}, \EEE we have 
\begin{align}\label{linear_homotopy_non_zero}
\begin{split}
\big|\big((1-t)\Psi+t\Psi_u\big)(\phi)\big|\
&\ \geq |\Psi(\phi)|-t|(\Psi_u-\Psi)(\phi)| \geq \ \mathrm{min}_{\phi\in\Gamma_{s_\theta}}|\Psi(\phi)|-\|\Psi_u-\Psi\|_{L^{\infty}(\mathcal{U}_0)}\\
&\ \geq m_{s_\theta}-C(\mathcal{U}_0)\theta\ \geq \ \theta>0\,.
\end{split}
\end{align}
In particular, \AAA by \eqref{linear_homotopy_non_zero} we find that\EEE 
\begin{equation*}
\big((1-t)\Psi+t\Psi_u\big)(\phi)\neq 0 \ \  \forall t\in [0,1] \ \mathrm{and}\ \phi\in \Gamma_{s_\theta}\,.
\end{equation*}
Since the degree around $0$ remains constant through this linear homotopy, if $\mathcal{U}_{s_\theta}\subseteq \mathcal{U}_{0}$ is the open neighbourhood around the $\mathrm{id}_{\s}$ in $Conf_+(\s)$ with $\partial \mathcal{U}_{s_\theta}=\Gamma_{s_\theta}$\AAA, i.e., $\mathcal{U}_{s_\theta}:=\mathrm{exp}_{\mathrm{id}_{\s}}(D_{s_\theta\sigma_0})$ (recall \eqref{foliation}), \EEE then  
\begin{equation}\label{degree_is_still_1}
\mathrm{deg}({\Psi_u},0;\ \mathcal{U}_{s_\theta})=\mathrm{deg}(\Psi,0;\ \mathcal{U}_{s_\theta})=1\,.
\end{equation}
\AAA As a consequence of \eqref{degree_is_still_1}, \eqref{Psi_u_def} and the characterizations \eqref{Projection_on_H_n,1,2_characterization_full}, \eqref{Projection_on_H_n,1,2_characterization},  \EEE there exists a Möbius transformation $\phi\in\mathcal{U}_{s_\theta}\subseteq Conf_+(\s)$, so that
\begin{equation*}
\Psi_u(\phi)=0 \iff \Pi_{3,0}(u\circ \phi)=0\,.
\end{equation*}
In the same fashion as have estimated \AAA in \eqref{grad_estimate_for_homotopies}, and by \eqref{local_family_of_maps_sufficient_for_conformal_case}$(ii)$ again\EEE,
\begin{equation*}
\d2\left|\t(u\circ\phi)-P_T\right|^2 \leq 2\big(\theta^2+\d2 \left|\t\phi-P_T\right|^2 \big)\leq 2(\theta^2+C_{s_\theta})\,,
\end{equation*}
where
\begin{equation}\label{constant_Cstheta}
C_{s_\theta}:=\max_{\phi\in\overline{\mathcal{U}_{s_\theta}}}\|\phi-\mathrm{id}_{\s}\|^2_{W^{1,2}(\s)}\,.
\end{equation}	
\textit{Since all topologies in the finite dimensional manifold $Conf_+(\s)$ are equivalent} and $\lim_{\theta\to 0^+}s_{\theta}=0$, \AAA by \eqref{constant_Cstheta} \EEE we also have that $\lim_{\theta\to 0^+}C_{s_{\theta}}=0$. Hence, we can first take \AAA $\sigma_0\in (0,1)$ small enough \EEE and then $\theta\in(0,1)$ small enough \AAA(see also \eqref{theta_small})\EEE, so that after replacing $u\circ\phi$  with $(u\circ\phi-\d2 u\circ\phi)$ if necessary to fix its mean value to 0, we have (by using again the conformal invariance of the deficit) that $u\circ\phi\in \mathcal{B}_{\tilde \theta,\varepsilon_0}$, where $\tilde \theta:=\sqrt{2(\theta^2+C_{s_\theta})}>0$ is again small accordingly. \AAA This finishes the proof of \eqref{identifying_the_correct_Moebius}. \EEE\qedhere
\end{proof}

\subsection{Proof of Theorem \ref{main_thm_conformal_case}}\label{subsec: 4.4}
We can now combine all the previous steps to complete the proof of the main theorem for the conformal case in dimension 3.
\begin{proof}[\bf{Proof of Theorem \ref{main_thm_conformal_case}}]
\AAA In view of Corollary \ref{it_suffices_to_prove_the_local_coformal_case}, \EEE for $\theta\in(0,1)$ and $\varepsilon_0\in (0,1)$ that will be chosen sufficiently small in the end, let us consider a map $u\in \mathcal{B}_{\theta,\varepsilon_0}$. By first using Lemma \ref{fixingMobius}
and then Lemma \ref{fixing_center_scale}, 
 we can find a Möbius transformation $\phi\in Conf_+(\s)$ such that the map 
\begin{equation*}
\tilde u:=\frac{1}{\d2 \langle u\circ \phi,x\rangle }\left(u\circ\phi-\d2 u\circ\phi\right)
\end{equation*}
has all the desired properties, i.e.,
\begin{equation}\label{desired_properties}
\tilde u\in \mathcal{B}_{\theta,\varepsilon_0}, \  \mathrm{\ with\ \ \ } \mathcal{E}_2(\tilde u)=\mathcal{E}_2(u),\ \ \ \d2 \langle\tilde u,x\rangle=1 \ \mathrm{\ \ and\ \ } \  \Pi_{3,0}(\tilde u-\mathrm{id}_{\s})=0\,,
\end{equation}
where we have abused notation by not replacing $\theta$ with $\tilde{\theta}$. Setting $\tilde w:=\tilde u-\mathrm{id}_{\s}$, we can again expand the deficit around the $\mathrm{id}_{\s}$ and arrive at \eqref{reduction_to_quadratic_estimate}. This estimate, together with \eqref{Korn_type_inequality_for_Q_3} and the fact that $\tilde w\in H_3$ is such that $\Pi_{3,0}\tilde w=0$ \AAA(by \eqref{H_space} and \eqref{desired_properties}) \EEE yield
\begin{equation}\label{final_estimate_conformal}
\frac{1}{4}\d2 |\t \tilde w|^2\leq Q_3(\tilde w)\leq \mathcal{E}_2(u)+\beta\d2|\t \tilde w|^2\,.
\end{equation}
\AAA By Lemma \ref{quadratization_of_conformal_deficit}, \EEE we can then choose $\varepsilon_0\in (0,1)$ small enough and subsequently $\theta\in (0,1)$ small enough, so that $\beta\leq\frac{1}{8}$, \EEE in order to absorb the last term on the right hand side \AAA of \eqref{final_estimate_conformal} \EEE in the one on the left, and conclude for $\phi$ as above and $\lambda:=\d2 \langle u\circ \phi,x\rangle>0$\,.
\end{proof}

\section{Linear stability estimates for $Isom(\S)$ and $Conf(\S)$}\label{sec: 5 linear_stability}
As we have seen in the Section \ref{Section 4}, the key step in proving Theorem \ref{main_thm_conformal_case} consists in establishing the corresponding linear estimate, i.e., Theorem \ref{main_coercivity_estimate_3_d_conformal}. What we would like to present in this section, is how the eigenvalue decomposition of $H_n$ into eigenspaces of $A$, \AAA as provided by Theorem \ref{eigenvalue_decomposition}, \EEE which is valid in every dimension (actually even in dimension $n=2$), can be used to prove the analogous linear estimate for the quadratic form $Q_n$ introduced in \eqref{Q_n_definition}, also in dimensions $n\geq 4$. As explained in the Introduction, this quadratic form is associated with the combined conformal-isoperimetric deficit $\mathcal{E}_{n-1}$ defined in \eqref{def_of_combined_deficit} \AAA(cf.~Appendix \ref{sec:B} for the precise calculations). \EEE We discuss in detail this case first, since the proof of the corresponding estimate for the isometric case, i.e., Theorem \ref{coercivity_estimate_for_Q_isom_Q_isop}, is essentially the same and is discussed in Subsection \ref{Subsection 5.3.}.
\subsection{Proof of Theorem \ref{main_coercivity_estimate_conformal_general_dimension}}\label{Subsection 5.1.}
For $n=3$, Theorem \ref{main_coercivity_estimate_conformal_general_dimension} is precisely Theorem \ref{main_coercivity_estimate_3_d_conformal}, whose proof was given in Subsection \ref{Subsection 4.2}, and actually the optimal constant for the linear estimate \AAA \eqref{Korn_type_inequality_for_Q_3} \EEE was calculated. In the higher dimensional case $n\geq 4$, the quadratic form $Q_n$ \AAA in \eqref{Q_n_definition} \EEE has an extra $\ds (\mathrm{div}_{\S}w)^2$-term, and the study of its coercivity properties is slightly more complicated than before.

\begin{remark}\label{Remark_for_L}
An abstract way to obtain the estimate \AAA \eqref{Korn_type_inequality_for_Q_n} \EEE of Theorem \ref{main_coercivity_estimate_conformal_general_dimension} would be to \textit{identify the kernel of $Q_n$ in $H_n$} \AAA(see \eqref{H_space}) \EEE (and prove that it is exactly the subspace $H_{n,0}$) and then \textit{use a standard contradiction$\backslash$compactness argument} (that we describe in Subsection \ref{Subsection 5.3.} for the corresponding result in the isometric case). Notice that \textit{$w\in H_n$ lies in the kernel of the nonnegative quadratic form $Q_n$ iff} $Q_n(v,w)=0\ \ \forall v\in H_n$, i.e., iff  $w\in H_n$ lies in the kernel of the associated Euler-Lagrange operator 
\begin{equation}\label{L_operator}
\hspace{-0.5em}\mathcal{L}(w):=-\frac{1}{n-1}\Delta_{\S}w-\frac{n-3}{(n-1)^2}\big(\t\mathrm{div}_{\S}w-(n-1)(\mathrm{div}_{\S}w)x\big)-\big((\mathrm{div}_{\S}w)x-\sum_{j=1}^n x_j\t w^j\big)\,,
\end{equation}  	
in the sense of distributions. When $n=3$, the second term \AAA in \eqref{L_operator} \EEE is dropping out, and since $\mathcal{L}$ leaves the subspaces $(H_{3,k,\mathrm{sol}})_{k\geq 1},(H_{3,k,\mathrm{sol}}^{\bot})_{k\geq 1}$ invariant in this case, the Euler-Lagrange equation can be solved explicitely, showing that $\mathrm{ker}\mathcal{L}=H_{3,0}$ \AAA(as Theorem \ref{main_coercivity_estimate_3_d_conformal} describes quantitatively)\EEE. Although in slightly hidden form, this was essentially the point of Lemma \ref{A_invariances} and the subsequent results in Subsection \ref{Subsection 4.2}.

In higher dimensions $n\geq 4$, since the operator $\big(\t\mathrm{div}_{\S}w-(n-1)(\mathrm{div}_{\S}w)x\big)$ neither commutes with $A$ \AAA (\eqref{Aoperator}), \EEE nor leaves the subspaces $(H_{n,k,\mathrm{sol}})_{k\geq 1},(H_{n,k,\mathrm{sol}}^{\bot})_{k\geq 1}$ invariant, it is not clear if there is a straightforward argument to solve the equation $\mathcal{L}(w)=0$ explicitely, and show that indeed $\mathrm{ker}\mathcal{L}=H_{n,0}$ in this case as well. In particular, \AAA(comparing with \eqref{Q_3_eigenspace_decomposition}), \EEE \textit{mixed terms of special type are expected to be present in the Fourier decomposition of $Q_n$ into the eigenspaces of $A$}\AAA, which were identified in Theorem \ref{eigenvalue_decomposition}\EEE. Nevertheless, it will turn out that we can still use the latter to show that \textit{the presence of the mixed divergence-terms is harmless}, i.e., it does not produce any further zeros (other than $H_{n,0}$) in $Q_n$. Simultaneously, we obtain the desired coercivity estimate \AAA \eqref{Korn_type_inequality_for_Q_n} \EEE (with an \textit{explicit lower bound for the optimal constant}) by examining how $Q_n$ behaves in each one of the eigenspaces $(H_{n,k,i})_{k\geq, i=1,2,3}$ of $A$ separately. 
\end{remark}
\vspace{-0.5em}
Following the notation we had in Subsection \ref{Subsection 4.2}, we first present two auxiliary lemmata that entail most of the essential ingredients for the proof \AAA of Theorem \ref{main_coercivity_estimate_conformal_general_dimension} \EEE also in dimensions $n\geq 4$. 

\begin{lemma}\label{div_Q_n_in_the_eigenspaces}
For $n\geq 3$ and $k\geq 1$, let us denote by $\lambda_{n,k}:=k(k+n-2)$ the eigenvalues of $-\Delta_{\S}$ \AAA$\rm{(}$see \eqref{properties}$\rm{)}$ \EEE and let $i=1,2,3$. For every $w\in H_{n,k,i}$ \AAA$\rm{(}$as in Theorem \ref{eigenvalue_decomposition}$\rm{)}$, \EEE  we have 
\begin{align}\label{Q_v_n_div_n_Q_n_in_eigenspaces}
\begin{split}
Q_{V_n}(w)&=c_{n,k,i}\ds |\t w|^2, \quad \text{where} \ \   c_{n,k,i}:=\frac{n\sigma_{n,k,i}}{2\lambda_{n,k}}\,,\\
\ds (\mathrm{div}_{\S}w)^2&=\alpha_{n,k,i}\ds |\t w|^2,\quad \text{where} \ \  \alpha_{n,k,i}:=\frac{\sigma_{n,k,i}^2(2\lambda_{n,k}c_{n,k,i}-n)}{n\lambda_{n,k}(2\sigma_{n,k,i}-n)}\,,\\
Q_n(w)&=C_{n,k,i}\ds |\t w|^2,\quad \text{where} \ \    C_{n,k,i}:=\frac{n}{2(n-1)}+\frac{n(n-3)}{2(n-1)^2}\alpha_{n,k,i}-c_{n,k,i}\,. 
\end{split}
\end{align}
\end{lemma}
\begin{proof}
The first identity is immediate \AAA from \eqref{quadratic_form_of_volume}, \eqref{Aoperator}, Theorem \ref{eigenvalue_decomposition} and \eqref{sharmonics}\EEE. For the second one, after integration by parts we see that the quadratic form $Q_{V_n}$ can be equivalently rewritten as
\begin{equation}\label{Q_V_n_alternative}
Q_{V_n}(w)=\frac{n}{2}\ \ds \Big(2\mathrm{div}_{\S} w\langle w,x\rangle-n\langle w,x\rangle^2+|w|^2\Big)\,.
\end{equation}
Together with the first identity and \eqref{Equation_for_divs}, \AAA\eqref{Q_V_n_alternative} \EEE yields the desired identity, and then the one for $Q_n$ follows immediately \AAA by its definition in \eqref{Q_n_definition} and the two previous identities that we just checked.\EEE\qedhere
\end{proof}
Since they will again play an important role in the sequel \AAA(as it was the case for the constants in \eqref{constants1})\EEE, let us list below the precise values of the previous constants \AAA appearing in \eqref{Q_v_n_div_n_Q_n_in_eigenspaces}\EEE. The last set of constants in the following table is considered for $k\geq 2$, because in any case $\Pi_{H_{n,1,3}}w=0$ for every $w\in H_n$, \AAA as we have justified in Remark \ref{triviality_of_H_1_3}. \EEE
\begin{align}\label{constants2}
\begin{split}
\mathrm{For \ } k\geq 1;\ \ \ &\left\{\begin{array}{lr}
c_{n,k,1}=\frac{-n}{2(k+n-2)}\\[5pt]
\alpha_{n,k,1}=\frac{k(k+1)}{(k+n-2)(2k+n)}\\[5pt]
C_{n,k,1}=\frac{n}{2}\Big(\frac{1}{n-1}+\frac{1}{k+n-2}+\frac{(n-3)k(k+1)}{(n-1)^2(k+n-2)(2k+n)}\Big)
\end{array}\right\}\,, \\[7pt]
\mathrm{For \ } k\geq 1;\ \ \ &\left\{\begin{array}{lr}
c_{n,k,2}=\frac{n}{2k(k+n-2)}\\[5pt]
\alpha_{n,k,2}=0\\[5pt]
C_{n,k,2}=\frac{n}{2}\frac{(k-1)(k+n-1)}{(n-1)k(k+n-2)}
\end{array}\right\}\,,\ \ \ \ \ \ \ \ \ \ \ \ \ \ \ \ \ \ \ \ \ \ \ \ \ \ \ \ \  \\[7pt]
\mathrm{For \  } k\geq 2;\ \ \ &\left\{\begin{array}{lr}
c_{n,k,3}=\frac{n}{2k}\\[5pt]
\alpha_{n,k,3}=\frac{(k+n-2)(k+n-3)}{k(2k+n-4)}\\[5pt]
C_{n,k,3}=\frac{n(k-2)\big((3n-5)k+(n^2-6n+7)\big)}{2(n-1)^2k(2k+n-4)}
\end{array}\right\}\,.\ \  \ \ \ \ \ \ \ \ \ \ \ \ \ \  
\end{split}
\end{align}
The next ingredient we need is the following.
\pagebreak
\begin{lemma}\label{bd_lemma}
The following statements hold true.	
\begin{itemize} 
\item[$(i)$] Let $n\geq 3$. For every $k,l\geq 1$ and $i,j=1,2,3$ with $(k,i)\neq (l,j)$, the subspaces $H_{n,k,i}$ and $H_{n,l,j}$ \AAA$\rm{(}$introduced in Theorem \ref{eigenvalue_decomposition}$\rm{)}$ \EEE are $Q_{V_n}$- and $\widetilde{Q}_n$-orthogonal, where 
\begin{equation}\label{tilde_Q_n}
\widetilde{Q}_n(w):=\frac{n}{2(n-1)}\ds|\t w|^2-Q_{V_n}(w)\,,
\end{equation}
i.e., for every $w_{n,k,i}\in H_{n,k,i}$ and $w_{n,l,j}\in H_{n,l,j}$,
\begin{equation}\label{H_k_is_Q_voln_tilde_Q_n_orth}
Q_{V_n}(w_{n,k,i},w_{n,l,j})=0 \mathrm{\ \ \ and \ \ \ } \widetilde{Q}_n(w_{n,k,i},w_{n,l,j})=0\,.
\end{equation} 
\item[$(ii)$] Let $n\geq 3$, $w\in H_n$ written in Fourier series as $w=\underset{\underset{(k,i)\neq(1,3)}{(k,i)\in\mathbb{N}^*\times\{1,2,3\}}}{\sum}w_{n,k,i}, \mathrm{\ where  \ } w_{n,k,i}\in H_{n,k,i}.$ 
Then,
\begin{flalign}\label{Q_n_Fourier_expansion_identity}
\begin{split}
\hspace{-3.5em}Q_{n}(w)=&\sum_{\underset{(k,i)\neq(1,3)}{(k,i)\in \mathbb{N}^*\times\{1,2,3\}}}Q_n( w_{n,k,i})+\frac{1}{2}\frac{n(n-3)}{(n-1)^2}\sum_{k\geq 1}\ds \mathrm{div}_{\S}w_{n,k,1}\mathrm{div}_{\S}w_{n,k+2,3} \\
& +\frac{1}{2}\frac{n(n-3)}{(n-1)^2}\sum_{k\geq 3}\ds\mathrm{div}_{\S}w_{n,k,3}\mathrm{div}_{\S}w_{n,k+2,1}\,. 
\end{split}
\end{flalign}
\end{itemize} 
\end{lemma}
\vspace{-2em}
\AAA \begin{remark}\label{interesting_k_geq_3} 
Before giving the proof of Lemma \ref{bd_lemma}, let us point out an interesting feature in formula \eqref{Q_n_Fourier_expansion_identity}, which will be useful in the proof of Theorem \ref{main_coercivity_estimate_conformal_general_dimension}, \textit{namely that the summation in the last term of the expression starts from $k=3$}. \EEE The reason for this is that in any case $w_{n,1,3}\equiv 0$ whenever $w\in H_n$ \AAA(see Remark \ref{triviality_of_H_1_3})\EEE, but one can also check that
\begin{equation}\label{second_mixed_term_vanishes}
\ds \mathrm{div}_{\S}w_{n,2,3}\mathrm{div}_{\S}w_{n,4,1}=0\,.
\end{equation}
In order to prove \AAA \eqref{second_mixed_term_vanishes}, recall that by definition \AAA(see \eqref{tildeHksol_definition}, \eqref{Hksol_definition} and \eqref{H_n_k_sol_eigenvalue_decomposition}), \EEE $\mathrm{div}(w_{n,4,1})_h\equiv0$ in $\overline{B_1}$, and \AAA by also using \eqref{Equation_for_divs} for $(k,i)=(2,3),(4,1)$ and the divergence theorem, \EEE we obtain 
\begin{equation}\label{second_mixed_term_vanishes_alternative}
\hspace{-0.5em}\ds \hspace{-1em}\mathrm{div}_{\S}w_{n,2,3}\mathrm{div}_{\S}w_{n,4,1}=-4n\ds\hspace{-1em}\langle \langle (w_{n,2,3})_h,x\rangle (w_{n,4,1})_h,x\rangle=-4\dashint_{B_1}\hspace{-0.5em}\Big\langle \nabla\langle(w_{n,2,3})_h,x\rangle, (w_{n,4,1})_h\Big\rangle.
\end{equation}
To justify that the last integral on \AAA the right hand side of \eqref{second_mixed_term_vanishes_alternative} \EEE is zero, observe that
\begin{equation}
\dashint_{B_1}\Big\langle \nabla\langle(w_{n,2,3})_h,x\rangle, (w_{n,4,1})_h\Big\rangle
= \sum_{i=1}^n\dashint_{B_1}(w_{n,4,1})_h^i\big\langle\partial_i(w_{n,2,3})_h,x\big\rangle\,,
\end{equation}
because $\dashint_{B_1}\big\langle (w_{n,4,1})_h, (w_{n,2,3})_h\big\rangle=0$. Moreover, we observe that for every $i=1,\dots,n$,
\begin{equation}\label{Poisson_equation_in_B_1}
\Delta\big(\big\langle \partial_i(w_{n,2,3})_h,x\big\rangle\big)=2\partial_i\mathrm{div}(w_{n,2,3})_h \mathrm{\ \ in\ } B_1\,.
\end{equation}
Since $(w_{n,2,3})_h$ is an $\R^n$-valued 2nd-order homogeneous harmonic polynomial, $\partial_i\mathrm{div}(w_{n,2,3})_h$ is simply a constant. \AAA But then, \eqref{Poisson_equation_in_B_1} implies that \EEE the function $$\big\langle \partial_i(w_{n,2,3})_h,x\big\rangle-\frac{\partial_i\mathrm{div}(w_{n,2,3})_h|x|^2}{n}$$
is a homogeneous harmonic polynomial of degree 2, hence also $L^2$-orthogonal to $(w_{n,4,1})_h^i$. Thus,  
\begin{align}\label{last_integral}
\dashint_{B_1} (w_{n,4,1})_h^i\big\langle\partial_i(w_{n,2,3})_h,x\big\rangle=\frac{\partial_i\mathrm{div}(w_{n,2,3})_h}{n}\dashint_{B_1} (w_{n,4,1})_h^i|x|^2=\frac{\partial_i\mathrm{div}(w_{n,2,3})_h}{n+6}\ds (w_{n,4,1})^i=0\,,
\end{align}
where we used the fact that the function $(w_{n,4,1})_h^i|x|^2$ is $6$-homogeneous, so that we can write its integral over $B_1$ as an integral over $\S$, up to the correct multiplicative constant. The last integral \AAA in \eqref{last_integral} \EEE is of course zero for every nontrivial spherical harmonic. \AAA Therefore, \eqref{second_mixed_term_vanishes_alternative}-\eqref{last_integral} imply \eqref{second_mixed_term_vanishes}. \EEE Note that the previous argument relies on the fact that $\partial_i\mathrm{div}(w_{n,2,3})_h$ is constant, and of course cannot be implemented for the mixed terms of higher order.
\end{remark}
\EEE
\begin{proof}[Proof of Lemma \ref{bd_lemma}]
As in Lemma \ref{Qdiag_in_dim_3}, part $(i)$ is an immediate consequence of the fact that the subspaces $(H_{n,k,i})_{k\geq 1, i=1,2,3}$ are mutually orthogonal in $W^{1,2}(\S;\R^n)$. For part $(ii)$, having established that the form $\widetilde{Q}_n$ \AAA in \eqref{tilde_Q_n} \EEE splits completely in the eigenspaces $(H_{n,k,i})_{(k,i)\in\mathbb{N^*}\times\{1,2,3\}\setminus(1,3)}$, what remains to be checked is that whenever $w_{n,k,i}\in H_{n,k,i}$ and $w_{n,l,j}\in H_{n,l,j}$, there holds
\begin{equation}\label{vanishing_of_mixed_terms}
\ds\mathrm{div}_{\S}w_{n,k,i}\mathrm{div}_{\S}w_{n,l,j}=0
\end{equation}
for all pairs $\{(k,i),(l,j)\}\in\mathbb{N^*}\times\{1,2,3\}\setminus(1,3)$ with $(k,i)\neq(l,j)$, \textit{except those of the form $\{(k,1),(k+2,3)\}$ and $\{(k,3),(k+2,1)\}$}.
	
This can be checked again using the different equivalent formulas for $Q_{V_n}$. 
Since $\mathrm{div}_{\S}w\equiv 0$ whenever $w\in H_{n,k,2}$ \AAA(see \eqref{Equation_for_divs} and \eqref{normal_part_0})\EEE, we may suppose without loss of generality that $i,j\in \{1,3\}$, and then \AAA by \eqref{Q_V_n_alternative} (written now in its bilinear expression),\EEE 
\begin{align}\label{how_to_use_different_formulas}
\begin{split}
Q_{V_n}(w_{n,k,i}, w_{n,l,j})&= \frac{n}{2}\ds \mathrm{div}_{\S}w_{n,k,i}\langle w_{n,l,j},x\rangle+\frac{n}{2}\ds\mathrm{div}_{\S}w_{n,l,j}\langle w_{n,k,i},x\rangle\\[3pt]
&\ -\frac{n^2}{2}\ds\langle w_{n,k,i},x\rangle\langle
w_{n,l,j},x\rangle+\frac{n}{2}\ds\langle w_{n,k,i},w_{n,l,j}\rangle\,.
\end{split}
\end{align} 
Actually, as we verified in part $(i)$,
$Q_{V_n}(w_{n,k,i}, w_{n,l,j})=0$,  
and in view of \eqref{Equation_for_divs}, \AAA \eqref{how_to_use_different_formulas} yields\EEE 
\begin{equation*}
\left(\sigma_{n,k,i}+\sigma_{n,l,j}-n\right) \ds\mathrm{div}_{\S}w_{n,k,i}\mathrm{div}_{\S}w_{n,l,j}=0\,,
\end{equation*}
i.e., \eqref{vanishing_of_mixed_terms} holds, unless the pairs $(k,i)\neq(l,j)$ are such that $\sigma_{n,k,i}+\sigma_{n,l,j}=n$. In this respect,
\begin{align*}
(i)\quad & \text{If}\ \  i=j=1, \quad \quad \sigma_{n,k,i}+\sigma_{n,l,j}=-k-l<0<n\,,\\
(ii)\quad & \text{If}\ \ i=j=3, \quad \quad \sigma_{n,k,i}+\sigma_{n,l,j}=k+l+2n-4\geq 2n-2>n\,, \\
(iii)\quad & \text{If}\ \  i=1, j=3, \quad \sigma_{n,k,i}+\sigma_{n,l,j}=n\iff -k+l+n-2=n\iff l=k+2\,,\\
(iv)\quad & \text{If}\ \  i=3, j=1,\quad \sigma_{n,k,i}+\sigma_{n,l,j}=n\iff k+n-2-l=n\iff k=l+2\,,
\end{align*}
which proves the desired claim and then the formula \eqref{Q_n_Fourier_expansion_identity} for $Q_n$ follows by the bilinearity of the expression, \AAA and the observation we made in Remark \ref{interesting_k_geq_3}\EEE. \qedhere
\end{proof}

We now have all the necessary ingredients to prove Theorem \ref{main_coercivity_estimate_conformal_general_dimension}. As a preliminary remark, let us note that by taking a closer look at the values of the constants \AAA \eqref{constants2} \EEE of Lemma \ref{div_Q_n_in_the_eigenspaces}, we see that in \AAA the case $n\geq 4$ \EEE one cannot merely neglect the $\ds(\mathrm{div}_{\S}w)^2$-term \AAA and argue exactly as in the proof of Theorem \ref{main_coercivity_estimate_3_d_conformal} ($n=3$)\EEE, something that was also indicated in Remark \ref{Remark_for_L}. Indeed, although the quadratic form $\widetilde{Q}_n(w)$ \AAA in \eqref{tilde_Q_n} \EEE is splitting among the eigenspaces $(H_{n,k,i})_{k\geq 1, i=1,2,3}$, for $n\geq 4$ it does not have a sign. Actually, $\widetilde{Q}_n(w)$ is negative in $H_{n,k,3}$ for every $k=2,\dots,n-2$, zero in $H_{n,1,2}, H_{n,n-1,3}$ and strictly positive in each one of the other
eigenspaces. On the other hand, we see that the quadratic form \textit{$Q_n$ vanishes again in the space $H_{n,0}:=H_{n,1,2}\oplus H_{n,2,3}$}.

\begin{proof}[Proof of Theorem \ref{main_coercivity_estimate_conformal_general_dimension} $\rm{(}$$n\geq 4\rm{)}$.]
For every $k\geq 1$, \AAA we use the Cauchy-Schwartz inequality, the second identity in \eqref{Q_v_n_div_n_Q_n_in_eigenspaces}, the inequality 
$$|ab|\leq \frac{\varepsilon}{2}a^2+\frac{1}{2\varepsilon}b^2,\quad a,b\in \R,\ \varepsilon>0\,,$$
with weight $\varepsilon_{n,k}=\frac{k+n}{k}>0$, and \eqref{constants2}, \EEE in order to estimate
\begin{flalign}\label{mixed_terms_first_estimate}
\begin{split}
\hspace{-1em}\left|\ds \mathrm{div}_{\S}w_{n,k,1}\mathrm{div}_{\S}w_{n,k+2,3}\right|&\leq\left(\alpha_{n,k,1}\ds |\t w_{n,k,1}|^2\right)^{\frac{1}{2}}\left(\alpha_{n,k+2,3}\ds |\t w_{n,k+2,3}|^2\right)^{\frac{1}{2}} \\[2pt]
&\leq\frac{\alpha_{n,k,1}\varepsilon_{n,k}}{2}\ds |\t w_{n,k,1}|^2+\frac{\alpha_{n,k+2,3}}{2\varepsilon_{n,k}}\ds|\t w_{n,k+2,3}|^2 \\[5pt]
&=\frac{(k+1)(k+n)}{2(k+n-2)(2k+n)}\ds|\t w_{n,k,1}|^2\\[2pt]
&\ \ \ +\frac{k(k+n-1)}{2(k+2)(2k+n)}\ds|\t w_{n,k+2,3}|^2\,.
\end{split}
\end{flalign}
For the last summand in \eqref{Q_n_Fourier_expansion_identity}, we shift the
summation index to start from $k=1$ and estimate as before,
\begin{flalign}\label{mixed_terms_second_estimate}
\begin{split}
\hspace{-2em}
\left|\ds \mathrm{div}_{\S}w_{n,k+2,3}\mathrm{div}_{\S}w_{n,k+4,1}\right|
&\leq\frac{\alpha_{n,k+2,3}}{2\varepsilon_{n,k}}\ds |\t w_{n,k+2,3}|^2+\frac{\alpha_{n,k+4,1}\varepsilon_{n,k}}{2}\ds|\t w_{n,k+4,1}|^2\\[2pt]
&=\frac{k(k+n-1)}{2(k+2)(2k+n)}\ds|\t w_{n,k+2,3}|^2\\[2pt]
&\ \ \ +\frac{(k+4)(k+5)(k+n)}{2k(k+n+2)(2k+n+8)}\ds|\t w_{n,k+4,1}|^2\,.
\end{split}
\end{flalign}
\AAA Note that the choice of the weights $\varepsilon_{n,k}$ was such that the second term in the last line of \eqref{mixed_terms_first_estimate} coincides with the first term in the last equality in \eqref{mixed_terms_second_estimate}. \EEE The series appearing in \AAA \eqref{Q_n_Fourier_expansion_identity} \EEE are all absolutely summable, $Q_n(w_{n,1,2})=Q_n(w_{n,2,3})=0$,  and \AAA using \eqref{mixed_terms_first_estimate} and \eqref{mixed_terms_second_estimate}, \EEE we can estimate the form $Q_n$ from below by,
\begin{align}\label{estimate_from_below_Q_n}
\begin{split}
Q_n(w)&\geq \sum_{\underset{(k,i)\neq(1,2),(1,3),(2,3)}{(k,i)\in \mathbb{N}^*\times\{1,2,3\}}}Q_n(w_{n,k,i})-\frac{1}{4}\frac{n(n-3)}{(n-1)^2}\sum_{k\geq 1}\frac{(k+1)(k+n)}{(k+n-2)(2k+n)}\ds|\t w_{n,k,1}|^2\\[2pt]
&\ -\frac{1}{4}\frac{n(n-3)}{(n-1)^2}\sum_{k\geq 1}\frac{(k+4)(k+5)(k+n)}{k(k+n+2)(2k+n+8)}\ds|\t w_{n,k+4,1}|^2\\[2pt]
&\ -\frac{1}{2}\frac{n(n-3)}{(n-1)^2}\sum_{k\geq 1}\frac{k(k+n-1)}{(k+2)(2k+n)}\ds|\t w_{n,k+2,3}|^2\,.
\end{split}
\end{align} 
After rearranging terms \AAA in \eqref{estimate_from_below_Q_n}\EEE, we arrive at the estimate
\begin{equation}\label{Q_n_new_estimate}
Q_n(w)\geq \sum_{k\geq 1} \tilde{C}_{n,k,1}\ds|\t w_{n,k,1}|^2+\sum_{k\geq 2}\tilde{C}_{n,k,2}\ds|\t w_{n,k,2}|^2+\sum_{k\geq 3}\tilde{C}_{n,k,3}\ds|\t w_{n,k,3}|^2\,,
\end{equation}
where the new constants are defined as
\begin{equation}\label{new_constants}
\begin{cases}
\tilde{C}_{n,k,1}:=\left[C_{n,k,1}-\frac{n(n-3)(k+1)}{4(n-1)^2(k+n-2)(2k+n)}\left(\frac{k(k+n-4)\chi_{m\geq 5}(k)}{k-4}+(k+n)\right)\right],\ \ \  k\geq 1\\
\tilde{C}_{n,k,2}:=\ C_{n,k,2}\,,\ \ \  k\geq 2,\\
\tilde{C}_{n,k,3}:=\left[C_{n,k,3}-\frac{1}{2}\frac{n(n-3)}{(n-1)^2}\frac{(k-2)(k+n-3)}{k(2k+n-4)}\right],\ \ k\geq 3\,.
\end{cases}
\end{equation}
By elementary algebraic calculations that we omit here for the sake of brevity, one can verify that
\begin{equation}\label{minimum_of_constants_first}
\underset{k\geq 1}{\mathrm{min}}\ \tilde C_{n,k,1}=:C_{n,1}>0\,,\quad \underset{k\geq 2}{\mathrm{min}}\ \tilde C_{n,k,2}=:C_{n,2}>0\,,\quad 
\underset{k\geq 3}{\mathrm{min}}\ \tilde C_{n,k,3}=:C_{n,3}>0\,.
\end{equation}
\AAA Indeed, for $i=1,2,3$, one can write $\tilde C_{n,k,i}:=\frac{p_{n,i}(k)}{q_{n,i}(k)}$, where $p_{n,i}(k), q_{n,i}(k)$ are explicit polynomials of the same degree in $k$ (of degree 3 when $i=1$, and degree 2 when $i=2,3$) and verify algebraically that 
\begin{equation}\label{positive_constants}
\tilde C_{n,k,1}>0\ \ \text{for}\ k\geq 1, \quad \tilde C_{n,k,2}>0\ \ \text{for}\ k\geq 2,\ \text{and\ } \tilde C_{n,k,3}>0\ \ \text{for } k\geq 3\,.
\end{equation}
Moreover, directly from \eqref{constants2} and \eqref{new_constants}, one sees that
\begin{equation}\label{limits_at_infty}
\lim_{k\to \infty}\tilde C_{n,k,1}= \lim_{k\to \infty}\tilde C_{n,k,2}=\lim_{k\to \infty}\tilde C_{n,k,3}=\frac{n}{2(n-1)}>0\,,
\end{equation}
and then \eqref{minimum_of_constants_first} follows from \eqref{positive_constants} and \eqref{limits_at_infty}. The precise values of the constants $C_{n,i}$ could be calculated as well, by examining the monotonicity with respect to $k$ of the sequences $(\tilde C_{n,k,i})$ respectively (or alternatively of the corresponding rational functions with respect to the continuum variable). \EEE 
Labelling 
\begin{equation}\label{minimum_of_constants_second}
C_n:=\mathrm{min}\{C_{n,1},C_{n,2},C_{n,3}\}>0\,,
\end{equation}
\AAA we can use \eqref{Q_n_new_estimate}-\eqref{minimum_of_constants_second}, \EEE to further estimate from below,
\begin{align*}
\begin{split}
Q_n(w)&\geq C_n\left(\sum_{k\geq 1}\ds|\t w_{n,k,1}|^2+\sum_{k\geq 2}\ds|\t w_{n,k,2}|^2+\sum_{k\geq 3}\ds|\t w_{n,k,3}|^2\right)\\[3pt]
&=C_n\ds \Big|\t w-\big(\t w_{n,1,2}+\t w_{n,2,3}\big)\Big|^2,
\end{split}
\end{align*}
which finishes the proof \AAA of \eqref{Korn_type_inequality_for_Q_n}\EEE. \textit{For $n\geq 4$, the constant $C_n$ \AAA in \eqref{minimum_of_constants_second} \EEE provides an explicit lower bound for the value of the optimal constant for which \eqref{Korn_type_inequality_for_Q_n} holds.}
\end{proof}
\AAA\begin{remark}\label{from_linear_to_nonlinear_n_bigger_than_3}
As we mentioned in the Introduction, the linear estimate  \eqref{Korn_type_inequality_for_Q_n} would then imply Corollary  \ref{fake_nonlinear_estimate_higher_dim}, by following exactly the same procedure as in Section \ref{Section 4}. In this setting, one can of course reduce to showing
\ref{fake_nonlinear_estimate_conformal_case_high_dimensions} for maps $u\in W^{1,\infty}(\S;\R^n)$ which satisfy the stronger conditions
$$\ds u=0, \quad \|\t u-P_T\|_{L^{\infty}}\leq \theta\ll 1,\quad \mathcal{E}_{n-1}(u)\leq \varepsilon_0\ll 1\,.$$ 
With the aid of Lemma \ref{expansions_n} and Remark \ref{higher_order_terms_negligible}, the reader can easily verify that the rest of the arguments in the proof of Theorem \ref{main_thm_conformal_case} carry over also in dimensions $n\geq 4$, essentially unchanged. \EEE
\end{remark} \EEE
\subsection{Comparison of Theorem \ref{main_coercivity_estimate_conformal_general_dimension} with Korn's inequality for the trace-free symmetrized gradient operator in the bulk}\label{Subsection 5.2.}

Let $n\geq 3$ and $U$ be an open bounded Lipschitz domain of $\R^n$. If $v\in W^{1,n}(U;\R^n)$, then
\begin{equation*}
\int_U\left(\frac{|\nabla v|^2}{n}\right)^{\frac{n}{2}}\geq \int_U \mathrm{det}\nabla v\,,
\end{equation*}
with equality iff $\nabla v\in \R_+SO(n)$ for a.e. $x \in U$, i.e., according to Liouville's theorem, iff $v$ is the restriction of an orientation-preserving Möbius transformation on $U$. Setting again $v:=w+\mathrm{id}|_U$, and expanding the deficit, one obtains formally
\begin{align}\label{bulk_quadratic_form_related_to_conformal_deficit}
&\hspace{-1.2em} \int_{U}\Bigg(\Big(\frac{|\nabla v|^2}{n}\Big)^{\frac{n}{2}}-\mathrm{det}\nabla v\Bigg)= \int_{U}\left|(\nabla w)_{\mathrm{sym}}-\frac{\mathrm{div}w}{n}I_n\right|^2+\int_{U}\mathcal{O}\left(|\nabla w|^3\right)\,.
\end{align}

Another well known fact regarding the connection of the quadratic form in the right hand side of \AAA \eqref{bulk_quadratic_form_related_to_conformal_deficit} \EEE to the geometry of $\R_+SO(n)$ is the following (see \cite[Chapters 2 and 3]{reshetnyak2013stability},  or \cite{faraco2005geometric} for more details). If $T\R_+SO(n)$ stands for the tangent space to the conformal group $\R_+SO(n)$ at $I_n$, it is immediate that
\begin{equation*}
A\in T\R_+SO(n)\iff A_{\mathrm{sym}}=\frac{\mathrm{Tr}A}{n}I_n\,,
\end{equation*}
so that the function $A\mapsto d(A):=\left| A_{\mathrm{sym}}-\frac{\mathrm{Tr}A}{n}I_n\right|$ is equivalent to the distance of $A$ from $T\R_+SO(n)$. Therefore, the linear subspace 
\begin{equation*}
\Sigma_n:=\Big\{u\in W^{1,2}(\R^n;\R^n)\colon (\nabla u)_{\mathrm{sym}}=\frac{\mathrm{div}u}{n}I_n \Big\}
\end{equation*} 
can be viewed as the Lie algebra of the Möbius group of $\overline{\R^n}$. If $\Pi_{\Sigma_n}: W^{1,2}(U;\R^n)\mapsto \Sigma_n$ is the $W^{1,2}$-projection on this finite-dimensional subspace, the following variant of \textit{Korn's inequality for the trace-free part of the symmetrized gradient} is known to hold. 

\begin{theorem}\textbf{(Y.G. Reshetnyak, cf. \cite[Theorem 3.3, Chapter 3]{reshetnyak2013stability})}\label{Korn_type_inequality_conformal_setting}\
Let $n\geq 3$ and $U$ be an open bounded Lipschitz domain of $\R^n$ that is starshaped with respect to a ball. There exists a constant $C:=C(n,U)>0$ such that for every $w\in W^{1,2}(U;\R^n)$, 
\begin{equation*}
\big\|\nabla w-\nabla (\Pi_{\Sigma_n}w)\big\|_{L^2(U)}\leq C\Big\|(\nabla w)_{\mathrm{sym}}-\frac{\mathrm{div}w}{n}I_n \Big\|_{L^2(U)}\,.
\end{equation*}
\end{theorem}
In a certain sense, Theorem \ref{main_coercivity_estimate_conformal_general_dimension} is the analogue of Theorem \ref{Korn_type_inequality_conformal_setting} for maps from $\S$ to $\R^n$. In particular, as an upshot of it we have encountered the following fact. \textit{Although the kernels of the nonnegative quadratic forms arising as the second derivatives of the conformal deficit $\big[D_{n-1}(u)\big]^{\frac{n}{n-1}}-\big[P_{n-1}(u)\big]^{\frac{n}{n-1}}$ and the isoperimetric deficit $\big[P_{n-1}(u)\big]^{\frac{n}{n-1}}-V_n(u)$ at the $\mathrm{id_{\S}}$ are both infinite-dimensional, the intersection of the two kernels is finite-dimensional and actually isomorphic to the Lie algebra of infinitesimal Möbius transformations of $\S$}. 

Indeed, by the calculations that we exhibit in \AAA Lemma \ref{expansions_n} in Appendix \ref{sec:B}, \EEE we have that
\begin{equation}\label{Q_n_conformal_sphere}
Q_{n,\mathrm{conf}}(w):=\frac{n}{n-1}\ds\left|(P_T^t\t w)_{\mathrm{sym}}-\frac{\mathrm{div}_{\S}w}{n-1}I_x \right|^2\,,
\end{equation}	
and
\begin{equation}\label{Q_isop_sphere}
\hspace{-0.8em}Q_{n,\mathrm{isop}}(w):=\frac{1}{2}\frac{n}{n-1}\left[\ds|\t w|^2+(\mathrm{div}_{\S}w)^2-2\left|(P_T^t\t w)_{\mathrm{sym}}\right|^2\right]-Q_{V_n}(w)\,.
\end{equation} 	

The nonnegative quadratic form $Q_{n,\mathrm{conf}}$ \AAA in \eqref{Q_n_conformal_sphere} \EEE corresponds to the one associated to the conformal deficit in \eqref{bulk_quadratic_form_related_to_conformal_deficit} for maps defined in the bulk, but the kernel of $Q_{n,\mathrm{conf}}$ is infinite-dimensional. Intuitively, the underlying geometric reason behind this, is the abundance of $C^2$ conformal maps from $\S$ into $\R^n$. Actually, we observe that for every $\phi\in W^{1,2}(\S;\R)$ with $\ds \phi=0$ and  $\ds \phi(x)x=0$, the map $w_\phi(x):=\phi(x)x$ belongs to $H_n$ \AAA(recall \eqref{H_space}) \EEE and one can easily verify that
\begin{align*}
(P_T^t\t w_\phi)_{\mathrm{sym}}-\frac{\mathrm{div}_{\S}w_\phi}{n-1}I_x=0\,.
\end{align*}
Since the space of such $\phi$ is infinite-dimensional, we have in particular that $\mathrm{dim}(\mathrm{ker}(Q_{n,\mathrm{conf}}))=\infty$.

Regarding the quadratic form $Q_{n,\mathrm{isop}}$ \AAA in \eqref{Q_isop_sphere}\EEE, we have that it is also nonnegative and its kernel is also  infinite-dimensional. In fact \AAA(with the notation of Theorem \ref{eigenvalue_decomposition})\EEE, one can directly check that 
\begin{equation}\label{infinite_dimensionality_of_Q_n_isop}
H_{n,2}:=\bigoplus_{k=1}^\infty H_{n,k,2}\subseteq \mathrm{ker}(Q_{n,\mathrm{isop}})\implies \mathrm{dim}(\mathrm{ker}(Q_{n,\mathrm{isop}}))=\infty\,.
\end{equation}
To verify \eqref{infinite_dimensionality_of_Q_n_isop} we use the following identity, referred to as \textit{Korn's identity}, that is interesting in its own right and whose derivation is a simple calculation which is also included at the end of Appendix \ref{sec:B}.

\begin{lemma}{\bf{(Korn's identity on $\S$)}}.\label{Korn's identity_sphere}
For every $w\in W^{1,2}(\S;\R^n)$ the following identity holds
\begin{equation}\label{Korn's identity}
\ds\left|(P_T^t\t w)_{\mathrm{sym}}\right|^2=\frac{1}{2}\ds\Big(\big|P_T^t\t w\big|^2
+\left(\mathrm{div}_{\S} w\right)^2\Big)-\frac{n-2}{n}Q_{V_n}(w)\,.
\end{equation}
\end{lemma}

The interesting point of this identity is that when $n\geq 3$, the quadratic form $Q_{V_n}$ appears in the right hand side \AAA of \eqref{Korn's identity} \EEE as some short of \textit{curvature contribution} and it is really a \textit{surface identity}, in the sense that the corresponding identity in the bulk is
\begin{equation}\label{Korn's identity_in_the_bulk}
\int_U \left|(\nabla w)_{\mathrm{sym}}\right|^2=\frac{1}{2}\int_U \left(|\nabla w|^2+(\mathrm{div}w)^2\right)-\frac{1}{2}\int_U\left((\mathrm{div}w)^2-\mathrm{Tr}(\nabla w)^2\right)\,,
\end{equation} 
but the last term on the right hand side \AAA of \eqref{Korn's identity_in_the_bulk} \EEE should now be interpreted as a \textit{boundary-term contribution}.

By using Korn's identity, the form $Q_{n,\mathrm{isop}}$ can be rewritten in a simpler form as
\begin{equation*}
Q_{n,\mathrm{isop}}(w)=\frac{n}{2(n-1)}\ds\Big(\Big|\sum_{j=1}^n x_j\t w^j\Big|^2-\big\langle w,A(w)\big\rangle\Big)\,.
\end{equation*}
But if $w\in H_{n,2}$, then $-\sum_{j=1}^nx_j\t w^j=A(w)=w$ on $\S$ \AAA (recall that  \eqref{Eigenvalue_equation_for_A}, \eqref{reduced_equation_on_S} are then satisfied with $\sigma=1$), \EEE and therefore
$Q_{n,\mathrm{isop}}(w)=0$, which proves the implication in \eqref{infinite_dimensionality_of_Q_n_isop}.

\subsection{Proof of Theorem \ref{coercivity_estimate_for_Q_isom_Q_isop}}\label{Subsection 5.3.}
As we have mentioned in the Introduction, if $u\in W^{1,2}(\S;\R^n)$ and $w:=u-\mathrm{id}_{\S}$, then the \textit{full $L^2$-isometric deficit of $u$} is formally expanded around the identity as 
\begin{equation*}\label{definition_of_isometric_deficit_again}
\delta_{\mathrm{isom}}^2(u):=\ds \left|\sqrt{\t u^t\t u}-I_x\right|^2=Q_{n,\mathrm{isom}}(w)+\ds \mathcal{O}(|\t w|^3)\,,
\end{equation*}
where 
\begin{equation}\label{Qisom_nonnegative}
Q_{n,\mathrm{isom}}(w):=\ds\left|(P_T^t\t w)_{\rm{sym}}\right|^2\,.
\end{equation}
For $n=2$ we have that $\mathrm{dim}(\mathrm{ker}(Q_{2,\mathrm{isom}}))=\infty$. Indeed, for $v(x):=\psi(x)\tau(x)+\phi(x)x:\mathbb{S}^1\mapsto \R^2$, where $\phi,\psi\in C^{\infty}(\mathbb{S}^1,\R)$ and $\tau(x):=(-x_2,x_1)$ is the unit tangent vector field on $\mathbb{S}^1$, it is an easy calculation to check that $
(P_T^t\t v)_{\mathrm{sym}}=\phi-\partial_{\tau}\psi$, i.e., for every $\psi\in C^{\infty}(\mathbb{S}^1;\R)$ the map $v_\psi(x):=\psi(x)\tau(x)+\partial_{\tau}\psi(x)x$ lies in the kernel of $ Q_{2,\mathrm{isom}}$.

For $n\geq 3$, and unlike $Q_{n,\mathrm{conf}}$, which as we have seen has infinite-dimensional kernel,  $Q_{n, \mathrm{isom}}$ has finite-dimensional kernel and actually
\begin{equation*}
\mathrm{ker}(Q_{n,\mathrm{isom}})\simeq Skew(n)\simeq\mathfrak{so}(n)\,,
\end{equation*}
this fact being known as the \textit{infinitesimal rigidity of the sphere}. The reader is referred to \cite[Chapter 12]{spivak1979comprehensive} for a detailed discussion and further references regarding this well known geometric fact that is the linear analogue of the $C^2$-rigidity of the sphere in the Weyl problem.

\begin{remark}\label{Q_n_isom_Q_n_isop_trivial}
Without referring to the classical proof of the infinitesimal rigidity of $\S$, it is very easy to deduce in particular that
\begin{equation}\label{intersection_of_kernels_isom_isop}
\mathrm{ker}(Q_{n,\mathrm{isom}})\cap\mathrm{ker}(Q_{n,\mathrm{isop}})\simeq\mathfrak{so}(n)\,.
\end{equation}
Indeed, if $w\in W^{1,2}(\S;\R^n)$ lies in the common null-space of these two nonnegative forms, we have
\begin{align}\label{null_space_of_Q_isom_Q_isop}
(P_T^t\t w)_{\rm{sym}}=0 \ \ \mathrm{and} \ \  Q_{n,\mathrm{isop}}(w)
=0\,.
\end{align}
By taking the trace in the first equation \AAA in \eqref{null_space_of_Q_isom_Q_isop}\EEE, we see that $\mathrm{div}_{\S}w\equiv0$ on $\S$, and then the second equation \AAA in \eqref{null_space_of_Q_isom_Q_isop}, in view also of \eqref{Q_isop_sphere}, \EEE reduces to 
\begin{equation}\label{null_space_of_Q_isop_alternative_1}
\frac{1}{2}\frac{n}{n-1}\ds |\t w|^2-Q_{V_n}\Big(w-\ds w\Big)=0\,.
\end{equation} 
By the alternative formula \AAA \eqref{Q_V_n_alternative}, \EEE  
the equation \eqref{null_space_of_Q_isop_alternative_1} results in 
\begin{equation}\label{common_eigenspace_of_Q_isom_Q_isop_final_equation}
\left(\frac{1}{n-1}\ds|\t w|^2-\ds\Big|w-\ds w\Big|^2\right)+n\ds\Big\langle w-\ds w,x\Big\rangle^2=0\,.
\end{equation}
Once again, the quantity in the parenthesis \AAA in \eqref{common_eigenspace_of_Q_isom_Q_isop_final_equation} \EEE is nonnegative, being the $L^2$-Poincare deficit of $w$, and therefore the only solutions to \eqref{common_eigenspace_of_Q_isom_Q_isop_final_equation} are maps $w$ for which 
\begin{align}\label{Lambda_skew_symmetric}
w(x)-\ds w=\Lambda x \mathrm{\ \ for\ } \Lambda \in \R^{n\times n} \mathrm{\ \ and \ \ } \Big\langle w-\ds w,x\Big\rangle \equiv 0 \ \mathrm{\ on\ } \S.
\end{align}
Hence, \AAA by \eqref{Lambda_skew_symmetric}, \EEE 
$\Lambda ^t=-\Lambda $ and reversely any map of the form $w(x)=\Lambda x+b$, where $\Lambda \in Skew(n), b\in \R^n$ is in the null-space of both quadratic forms. \textit{It would be interesting if by some algebraic manipulations one could directly prove that whenever $n\geq 3$, $Q_{n,\mathrm{isom}}(w)=0\implies Q_{n,\mathrm{isop}}(w)=0$. Combined with the argument just presented, this would give an alternative algebraic proof of the infinitesimal rigidity of the sphere, supplementing the differential-geometric one.} 
\end{remark}

\begin{remark}
Even though for $n\geq 3$ we have $\mathrm{ker}(Q_{n,\mathrm{isom}})\simeq \mathfrak{so}(n)$, an estimate of the type
\begin{equation*}
\ds |\t w-FP_T |^2\lesssim Q_{n,\mathrm{isom}}(w) \quad \text{for some } \ F\in Skew(n),
\end{equation*}
does not hold, the obstacle being (loosely speaking) the derivatives of the normal component of $w$. For example, if one considers purely normal displacements
$w_{\phi}(x):= \phi (x)x$ with $\phi\in W^{1,2}(\S;\R)$, then by a straightforward computation one can check that
\begin{equation*}
(P_T^t\t w_\phi)_{\mathrm{sym}} = \phi I_x \implies Q_{n,\mathrm{isom}}(w_{\phi})=(n-1)\ds |\phi|^2\,,
\end{equation*}
whereas the full gradient of $w_{\phi}$ also has derivatives of $\phi$ in it. In coordinates,
\begin{equation*}
(\t w_{\phi})_{ij}= \phi(P_T)_{ij} + x_i \partial_{\tau_j}\phi\,,
\end{equation*}
so if the estimate above was to be valid, it would resemble some short of reverse-Poincare inequality, which is of course generically false. 
Combining additively $Q_{n,\mathrm{isom}}$ and $Q_{n,\mathrm{isop}}$ has though the merit of providing a Korn-type inequality in terms of the full gradient, as described in Theorem \ref{coercivity_estimate_for_Q_isom_Q_isop}.
\end{remark}	

\begin{proof}[Proof of Theorem \ref{coercivity_estimate_for_Q_isom_Q_isop}]
For $\alpha>0$ let us call
\begin{equation}\label{Q_n_alpha_form}
Q_{n,\alpha}(w) := \alpha Q_{n,\mathrm{isom}}(w) + Q_{n,\mathrm{isop}}(w)\,.
\end{equation}
By looking at the formulas \AAA \eqref{Q_isop_sphere} and \eqref{Qisom_nonnegative}, \EEE and rearranging terms, \AAA \eqref{Q_n_alpha_form} \EEE gives
\begin{equation}\label{Q_n_n_n-1_form}
Q_{n,\frac{n}{n-1}}(w) =\frac{n}{2(n-1)}\ds \big(|\t w|^2+(\mathrm{div}_{\S}w)^2\big)-Q_{V_n}(w)\,.
\end{equation}
With the notation we had introduced in Subsection \ref{Subsection 4.2}, the estimate to be proven for $Q_{n,\frac{n}{n-1}}$ reads
\begin{equation}\label{estimate_for_one_combination_of_Q_isom_Q_isop}
Q_{n,\frac{n}{n-1}}(w) \gtrsim \ds \big|\t w-\t w_{1,2}\big|^2 \quad \forall w\in W^{1,2}(\S;\R^n)\,.
\end{equation}
Once we have \eqref{estimate_for_one_combination_of_Q_isom_Q_isop}, the case of general $\alpha>0$ is immediate, since:
\begin{align*}
\begin{split}
\mathrm{For}\ \alpha>\frac{n}{n-1}>0: \ Q_{n,\alpha}(w)&\geq Q_{n,\frac{n}{n-1}}(w)\gtrsim \ds \big|\t w-\t w_{1,2}\big|^2\,.\\
\mathrm{For}\ 0<\alpha\leq\frac{n}{n-1}: \ Q_{n,\alpha}(w)&=\frac{(n-1)\alpha}{n}\left(\frac{n}{n-1} Q_{n,\mathrm{isom}}(w)+\frac{n}{(n-1)\alpha}Q_{n,\mathrm{isop}}(w)\right)\\
&\geq \frac{(n-1)\alpha}{n}Q_{n,\frac{n}{n-1}}(w)\gtrsim\alpha\ds \big|\t w-\t w_{1,2}\big|^2\,.
\end{split}
\end{align*}
The proof of \eqref{estimate_for_one_combination_of_Q_isom_Q_isop} can now be performed by following exactly the same procedure as in the proof of
Theorem \ref{main_coercivity_estimate_conformal_general_dimension} in Subsection \ref{Subsection 5.1.}. The latter could of course also be phrased in terms of any positive combination of $Q_{n,\mathrm{conf}}$ \AAA in \eqref{Q_n_conformal_sphere} \EEE and $Q_{n,\mathrm{isop}}$ \AAA in \eqref{Q_isop_sphere}, \EEE and not only of $Q_n=Q_{n,\mathrm{conf}}+Q_{n,\mathrm{isop}}$, \AAA as in \eqref{Q_n_definition}\EEE.

The arguments after Lemmata \ref{div_Q_n_in_the_eigenspaces} and \ref{bd_lemma} (handling the mixed
$\mathrm{div}_{\S}$-terms after the expansion in spherical harmonics \AAA as in \eqref{mixed_terms_first_estimate}-\eqref{Q_n_new_estimate}\EEE) are then modified accordingly. A last trivial comment in this case is that for $k \geq 1$ and $i= 1,2, 3$,
if $w\in H_{n,k,i}$ \AAA(recall Theorem \ref{eigenvalue_decomposition} and Lemma \ref{div_Q_n_in_the_eigenspaces})\EEE,
\begin{equation*}
Q_{n,\frac{n}{n-1}}(w) = C'_{n,k,i}\ds |\t w|^2,  \quad \text{where } C'_{n,k,i}:=\frac{n}{2(n-1)}+\frac{n}{2(n-1)}\alpha_{n,k,i}-c_{n,k,i}\,.
\end{equation*}
Recalling the values of the constants from \AAA \eqref{Q_v_n_div_n_Q_n_in_eigenspaces} \EEE and the table \eqref{constants2}, we have that
\begin{equation}\label{new_constant_bigger_than_old_one}
C'_{n,k,i} > C_{n,k,i} \mathrm{\ \ for\ \ } i=1,3, \ \ \mathrm{while}\ \   C'_{n,k,2} = C_{n,k,2}\ \  (\mathrm{because}\ \ \alpha_{n,k,2}= 0)\,.
\end{equation}
So for the quadratic form \AAA in \eqref{Q_n_n_n-1_form}, one can directly check from \AAA \eqref{new_constant_bigger_than_old_one} \EEE that  $C'_{n,1,2}=0$, but otherwise $
\underset{(k,i)\neq(1,2)}{\min}C'_{n,k,i}>0$, which eventually leads to the desired coercivity estimate \AAA \eqref{estimate_for_one_combination_of_Q_isom_Q_isop}\EEE.
\end{proof}

\begin{remark}
Since in this case we had an easy argument to infer that $\mathrm{ker}(Q_{n,\mathrm{isom}})\cap\mathrm{ker}(Q_{n,\mathrm{isop}})\simeq\mathfrak{so}(n)$ \AAA(see Remark \ref{Q_n_isom_Q_n_isop_trivial})\EEE, the proof of Theorem \ref{coercivity_estimate_for_Q_isom_Q_isop}, for $\alpha=\frac{n}{n-1}$ for example, could also be performed by a standard \textit{contradiction}$\backslash$\textit{compactness argument}, for an abstract constant though.
Indeed, suppose that for $\alpha:=\frac{n}{n-1}$ the estimate \eqref{alpha_Q_isom_Q_isop_coercivity_estimate} \AAA (or equivalently \eqref{estimate_for_one_combination_of_Q_isom_Q_isop}) \EEE is false. By translation and scaling invariance of the estimate, there exists a sequence $(w_k)_{k\in \mathbb{N}}\subset W^{1,2}(\S;\R^n)$ such that for all $k\in \mathbb{N}$,
\begin{equation}\label{properties_of_contradictory_sequence_for_Korn_isom_isop}
\ds w_k=0, \quad \Pi_{H_{n,1,2}}w_k=0, \quad \ds |\t w_k|^2=1\,,
\end{equation}  
and
\begin{equation}\label{further_prop_of_contradictory_sequence_for_Korn_isom_isop}
Q_{n,\frac{n}{n-1}}(w_k)\leq \frac{1}{k}\ds |\t w_k|^2= \frac{1}{k}\,,
\end{equation}
in particular, 
\begin{equation*}
Q_{n,\frac{n}{n-1}}(w_k)\longrightarrow 0 \quad \text{as } k\to \infty\,.
\end{equation*}
Up to a further nonrelabeled subsequence we can assume that there exists $w\in W^{1,2}(\S;\R^n)$ with $\ds w=0$ such that $w_k\rightharpoonup w$ in $W^{1,2}(\S;\R^n)$, and also pointwise $\mathcal{H}^{n-1}$-a.e. on $\S$. But then (recalling \AAA \eqref{quadratic_form_of_volume} and \eqref{Aoperator})\EEE,
\begin{align}\label{volumes_go_to_volume}
Q_{V_n}(w_k)=\frac{n}{2}\ds \langle w_k,A(w_k)\rangle\longrightarrow \frac{n}{2}\ds \langle w,A(w)\rangle=Q_{V_n}(w) \ \ \mathrm{as}\ k\to \infty,
\end{align}
since $w_k\to w$ strongly in $L^2(\S;\R^n)$ and $A(w_k)\rightharpoonup A(w)$ weakly in $L^2(\S;\R^n)$. Thus, by lower semicontinuity of the first two terms in \eqref{Q_n_n_n-1_form} under weak convergence in $W^{1,2}(\S;\R^n)$, we obtain
\begin{equation*}
0\leq Q_{n,\frac{n}{n-1}}(w)\leq \liminf_{k\to\infty} Q_{n,\frac{n}{n-1}}(w_k)=0\,,
\end{equation*}
i.e. $Q_{n,\frac{n}{n-1}}(w)=0$ and therefore, \AAA by \eqref{intersection_of_kernels_isom_isop}, \EEE
\begin{equation*}
w\in \mathrm{ker}(Q_{n,\mathrm{isom}})\cap\mathrm{ker}(Q_{n,\mathrm{isop}})\simeq\mathfrak{so}(n)\,,
\end{equation*} i.e., $w\in H_{n,1,2}$. Moreover, being the ($\mathcal{H}^{n-1}$-a.e.) pointwise limit of $(w_k)_{k\in \mathbb{N}}\subset H_{n,1,2}^{\bot}$, we must also have that $w\in H_{n,1,2}^{\bot}$, since $H_{n,1,2}$ is finite-dimensional and therefore its orthogonal complement in $W^{1,2}(\S;\R^n)$ is a closed subspace. This forces $w\equiv 0$ on $\S$ and in particular, \AAA by \eqref{volumes_go_to_volume}, \EEE 
\begin{equation}\label{Q_V_n_to_0}
Q_{V_n}(w_k)\to 0\ \text{ as }k\to \infty\,.
\end{equation}
But then \AAA\eqref{Q_n_n_n-1_form}, \eqref{further_prop_of_contradictory_sequence_for_Korn_isom_isop} and \eqref{Q_V_n_to_0} \EEE imply that
\begin{equation*}
0\leq \frac{n}{2(n-1)}\ds |\t w_k|^2\leq \frac{1}{k}+Q_{V_n}(w_k)\longrightarrow 0 \ \mathrm{\ as \ } k\to \infty\,,
\end{equation*}
contradicting the assumption \AAA in \eqref{properties_of_contradictory_sequence_for_Korn_isom_isop} \EEE that $\ds |\t w_k|^2=1$.
\end{remark}

\begin{remark}
As we had mentioned in Remark \ref{Remark_for_L}, the argument just described could have also been used to prove Theorem \ref{main_coercivity_estimate_conformal_general_dimension} if we knew already that the kernel of $Q_n$ \AAA(defined in \eqref{Q_n_definition}) \EEE is finite-dimensional (and actually equal to $H_{n,0}$). Nevertheless, as we had remarked therein, there does not seem to be a direct argument to show this fact when $n\geq 4$, neither by trying to solve the Euler-Lagrange equation associated to the operator $\mathcal{L}$ in $\eqref{L_operator}$ explicitely, nor by trying to argue as above.

Indeed, if $w\in \mathrm{ker}(Q_n)\iff w\in \mathrm{ker}(Q_{n,\mathrm{conf}})\cap\mathrm{ker}(Q_{n,\mathrm{isop}})$ (see \eqref{Q_n_conformal_sphere}, \eqref{Q_isop_sphere}), then again the following two equations must be satisfied simultaneously,
\begin{equation}\label{equations_in_the_common_eigenspace_conformal_isoperimetric}
(P_T^t\t w)_{\rm{sym}}=\frac{\mathrm{div}_{\S}w}{n-1}I_x \mathrm{\ \ \ and}\ \ \  
Q_{n,\mathrm{isop}}(w)=0\,.
\end{equation}
Because of the first equation \AAA in \eqref{equations_in_the_common_eigenspace_conformal_isoperimetric}\EEE, the second one therein results in the equation
\begin{align}\label{equations_in_the_common_eigenspace_conformal_isoperimetric_3}
\frac{1}{2}\frac{n}{n-1}\ds\left[|\t w|^2+(\mathrm{div}_{\S}w)^2-2\frac{(\mathrm{div}_{\S}w)^2}{n-1}\right]=Q_{V_n}(w) \nonumber\\[5pt]
\iff \frac{1}{n-1}\ds |\t w|^2+\frac{n-3}{(n-1)^2}\ds(\mathrm{div}_{\S}w)^2-\ds \langle w,A(w)\rangle =0\,,
\end{align}	
i.e., \AAA \eqref{equations_in_the_common_eigenspace_conformal_isoperimetric_3} \EEE ends up back to the original equation $Q_n(w)=0$. Arguing directly with the eigenvalue decomposition with respect to $A$ \AAA(recall Theorem \ref{eigenvalue_decomposition}) \EEE also has the extra benefit of showing explicitely how the form $Q_n$ behaves in each one of the eigenspaces separately \AAA(see Lemma \ref{div_Q_n_in_the_eigenspaces}), \EEE and also gives an explicit lower bound \AAA (\eqref{minimum_of_constants_second}) \EEE for the value of the optimal constant $C_n$ in the coercivity estimate.
\end{remark}

\AAA \section*{Acknowledgements} The second author is supported by the \textit{Deutsche Forschungsgemeinschaft (DFG, German Research Foundation) under Germany's Excellence Strategy EXC 2044-390685587, Mathematics M\"unster: Dynamics--Geometry--Structure} and would also like to thank the \textit{Max Planck Institute for Mathematics in the Sciences in Leipzig}, where his PhD project was carried out.
\EEE

\appendix

\section{A new simple proof of Liouville's theorem on $\S$ and a qualitative analogue}\label{New_proof_of_Liouville_Appendix_A}

\begin{proof}[Proof of Theorem \ref{spherical Liouville's Theorem}]
We present the proof in the case of generalized orientation-preserving maps, since the case of orienation-reversing ones is identical (or it can be retrieved by the former by composing with the \textit{flip} $x:=(x_1,\dots,x_{n-1},x_n)\mapsto(x_1,\dots,x_{n-1},-x_n)$).

For part $(i)$ of the theorem, we have that any map of the form $u(x)=Rx$ with $R\in SO(n)$ is of course an orientation-preserving isometry of $\S$. Conversely, let $n\geq 2$, $p\in [1,\infty]$ and $u\in W^{1,p}(\S;\S)$ be a generalized orientation-preserving isometric map. By definition, this means that at $\mathcal{H}^{n-1}$-a.e. $x\in \S$ the intrinsic gradient of $u$ is an orientation-preserving linear map between $T_x\S$ and $T_{u(x)}\S$, such that $$d_xu^td_xu=I_x\,,$$ or equivalently, in terms of the extrinsic gradient, $$\left(\t u^t\t u\right)(x)=I_x\,.$$ In particular, for $\mathcal{H}^{n-1}$-a.e. $x\in \S$ one has 
\begin{equation}\label{norm_and_pullback}
\frac{\left|\t u(x)\right|^2}{n-1}=1, \mathrm{\ and \ } u^\sharp(\omega)= \omega \mathrm{\ \ for \  every \ } (n-1)\text{-form \ } \omega \mathrm{ \ on \ } \S\,.
\end{equation} 
By the change of variables formula applied to the vector-valued $(n-1)$-form $xdv_g$, we obtain
\begin{equation}\label{isometries_of_S_have_zero_mean}
0 = \ds xdv_g= \ds u^\sharp(xdv_g)=\ds u(x)dv_g(x)\,, \mathrm{\ i.e.,\ } \ds u= 0\,.
\end{equation}
Hence, Poincare's inequality \AAA\eqref{Poincare} \EEE on $\S$, together with \AAA\eqref{norm_and_pullback}, \eqref{isometries_of_S_have_zero_mean}, \EEE and the fact that $|u|\equiv 1$, yield
\begin{equation*}
1\ = \ \ds \frac{\left|\t u\right|^2}{n-1} \geq \ds |u|^2 =\ 1\,.
\end{equation*}
As we have also encountered before \AAA(see also Appendix \ref{sec:C})\EEE, the equality case in the Poincare inequality implies that in the Fourier expansion of $u$ in spherical harmonics, no other spherical harmonics except the first order ones should appear, hence $u(x)=Rx$ for some $R\in \R^{n\times n}$. But this linear map would transform $\S$ into the boundary of an ellipsoid, which after possibly an orthogonal change of coordinates is
\begin{equation*}\label{ellipsoidal_equation}
u(\S)=\Big\{y=(y_1,\dots,y_n) \in \mathbb{R}^n: \frac{y_1^2}{\alpha_1^2}+\dots+\frac{y_n^2}{\alpha_n^2}=1\Big\}\,,
\end{equation*}
where $0\leq\alpha_1 \leq \dots \leq \alpha_n$\  are the eigenvalues of $\sqrt{R^tR}$. By assumption, $u(\S)\equiv\S$ and this forces $\alpha_1=\dots=\alpha_n=1$, i.e., $R\in O(n)$ and in particular, since $u$ is assumed to be orientation-preserving, $R\in SO(n)$.

For part $(ii)$ we can argue similarly, after making use of the following useful fact.

\begin{claim}Given $u\in W^{1,n-1}(\S;\S)$ of degree one, one can always find a Möbius transformation $\phi_{\xi_0,\lambda_0}$ of $\S$ \AAA$\rm{(}$see \eqref{conformal_group_of_S_representation}$\rm{)}$ \EEE so that 
\begin{equation}\label{fixing_mean_value_with_Moebius}
\ds u\circ\phi_{\xi_0,\lambda_0}=0\,.
\end{equation}
\end{claim}

Indeed, assume first that $u\in C^{\infty}(\S;\S)$ and has degree one, in particular $u$ is surjective. If $b_u:=\ds u=0$, there is nothing to prove. If $b_u\neq 0$, consider the map $F:\S \times [0,1]\mapsto \overline {B_1}$, defined as 
\begin{equation*}\label{homotopy_for_fixing_the_center}
F(\xi,\lambda):=\ds u\circ\phi_{\xi,\lambda} \ \ \mathrm{for}\ \lambda\in(0,1]\,, \ \ \mathrm{and} \ \ F(\xi,0):=\lim_{\lambda\downarrow 0^+}F(\xi,\lambda)\,.
\end{equation*}
The map $F$ is continuous with $F(\xi,0)=u(\xi)$ for every $\xi\in\S$, i.e., $F(\S,0)=u(\S)=\S$, whereas $F(\S,1)=\{b_u\}$.
In other words, $F$ is a continuous homotopy between $\S$ and the point $b_u\in \overline {B_1}\setminus \{0\}$, and therefore there exists $\lambda_0\in(0,1)$ such that $0\in F(\S,\lambda_0)$, i.e., there exists also $\xi_0\in \S$ such that $F(\xi_0,\lambda_0)=0$.

In the general case of a map $u\in W^{1,n-1}(\S;\S)$ of degree 1, by the approximation property given in \cite[Section I.4, Lemma 7]{brezis1995degree}, there exists a sequence $(u_j)_{j\in \mathbb{N}}\subset C^{\infty}(\S;\S)$ with the property that
\begin{equation*}\label{approximation_property}
u_j\rightarrow u \mathrm{\ \ strongly \ in} \ \ W^{1,n-1}(\S;\S) \mathrm{\ and \ } \mathrm{deg}u_j=\mathrm{deg}u=1\ \ \forall j\in \mathbb{N}\,.
\end{equation*} 
Up to passing to a non-relabeled subsequence, we can without loss of generality also suppose that $u_j\rightarrow u$ and $\t u_j\rightarrow \t u$ pointwise $\mathcal{H}^{n-1}$-a.e. on $\S$. Since the maps $(u_j)_{j\in \N}$ are smooth and surjective, by the previous argument there exist $(\xi_j)_{j\in \mathbb{N}}\subset\S$ and $(\lambda_j)_{j\in \mathbb{N}}\subset (0,1]$ so that for every $j\in \mathbb{N}$,
\begin{equation}\label{sequence_of_Möbius_fixing_centers}
\ds u_j\circ \phi_{\xi_j,\lambda_j}=0\,.
\end{equation}
Up to non-relabeled subsequences, we can suppose further that $\xi_j\rightarrow\xi_0\in \S$ and $\lambda_j\rightarrow\lambda_0\in [0,1]$, thus $\phi_{\xi_j,\lambda_j}\to \phi_{\xi_0,\lambda_0}$ pointwise $\mathcal{H}^{n-1}$-a.e. on $\S$, and also weakly in $W^{1,n-1}(\S;\S)$.

In fact $\lambda_0\in (0,1]$, i.e., the Möbius transformations $(\phi_{\xi_j,\lambda_j})_{j\in \mathbb{N}}$ do not converge to the trivial map $\phi_{\xi_0,0}(x)\equiv\xi_0$. Indeed, suppose that this was the case. Then $u_j\circ\phi_{\xi_j,\lambda_j}\to u(\xi_0)$ pointwise $\mathcal{H}^{n-1}$-a.e., and $|u_j\circ\phi_{\xi_j,\lambda_j}|\equiv1$, so we could use the Dominated Convergence Theorem \AAA and \eqref{sequence_of_Möbius_fixing_centers}, \EEE to infer that
\begin{align*}
\begin{split}
u(\xi_0)=\ds u(\xi_0)=\lim_{j\to\infty}\ds u_j\circ\phi_{\xi_j,\lambda_j}=0\,,\\ 
|u(\xi_0)|=\ds |u(\xi_0)|=\lim_{j\to\infty}\ds |u_j\circ\phi_{\xi_j,\lambda_j}|=1\,,
\end{split}
\end{align*}
and derive a contradiction. Having justified that $\lambda_0\in (0,1]$, what we actually obtain by the Dominated Convergence Theorem and \eqref{sequence_of_Möbius_fixing_centers} is that 
\begin{equation}\label{precomposition_mean_value_0}
\d2 u\circ \phi_{\xi_0,\lambda_0}= 0\,.
\end{equation}

Continuing with the proof of \AAA Theorem \ref{spherical Liouville's Theorem}\EEE$(ii)$, we of course have that all the maps given by \AAA \eqref{conformal_group_of_S_representation} \EEE are orientation-preserving conformal diffeomorphisms of $\S$. Conversely, if $u\in W^{1,n-1}(\S;\S)$ is a generalized orientation-preserving conformal map, similarly to $(i)$ we have that $\mathcal{H}^{n-1}$-a.e. on $\S$,  
$$\left(d_xu^td_xu\right)(x)=\frac{|d_x u|^2}{n-1}I_x\,,$$
or equivalently, in terms of the extrinsic gradient,
\begin{equation}\label{conformal_extrinsic}
\left(\t u^t\t u\right)(x)=\frac{|\t u|^2}{n-1}I_x\,.
\end{equation}
By taking the determinant in both sides of \AAA \eqref{conformal_extrinsic}, \EEE we get that $\mathcal{H}^{n-1}$-a.e. on $\S$,
\begin{equation}\label{determinant=n-1_Dirichlet_energy}
\sqrt{\mathrm{det}\left(\t u^t\t u\right)}=\left(\frac{\left|\t u\right|^2}{n-1}\right)^{\frac{n-1}{2}}.
\end{equation}
Precomposing with the Möbius map $\phi_{\xi_0,\lambda_0}\in Conf_+(\S)$ of the previous claim, we have that the map $\tilde u:=u\circ \phi_{\xi_0,\lambda_0}$, whose mean value is 0 \AAA by \eqref{precomposition_mean_value_0}\EEE, is also a generalized orientation-preserving conformal transformation of $\S$ of degree 1, and therefore \AAA by \eqref{determinant=n-1_Dirichlet_energy}, \EEE
\begin{equation*}
\tilde u^\sharp(dv_g)=\sqrt{\mathrm{det}(\t \tilde u^t\t \tilde u)}\ dv_g=\left(\frac{|\t \tilde u|^2}{n-1}\right)^{\frac{n-1}{2}}dv_g\,.
\end{equation*}
By approximation, the analytic formula for the degree in terms of integration of $(n-1)$-forms on $\S$ holds true for $\tilde u$ as well, i.e.,
\begin{equation}\label{n-1_Dirichlet_1}
\ds\left(\frac{|\t \tilde u|^2}{n-1}\right)^{\frac{n-1}{2}}=\ds \tilde u^\sharp(dv_g) = \mathrm{deg}\tilde u\ \ds dv_g = 1\,.
\end{equation}
We can now use \AAA \eqref{n-1_Dirichlet_1}\EEE, Jensen's inequality, and the sharp Poincare inequality on $\S$ \AAA(\eqref{Poincare})\EEE, to obtain the chain of inequalities
\begin{align}\label{basic_chain_of_inequalities_for_Liouvilles_thm}
1=\ds\left(\frac{|\t \tilde u|^2}{n-1}\right)^\frac{n-1}{2}\geq\left(\ds\frac{|\nabla_{T}\tilde u|^2}{n-1}\right)^{\frac{n-1}{2}}\geq\left(\ds|\tilde u|^2\right)^\frac{n-1}{2}=1\,,
\end{align}
since $\ds \tilde u=0$ and $|u|\equiv 1$ $\mathcal{H}^{n-1}$-a.e. on $\S$. Arguing as in part $(i)$, we deduce that $\tilde u=R\mathrm{id}_{\S}$, i.e., $u=R\phi_{\xi,\lambda}$, where $R\in SO(n)$, $\xi:=\xi_0\in \S$ and $\lambda:=\frac{1}{\lambda_0}>0$.\qedhere
\end{proof}
	
\begin{remark}\label{tangent_space_of_conformal_group_to_identity}
The Möbius transformations of $\S$ \AAA given by \eqref{conformal_group_of_S_representation} \EEE could of course alternatively be described by performing an inversion in $\overline{T_\xi\S}$ with respect to some center, say the origin $\xi$ of the affine hyperplane $T_\xi\S$ of $\R^n$, and some radius, say $\sqrt{\lambda}>0$. These maps however would correspond exactly to the Möbius transformations produced by dilation in $T_\xi\S$ by factor $\frac{1}{\lambda}$, composed finally with a flip in $\R^n$, i.e., an orthogonal map that would change back the orientation.
\AAA By Liouville's Theorem \ref{spherical Liouville's Theorem}, the conformal group of the sphere is given by
\begin{equation*}
Conf(\S)=\{O\phi_{\xi,\lambda}\colon\ O\in O(n),\xi\in \S,\lambda>0\}\,,
\end{equation*}
and is a actually a Lie group, i.e., a differentiable manifold (of dimension $\tfrac{n(n+1)}{2}$) with a group structure, given by composition of maps. \EEE
Analytically, the maps $(\phi_{\xi,\lambda})_{\xi\in \S,\lambda>0}$ are given by the formula
\begin{equation}\label{Möbius_transformations_formula}
\phi_{\xi,\lambda}(x):=\frac{-\lambda^2\big(1-\langle x,\xi \rangle\big)\xi+2\lambda\big(x-\langle x,\xi \rangle \xi\big)+\big(1+\langle x,\xi \rangle\big)\xi}{\lambda^2\big(1-\langle x,\xi \rangle\big)+\big(1+\langle x,\xi \rangle\big)}\,.
\end{equation} 	 
\AAA Using \AAA\eqref{Möbius_transformations_formula}, the corresponding Lie algebra of infinitesimal Möbius transformations, i.e., the tangent space of $Conf(\S)$ at the $\mathrm{id}_{\S}$, can easily be identified. \EEE Indeed, it is then an elementary exercise in differential geometry to check that  
\begin{flalign}\label{tangent_space_of_the_conformal_group_at_the_identity_1}
\hspace{-1em}T_{\mathrm{id}_{\S}}Conf(\S)\equiv\left\{Sx+\mu\big(\langle x,\xi\rangle x-\xi\big):\S\mapsto \R^n;\ S\in Skew(n), \xi\in\S, \mu\in \R\right\}\,,
\end{flalign}
which is the representation that we made use of in the proof of Lemma \ref{fixingMobius}.
\end{remark}

\begin{remark}
In the conformal case, the argument for the proof of Theorem \ref{spherical Liouville's Theorem}$(ii)$ can easily be modified in order to give a compactness statement for sequences of orientation-preserving (resp. orientation-reversing) approximately conformal maps on $\S$ of degree $1$ (resp. $-1$)\AAA, see the subsequent Lemma \ref{compactness_of_degree_one_almost_conformal_maps_on_S}\EEE. When $n=3$, the statement therein essentially reduces to a well known compactness result for harmonic maps of degree $\pm 1$ on $\mathbb{S}^2$ (cf.~\cite[Theorem 2.4, Step 1 in the proof]{bernand2020quantitative}, as well as \cite{lin1999mapping}), which was proven using a \textit{concentration-compactness argument} in the spirit of {\sc P.L. Lions} \cite{lions1985concentration}. With the observation that the Möbius transformations can be used to \textit{``globally invert''} a map from $\S$ to itself, \textit{in the sense of fixing its mean value to be 0}, we can give a simpler and more elementary proof of this fact, which can be appropriately generalized in every dimension $n\geq 3$. For further applications of this simple observation the interested reader is also referred to \cite{hirsch2020},\cite{topping2020} for two different and shorter proofs of \cite[Theorem 2.4]{bernand2020quantitative}.
In the following lemma we present again for simplicity the case of orientation-preserving degree 1 maps. 
\end{remark}

\begin{lemma}\label{compactness_of_degree_one_almost_conformal_maps_on_S}
Let $n\geq 3$ and $(u_j)_{j\in\mathbb{N}}\subset W^{1,n-1}(\S;\S)$ be a sequence of generalized orientation-preserving maps of degree 1 which are approximately conformal, in the sense that
\begin{equation*}
\lim_{j\to\infty}\ds\left(\left(\frac{|\t u_j|^2}{n-1}\right)^\frac{n-1}{2}-\sqrt{\mathrm{det}\left(\t u_j^t\t u_j\right)}\right)=0\,,
\end{equation*}
which as a condition is in this case equivalent to 
\begin{equation}\label{n-1_Dirichlet_energy_goes_to_1}
\lim_{j\to\infty}\ds\left(\frac{|\t u_j|^2}{n-1}\right)^\frac{n-1}{2}=1\,.
\end{equation}
Then, there exist Möbius transformations $(\phi_j)_{j\in\mathbb{N}}\subset Conf_+(\S)$ and $R\in SO(n)$ so that up to a non-relabeled subsequence,
\begin{equation}\label{compactness_conformal_maps_sphere}
u_j\circ\phi_j\rightarrow R\mathrm{id}_{\S} \ \ \mathrm{strongly \ in \ } W^{1,n-1}(\S;\S)\,.
\end{equation}
\end{lemma}	

\begin{proof}
By the degree 1 condition, \AAA as in \eqref{fixing_mean_value_with_Moebius},\EEE we can again find $(\xi_j)_{j\in \mathbb{N}}\subset \S$ and $(\lambda_j)_{j\in \mathbb{N}}\subset (0,1]$, so that after setting $\phi_j:=\phi_{\xi_j,\lambda_j}\in Conf_+(\S)$ and $\tilde{u}_j:=u_j\circ\phi_j$, we have 
\begin{equation}\label{fixing_mean_value_again}
\ds \tilde{u}_j=0\,.
\end{equation}
Thanks to the conformal invariance of the $(n-1)$-Dirichlet energy, \eqref{n-1_Dirichlet_energy_goes_to_1} is left unchanged, i.e.,
\begin{equation}\label{conformal_deficit_goes_to_0_for_tilde_u_alternative}
\lim_{j\to\infty}\ds\left(\frac{|\t \tilde{u}_j|^2}{n-1}\right)^\frac{n-1}{2}=1\,.
\end{equation}
\AAA Because of \eqref{fixing_mean_value_again} and \eqref{conformal_deficit_goes_to_0_for_tilde_u_alternative}, \EEE the sequence $(\tilde{u}_j)_{j\in\mathbb{N}}$ is in particular uniformly bounded in $W^{1,n-1}(\S;\S)$, hence up to a non-relabeled subsequence converges weakly in $W^{1,n-1}(\S;\S)$, and up to a further one
also pointwise $\mathcal{H}^{n-1}$-a.e. to a map $\tilde{u}\in W^{1,n-1}(\S;\S)$. Since $\tilde {u}_j\to \tilde{u}$ strongly in $L^{n-1}(\S;\S)$, we obtain \AAA by \eqref{fixing_mean_value_again} \EEE that in particular, 
\begin{equation}\label{tilde_u_has_mean_zero}
\ds \tilde u=\lim_{j\to\infty} \ds \tilde{u}_j=0\,,
\end{equation}
and by lower semicontinuity of the $(n-1)$-Dirichlet energy under weak convergence \AAA and \eqref{conformal_deficit_goes_to_0_for_tilde_u_alternative}\EEE,
\begin{equation}\label{weak_l_s_c_of_L_n-1_norm}
\ds \left(\frac{|\t \tilde u|^2}{n-1}\right)^{\frac{n-1}{2}}\leq\liminf_{j\to \infty}\ds\left(\frac{|\t \tilde{u}_j|^2}{n-1}\right)^\frac{n-1}{2}=1\,.
\end{equation}
We can then apply the same argument as in \AAA \eqref{basic_chain_of_inequalities_for_Liouvilles_thm} in \EEE the proof of Theorem \ref{spherical Liouville's Theorem}$(ii)$, to obtain the chain of inequalities
\begin{equation}\label{main_line_of_inequalities_compactness}
1\geq \ds \left(\frac{|\t \tilde u|^2}{n-1}\right)^{\frac{n-1}{2}}\geq\left(\ds\frac{|\nabla_{T}\tilde u|^2}{n-1}\right)^{\frac{n-1}{2}}\geq\left(\ds\left|\tilde u\right|^2\right)^\frac{n-1}{2}=1\,,
\end{equation}
and with the same reasoning as in there, we conclude that $\tilde{u}(x)=Rx$ for some $R\in O(n)$. Finally, by \eqref{tilde_u_has_mean_zero}-\eqref{main_line_of_inequalities_compactness} we actually obtain that
$\tilde{u}_j\to\tilde u$ strongly in  $W^{1,n-1}(\S;\S)$. Since the degree is stable under this notion of convergence, 
\begin{equation*}
1=\mathrm{deg}\tilde u=\ds\Big\langle Rx,\bigwedge_{i=1}^{n-1}\partial_{\tau_i}(Rx)\Big\rangle= \mathrm{det}R\,, 
\end{equation*}
i.e., indeed $R\in SO(n)$. \qedhere
\end{proof}
\vspace{-1.5em}
\AAA\section{Integral identities for Jacobians, Taylor expansions of the deficits and proof of Korn's identity}\label{sec:B}
We start this appendix by collecting and proving some integral identities for Jacobians that we used in the bulk of the paper, and especially in the proof of Lemma \ref{separate_determinant_estimate}. 
\begin{lemma}\label{integral_identities}
Let $u\in W^{1,2}(\s;\R^3)$ and let $u_h\colon \overline{B_1}\mapsto \R^3$ be as usual its harmonic continuation in $B_1$, taken componentwise. Then $\rm{(}$with the notation adopted in \eqref{determinant_of_signed_volume} for $n=3$$\rm{)}$,
\begin{equation}\label{determinant_bulk_surface}
\dashint_{B_1}\mathrm{det}\nabla u_h=V_3(u):=\d2\big\langle u, \partial_{\tau_1}u\wedge\partial_{\tau_2}u\big\rangle\,.
\end{equation}
Moreover, if $w\in W^{1,2}(\s;\R^3)$ and $w_h$ is defined analogously, then
\begin{align}\label{determinant_expansion_around_I}
\begin{split}
\dashint_{B_1}\mathrm{det}(I_3+\nabla w_h)
&=1+\dashint_{B_1}\mathrm{div}w_h+\frac{1}{2}\dashint_{B_1}\big((\mathrm{div}w_h)^2-\mathrm{Tr}(\nabla w_h)^2\big)+\dashint_{B_1}\mathrm{det}\nabla w_h\\
&=1+3\d2\langle w,x\rangle+Q_{V_3}(w)+V_3(w)\,,
\end{split}
\end{align}
where the quadratic form $Q_{V_3}$ is defined as
\begin{equation}\label{quadratic_form_Q_V_3}
Q_{V_3}(w):=\frac{3}{2}\d2 \Big\langle w,(\mathrm{div}_{\s}w)x-\sum_{j=1}^3x_j\nabla_Tw^j\Big\rangle\,.
\end{equation}
\end{lemma}
\begin{remark}
As the reader might already know from the theory of null-Lagrangians or notice from the next proof, the above formulas actually hold true with $B_1$ being replaced by any other open bounded domain $U\subset \R^3$ with sufficiently regular boundary, and $u_h$ (resp. $w_h$) being replaced by any other interior extension of $u$ (resp. $w$), for which the previous bulk integrals are well defined. Since we only used the expressions for the harmonic extension, which is smooth in the interior of the unit ball, we have preferred to state the previous lemma in this particular form.    
\end{remark}
\begin{proof}[Proof of Lemma \ref{integral_identities}]
\EEE Regarding the proof of the identity \eqref{determinant_bulk_surface}, the determinant of $\nabla u_h:=(\partial_ju_h^i)_{1\leq i,j\leq3}$ can be rewritten as
\begin{equation}\label{different_expressions_for_jacobian_determinant}
\mathrm{det}\nabla u_h=\langle \nabla u_h^1,\nabla u_h^2\times \nabla u_h^3\rangle=\langle \nabla u_h^2,\nabla u_h^3\times \nabla u_h^1\rangle=\langle \nabla u_h^3,\nabla u_h^1\times \nabla u_h^2\rangle\,,
\end{equation}
where $a\times b\in \R^3$ denotes the exterior product of two vectors $a,b\in \R^3$.
Using the first identity in \eqref{different_expressions_for_jacobian_determinant}, and integrating by parts, we have
\begin{align}\label{first_expression_for_integral_of_jacobian}
\begin{split}
\int_{B_1}\mathrm{det}\nabla u_h&= \int_{B_1}\langle \nabla u_h^1,\nabla u_h^2\times \nabla u_h^3\rangle=\int_{B_1}\mathrm{div}\big(u_h^1(\nabla u_h^2\times \nabla u_h^3)\big)-\int_{B_1}u_h^1\mathrm{div}(\nabla u_h^2\times \nabla u_h^3)\\
&=\int_{\s}u_h^1\langle \nabla u_h^2\times \nabla u_h^3, x\rangle-\int_{B_1}u_h^1\big(\langle\mathrm{curl}\nabla u_h^2,\nabla u_h^3\rangle-\langle\nabla u_h^2,\mathrm{curl}\nabla u_h^3\rangle\big)\\
&=\int_{\s}u^1\left\langle \big((\partial_{\tau_1}u^2)\tau_1+(\partial_{\tau_2}u^2)\tau_2+(\partial_{\vec{\nu}}u_h^2)x\big)\times\big( (\partial_{\tau_1}u^3)\tau_1+(\partial_{\tau_2}u^3)\tau_2+(\partial_{\vec{\nu}}u_h^3)x\big), x\right\rangle\\
&=\int_{\s}u^1\big(\partial_{\tau_1}u^2\partial_{\tau_2}u^3-\partial_{\tau_2}u^2\partial_{\tau_1}u^3\big)\,.
\end{split}
\end{align}
Here, we have used the vector calculus identities
\begin{equation*}
\mathrm{div}(A\times B)=\langle \mathrm{curl}A,B\rangle-\langle A,\mathrm{curl}B\rangle, \quad \mathrm{curl}\nabla f=0  \quad \forall A,B\in C^2(\R^3;\R^3)\,,\ f\in C^2(\R^3;\R)\,.
\end{equation*} 
and we have written the full gradients $\nabla u_h^2$ and $\nabla u_h^3$ on $\s$ in terms of the local orthonormal oriented frame $\{\tau_1,\tau_2,x\}$ for $\s$, which is such that $\tau_1\times\tau_2=x, \tau_2\times x=\tau_1, x\times\tau_1=\tau_2$. Using the other two expressions from \eqref{different_expressions_for_jacobian_determinant} and arguing in the same manner, we also obtain
\begin{equation}\label{second_expression_for_integral_of_jacobian}
\int_{B_1}\mathrm{det}\nabla u_h=-\int_{\s}u^2\big(\partial_{\tau_1}u^1\partial_{\tau_2}u^3-\partial_{\tau_2}u^1\partial_{\tau_1}u^3\big)\,,
\end{equation}
and 
\begin{equation}\label{third_expression_for_integral_of_jacobian}
\int_{B_1}\mathrm{det}\nabla u_h=\int_{\s}u^3\big(\partial_{\tau_1}u^1\partial_{\tau_2}u^2-\partial_{\tau_2}u^1\partial_{\tau_1}u^2\big)\,.
\end{equation}
Therefore, summing \eqref{first_expression_for_integral_of_jacobian}-\eqref{third_expression_for_integral_of_jacobian}, and recalling the notation \eqref{determinant_of_signed_volume} (for $n=3$), we arrive at \eqref{determinant_bulk_surface}. 
Regarding \eqref{determinant_expansion_around_I}, the first equality is immmediate from the expansion of the determinant around the identity matrix $I_3$. The expression in the second line of \eqref{determinant_expansion_around_I} follows from the fact that the resulting terms can be written as boundary integrals in the following fashion.
Using again Stokes' theorem, we can calculate
\begin{align}\label{linear_term_V3}
\dashint_{B_1}\mathrm{div}w_h=3\d2 \langle w,x\rangle\,,
\end{align}
\begin{align*}
\int_{B_1}(\mathrm{div}w_h)^2&=\int_{B_1}\mathrm{div}\big((\mathrm{div}w_h) w_h\big)-\int_{B_1}\big\langle w_h,\nabla\mathrm{div}w_h\big\rangle=\int_{\s}\big\langle w,(\mathrm{div}w_h)x\big\rangle-\int_{B_1}\big\langle w_h,\nabla\mathrm{div}w_h\big\rangle\\
&=\int_{\s}\big\langle w,(\mathrm{div}_{\s}w)x\big\rangle+\int_{\s}\big\langle w,x\big\rangle\big\langle(\nabla w_h)x,x\big\rangle-\int_{B_1}\big\langle w_h,\nabla\mathrm{div}w_h\big\rangle\,,
\end{align*}
and
\begin{align*}
\int_{B_1}\mathrm{Tr}(\nabla w_h)^2&=\sum_{i,j=1}^3\int_{B_1}\partial_jw_h^i\partial_iw_h^j=\sum_{j=1}^3\int_{B_1}\partial_j\big(\langle w_h,\nabla w_h^j\rangle\big)-\sum_{i=1}^3\int_{B_1}w_h^i\partial_i(\mathrm{div}w_h)\\
&=\sum_{j=1}^3\int_{\s}\langle w,x_j\nabla w_h^j\rangle-\int_{B_1}\langle w_h,\nabla \mathrm{div}w_h\rangle\\
&=\int_{\s}\Big\langle w,\sum_{j=1}^3x_j\nabla_Tw^j\Big\rangle+\int_{\s}\langle w,x\rangle\langle(\nabla w_h)x,x\rangle-\int_{B_1}\langle w_h,\nabla \mathrm{div}w_h\rangle\,.
\end{align*}
Subtracting the last two identities we arrive at
\begin{equation}\label{quadratic_term_V3}
\frac{1}{2}\dashint_{B_1}\big((\mathrm{div}w_h)^2-\mathrm{Tr}(\nabla w_h)^2\big)=\frac{3}{2}\d2 \Big\langle w,(\mathrm{div}_{\s}w)x-\sum_{j=1}^3x_j\nabla_Tw^j\Big\rangle=: Q_{V_3}(w)\,.
\end{equation}
Hence, the second equality in \eqref{determinant_expansion_around_I} follows from \eqref{linear_term_V3}, \eqref{quadratic_term_V3} and \eqref{determinant_bulk_surface} for $w$ in the place of $u$ here.
\end{proof}
\EEE
Next, we calculate in detail the Taylor expansions up to second order of the geometric quantities that we used in the main body of the paper. The computations presented in the following lemma are formal, and we assume without further clarification that the maps in consideration are always regular enough so that we can perform the expansions. In the case $n=3$, we had directly performed the expansion of the $2$-Dirichlet energy of $u$ around the $\mathrm{id}_{\S}$ in the proof of Lemma \ref{quadratization_of_conformal_deficit}, and the one for $V_3(u)$ was performed in the previous Lemma \ref{integral_identities}. Thus, the focus in the next lemma is mostly on the case $n\geq 4$. \AAA Moreover, notice that in Subsection \ref{Subsection 4.2} we had already translated and scaled the initial map $u$ properly \AAA(recall \eqref{local_family_of_maps_sufficient_for_conformal_case} and \eqref{fixing_scale}), which we will also assume for convenience next.\EEE

\AAA
\begin{lemma}\label{expansions_n}
Let $n\geq 3$, $u:\S\mapsto \R^n$ $\rm{(}$be sufficiently regular$\rm{)}$ and as always, $w:=u-\mathrm{id}_{\S}$. Assuming that $\ds w=0, \ds \langle w,x\rangle =0$, and recalling the notation introduced in  \eqref{AM-GM-isoperimetric-inequality}, we can formally write 
\begin{align}\label{formulas_for_the_expansions}
\begin{split}
(i)&\quad [D_{n-1}(u)]^{\frac{n}{n-1}}:=1
+ Q_{[D_{n-1}]^{\frac{n}{n-1}}}(w)+R_{1,n}(w)\,,\\ 
(ii)&\quad [P_{n-1}(u)]^{\frac{n}{n-1}}:=1
+ Q_{[P_{n-1}]^{\frac{n}{n-1}}}(w)+R_{2,n}(w)\,, \\
(iii)&\quad\quad \quad \quad V_n(u):=1
+ Q_{V_n}(w)+R_{3,n}(w)\,, 
\end{split}
\end{align}
where the corresponding quadratic forms are given by the expressions
\begin{align}\label{formulas_for_the_quadratic_forms}
\begin{split}
(i)&\quad Q_{[D_{n-1}]^{\frac{n}{n-1}}}(w):=\frac{1}{2}\frac{n}{n-1}\ds\Big(| \t w|^2+\frac{n-3}{n-1}(\mathrm{div}_{\S}w)^2\Big)\,,\\ 
(ii)&\quad Q_{[P_{n-1}]^{\frac{n}{n-1}}}(w):=\frac{1}{2}\frac{n}{n-1}\ds\Big(|\t w|^2+(\mathrm{div}_{\S}w)^2-2\left|(P_T^t\nabla_Tw)_{\mathrm{sym}}\right|^2\Big)\,, \\
(iii)&\quad \quad \quad \ \  Q_{V_n}(w):= \frac{n}{2}\ds\big\langle w,(\textup{div}_{\S}w)x-\sum_{j=1}^nx_j\nabla_Tw^j\big\rangle\,,
\end{split}
\end{align}
and the remainder terms $(R_{i,n}(w))_{i=1,2,3}$ are of higher order.

\end{lemma}

\begin{proof}
\EEE Regarding the expansion of the $(n-1)$-Dirichlet energy, in the case $n=3$ the calculation was performed in the proof of Lemma \ref{quadratization_of_conformal_deficit} (see \eqref{D_3_expansion}). For $n\geq 4$, by using the fact that
\begin{equation}\label{average_div_0}
\ds \mathrm{div}_{\S}w=(n-1)\ds \langle w,x\rangle =0\,,
\end{equation}
we can formally calculate \EEE	
\begin{flalign}\label{Qc}
\begin{split}
[D_{n-1}(u)]^{\frac{n}{n-1}}=&\ \left(\ds \left(\frac{|\t u|^2}{n-1}\right)^\frac{n-1}{2}\right)^\frac{n}{n-1}=\left(\ds \left(1+\frac{2}{n-1}\mathrm{div}_{\S}w+\frac{|\t w|^2}{n-1}\right)^\frac{n-1}{2}\right)^\frac{n}{n-1} \\
=&\left(\ds \left(1+\mathrm{div}_{\S}w+\frac{1}{2}|\t w|^2+\frac{1}{2}\frac{n-3}{n-1}(\mathrm{div}_{\S}w)^2+\mathcal{O}\big(|\t w |^3\big)\right)\right)^\frac{n}{n-1} \\
=&\left[1+\ds\left(\frac{1}{2}|\t w|^2+\frac{1}{2}\frac{n-3}{n-1}(\mathrm{div}_{\S}w)^2\right)+\ds\mathcal{O}\big(|\t w |^3\big)\right]^{\frac{n}{n-1}} \\
=&\ 1+\frac{n}{2(n-1)}\ds |\t w|^2+\frac{n(n-3)}{2(n-1)^2}\ds(\mathrm{div}_{\S}w)^2  \\
&\ \ +\ds\mathcal{O}\big(|\t w|^3\big)+\mathcal{O}\Bigg(\left(\ds|\t w|^2\right)^2\Bigg)\,.
\end{split}
\end{flalign}
Therefore, \AAA \eqref{Qc} gives the expansion \eqref{formulas_for_the_expansions}$(i)$, with the formula \eqref{formulas_for_the_quadratic_forms}$(i)$ for the quadratic term and the growth behaviour for the higher order term $R_{1,n}(w)$. \EEE

For the expansion of the generalized perimeter-term around the $\mathrm{id}_{\S}$, for every $n\geq 3$ we have, 
\begin{align}\label{perimeter_expansion_1}
\ds \sqrt{\mathrm{det}(\t u^t\t u)}&=\ds\sqrt{\mathrm{det}(I_x+K(w))}\,,
\end{align}
where 
\begin{equation}\label{symmetrized_nonlinear_part}
K(w):=2(P_T^t\nabla_Tw)_{\mathrm{sym}}+\nabla_Tw^t\nabla_Tw\,.
\end{equation}
The Taylor expansion of the determinant around $I_x$ gives
\begin{equation}\label{Taylor_of_determinant}
\mathrm{det}(I_x+K(w))=1+\mathrm{Tr}K(w)+\frac{1}{2}\Big((\mathrm{Tr}K(w))^2-\mathrm{Tr}(K(w)^2)\Big)+\mathcal{O}(|K(w)|^3)\,,
\end{equation}
and since in our case,
\begin{equation}\label{trace_identity}
\mathrm{Tr}K(w)=2\mathrm{div}_{\S}w+|\nabla_Tw|^2, (\mathrm{Tr}K(w))^2-\mathrm{Tr}(K(w)^2)=4\big((\mathrm{div}_{\S}w)^2-|(P_T^t\nabla_Tw)_{\mathrm{sym}}|^2\big)+\mathcal{O}(|\nabla_Tw|^3)\,,
\end{equation}
\AAA by the identities \eqref{perimeter_expansion_1}-\eqref{trace_identity}, \EEE we obtain the formal expansion 
\begin{equation}\label{perimeter_expansion_2}
\ds \sqrt{\mathrm{det}(\t u^t\t u)}=\ds\sqrt{1+\Theta(w)+\mathcal{O}(|\nabla_Tw|^3)}\,,
\end{equation}
where
\begin{equation}\label{Theta_equation}
\Theta(w):=2\mathrm{div}_{\S}w+|\nabla_Tw|^2+2(\mathrm{div}_{\S}w)^2-2\left|(P_T^t\nabla_Tw)_{\mathrm{sym}}\right|^2\,.
\end{equation}
Since $(\Theta(w))^2=4(\mathrm{div}_{\S}w)^2+\mathcal{O}(|\t w|^3)$, we can perform a Taylor expansion of the square root inside the integral \AAA in \eqref{perimeter_expansion_2} and use \eqref{average_div_0}, \eqref{Theta_equation}, \EEE to get
\begin{equation*}
\ds \sqrt{\mathrm{det}(\t u^t\t u)}
=1+\frac{1}{2}\ds \Big(|\t w|^2+(\mathrm{div}_{\S}w)^2-2\left|(P_T^t\nabla_Tw)_{\mathrm{sym}}\right|^2\big)+\ds\mathcal{O}\left(|\t w|^3\right)\,.
\end{equation*}
A final Taylor expansion of the function $t\mapsto (1+t)^{\frac{n}{n-1}}$ around 0 gives,
\begin{equation}\label{Qp}
[P_{n-1}(u)]^{\frac{n}{n-1}}=1+\frac{1}{2}\frac{n}{n-1}\ds\left(|\t w|^2+(\mathrm{div}_{\S}w)^2-2\left|(P_T^t\nabla_Tw)_{\mathrm{sym}} \right|^2\right) +\ds\mathcal{O}\left(|\t w|^3\right)\,.
\end{equation}
Therefore, \AAA \eqref{Qp} gives the expansion \eqref{formulas_for_the_expansions}$(ii)$, with the formula \eqref{formulas_for_the_quadratic_forms}$(ii)$ for the corresponding quadratic term and the growth behaviour for the higher order term $R_{2,n}(w)$. 

The expansion of the generalized signed-volume $V_n(u)$ in \eqref{definition_of_signed_volume},
\eqref{determinant_of_signed_volume} around the  $\mathrm{id}_{\S}$, \AAA for $n=3$ was given in Lemma \ref{integral_identities}. An \textit{intrinsic way} to perform the calculation in every dimension $n\geq 3$ is the following.  
\begin{align}\label{V_n_first}
\begin{split}
V_n(u):&= \ds \Big\langle u, \bigwedge_{i=1}^{n-1}\partial_{\tau_i}u\Big\rangle=\ \ds \Big\langle w+x, \bigwedge_{i=1}^{n-1}\left(\partial_{\tau_i}w+\partial_{\tau_i}x\right)\Big\rangle\\
&= \ds\sum_{k=0}^{n-1}\sum_{|\alpha|=k}\sigma(\alpha,\bar\alpha)\ \Big\langle w+x, \Big(\bigwedge_{\alpha}\partial_{\tau_{\alpha}}
w\Big)\wedge\ \big(\bigwedge_{\bar\alpha}\partial_{\tau_{\bar\alpha}}x\big)\Big\rangle \\
&= I_{n,0}(w)+I_{n,1}(w)+I_{n,2}(w)+I_{n,3}(w)\,. 
\end{split}
\end{align}

Here, we have used standard multiindex notation. For every $k\in \{0,1,\dots,n-1\}$ and for every multiindex $\alpha:=(\alpha_1,\dots,\alpha_k)$, where $(a_i)_{i=1,\dots,k}\in \mathbb{N}$ with $1\leq\alpha_1<\dots<\alpha_k\leq n-1$ we denote $\bar\alpha$ its complementary multiindex (with its entries also in increasing order), $\sigma(\alpha,\bar\alpha)$ the sign of the permutation that maps $(\alpha,\bar\alpha)$ to the standard ordering $(1,\dots,n)$ and $$\partial_{\tau_{\alpha}}w:=\partial_{\tau_{\alpha_1}}w\wedge\dots\wedge\partial_{\tau_{\alpha_k}}w\,.$$ We have also denoted by $(I_{n,i}(w))_{i=0,1,2}$ the zeroth, first and second order terms with respect to $w$ and $\t w$ in the expansion of $V_n(u)$ around the $\mathrm{id}_{\S}$ respectively, and by $I_{n,3}(w)$ the remaining term which is a \textit{polynomial} of order at least $3$ and at most $n$ in $w$ and its first derivatives.\
Keeping in mind that $\partial_{\tau_i}x=\tau_i$ for $i=1,\dots,n-1$ and that by an abuse of notation, $\tau_1\wedge\dots\wedge\tau_{n-1}\equiv x$, we can compute each term separately.
\begin{align}\label{I_0_I_1}
I_{n,0}(w)&=\ \ds \langle x, \partial_{\tau_1}x\wedge\dots\wedge\partial_{\tau_{n-1}}x\rangle=\ds |x|^2=1\,.\\
I_{n,1}(w) &=\ \ds \langle w,x\rangle+\sum_{i=1}^{n-1}\ds \Big\langle x, \Big(\bigwedge_{l=1}^{i-1}\tau_l\Big)\wedge\partial_{\tau_i}w\wedge\Big(\bigwedge_{m=i+1}^{n-1}\tau_m\Big)\Big\rangle\nonumber\\
&=\ds \langle w,x\rangle\ +\ds \sum_{i=1}^{n-1}\langle\partial_{\tau_i}w,\tau_i\rangle =\ds \langle w,x\rangle+\ds \mathrm{div}_{\S}w=0\,, 
\end{align}
the last equality following from \eqref{average_div_0}.
For the quadratic term, we can write it as 
\begin{equation}\label{quadratic_term_split}
I_{n.2}(w):=I_{n,2,1}(w)+I_{n,2,2}(w)\,, \quad \rm{where}
\end{equation}
\begin{align}\label{I_n_2_1}
\hspace{-1.5em}I_{n,2,1}(w)&= \sum_{i=1}^{n-1}\ds \Big\langle w, \Big(\bigwedge_{l=1}^{i-1}\partial_{\tau_l}x\Big)\wedge\partial_{\tau_i}w\wedge\Big(\bigwedge_{m=i+1}^{n-1}\partial_{\tau_m}x\Big)\Big\rangle\nonumber\\
\ &=\ \sum_{i=1}^{n-1}\ds \Big\langle w, \Big(\bigwedge_{l=1}^{i-1}\tau_l\Big)\wedge\Big(\sum_{j=1}^{n-1}\langle\partial_{\tau_i}w,\tau_j\rangle\tau_j+\langle\partial_{\tau_i}w,x\rangle x\Big)\wedge\Big(\bigwedge_{m=i+1}^{n-1}\tau_m\Big)\Big\rangle \nonumber\\
\ &=\ \ds \Big(\mathrm{div}_{\S}w\langle w,x\rangle- \sum_{i=1}^{n-1}\langle w,\tau_i \rangle\langle\partial_{\tau_i}w,x\rangle\Big) =\ \ds \Big\langle w, (\mathrm{div}_{\S}w)x-\sum_{j=1}^n x_j\t w^j\Big\rangle\,.
\end{align}
The change of sign in the last line of \eqref{I_n_2_1} is due to orientation reasons, since we have taken the local orthonormal basis $\{ \tau_1,\dots,\tau_{n-1}\}$ of $T_x\S$ in such a way that at every $x\in \S$ the set of vectors $\{ \tau_1(x),\dots,\tau_{n-1}(x),x\}$ is a positively oriented frame of $\R^n$. Similarly, 
\begin{align}\label{I_n_2_2}
I_{n,2,2}(w)&=\ \ds \sum_{1\leq i<j\leq n-1}\Big\langle x, \Big(\bigwedge_{k=1}^{i-1}\partial_{\tau_k}x\Big)\wedge\partial_{\tau_i}w\wedge\Big(\bigwedge_{l=i+1}^{j-1}\partial_{\tau_l}x\Big)\wedge\partial_{\tau_j}w\wedge\Big(\bigwedge_{m=j+1}^{n-1}\partial_{\tau_m}x \Big) \Big\rangle\nonumber \\
&=\ \frac{1}{2}\ \ds \sum_{1\leq i,j\leq n-1}\left(\langle \partial_{\tau_i}w,\tau_i\rangle\langle \partial_{\tau_j}w,\tau_j\rangle-\langle\partial_{\tau_i}w,\tau_j\rangle\langle\partial_{\tau_j}w,\tau_i\rangle\right)\,.
\end{align}
After integrating by parts it is easy to see that the first term in the last line of \eqref{I_n_2_2} is
\begin{equation}\label{first_subsummand}
\ds \sum_{1\leq i,j\leq n-1}\langle \partial_{\tau_i}w,\tau_i\rangle\langle \partial_{\tau_j}w,\tau_j\rangle=\ \ds\langle w,(n-1)(\mathrm{div}_{\S}w)x-\t \mathrm{div}_{\S}w\rangle\,,
\end{equation}
while the second one therein is
\begin{equation}\label{second_subsummand}
\ds \sum_{1\leq i,j\leq n-1}\langle \partial_{\tau_i}w,\tau_j\rangle\langle \partial_{\tau_j}w,\tau_i\rangle \ = \ \ds \langle w,(\mathrm{div}_{\S}w)x-\t\mathrm{div}_{\S}w + (n-2)\sum_{j=1}^n x_j\t w^j\rangle\,.
\end{equation}
Subtracting \eqref{first_subsummand} and \eqref{second_subsummand} by parts, we infer  that \eqref{I_n_2_2} implies 
\begin{equation}\label{quadratic_volume_second}
I_{n,2,2}(w)=\ \left(\frac{n}{2}-1\right) \ds \big\langle w, \ (\mathrm{div}_{\S}w)x-\sum_{j=1}^nx_j\t w^j\big\rangle\,,
\end{equation}
and therefore, \eqref{I_0_I_1}-\eqref{quadratic_volume_second} give the desired expansion \eqref{formulas_for_the_expansions}$(iii)$, with the formula \eqref{formulas_for_the_quadratic_forms}$(iii)$ for the corresponding quadratic term $Q_{V_n}$. Regarding the remainder term $R_{3,n}$ it is easy to see from \eqref{V_n_first} that it has the algebraic structure
\begin{equation}\label{algebraic_structure_of_V_n_remainder}
R_{3,n}(w) =\sum_{k=2}^{n-1} \ds \big\langle w,A_{n,k}(w)\big\rangle 
\end{equation}
where for each $k=2,\dots,n-1$, $A_{n,k}$ is a nonlinear first order differential operator that is a \textit{homogeneous polynomial} of order $k$ in the first derivatives of $w$. \qedhere
\end{proof}

\begin{remark}\label{higher_order_terms_negligible}
From the growth behaviour of the higher order terms $(R_{i,n})_{i=1,2,3}$ in \eqref{formulas_for_the_expansions}, as these can be derived from \eqref{Qc}, \eqref{Qp} and \eqref{algebraic_structure_of_V_n_remainder}, it is immediate to deduce the following simple fact. Even in the case $n\geq 4$, if $u\in W^{1,\infty}(\S;\R^n)$ is such that $\|\t u-P_T\|_{L^{\infty}(\S)}\leq \theta$ for some $\theta\in (0,1)$ sufficiently small (as in the setting of Corollary \ref{fake_nonlinear_estimate_higher_dim}, see also \eqref{expansion_conf_higher_dim}), and moreover $\ds u=0,\ \ds\langle u,x\rangle=1$, then for $w:=u-\mathrm{id}_{\S}$, one has
\begin{equation*}
|R_{1,n}(w)|+|R_{2,n}(w)|+|R_{3,n}(w)|\leq c_\theta\ds |\t w|^2\,,
\end{equation*}
for some constant $c_\theta\in (0,1)$, such that $c_\theta\to 0$ as $\theta\to0$.
\end{remark}
\EEE

Let us conclude this appendix by giving a proof of Korn's identity on $\S$. 

\begin{proof}[\bf{Proof of Lemma \ref{Korn's identity_sphere}}] We have
\begin{equation*}
\ds\left|(P_T^t\t w)_{\mathrm{sym}}\right|^2=\frac{1}{2}\ds |P_T^t\t w|^2+\frac{1}{2}\ds\mathrm{Tr}\big((P_T^t\t w)^2\big)\,,
\end{equation*}
and \AAA recalling \eqref{I_n_2_2}, \EEE
\begin{equation*}
\ds \mathrm{Tr}\big((P_T^t\t w)^2\big)
=\ds \sum_{i,j=1}^{n-1}\langle\partial_{\tau_i}w,\tau_j\rangle\langle \partial_{\tau_j} w,\tau_i\rangle= \ds (\mathrm{div}_{\S} w)^2-2I_{n,2,2}(w)\,.
\end{equation*}
\AAA By the definition of $Q_{V_n}(w)$ in \eqref{formulas_for_the_quadratic_forms}($iii$) and \eqref{quadratic_volume_second} we have that $
I_{n,2,2}(w)=\ \frac{n-2}{n}Q_{V_n}(w)$, \EEE and Korn's identity \eqref{Korn's identity_sphere} follows immediately.
\end{proof}

\AAA \section{Spherical Harmonics}\label{sec:C} \EEE
We first recall that in local coordinates, the spherical Laplace-Beltrami operator $-\Delta_{\S}$ is given for every $f\in C^2(\S;\R^n)$ through the expression
\begin{equation}\label{laplace_Beltrami_definition}
-\Delta_{\S}f:=-\frac{1}{\sqrt{\mathrm{det}(g)}}\sum_{i,j=1}^{n-1}\frac{\partial}{\partial \tau_i}\Big(\sqrt{|g|}g^{ij}\frac{\partial f}{\partial \tau_j}\Big)\,.
\end{equation}
Here, $g:=(g_{ij})_{i,j=1,\dots,n-1}$ denotes the standard round metric on $\S$, $(g^{ij})_{i,j=1,\dots,n-1}$ its inverse, and $\{\tau_1,\dots,\tau_{n-1}\}$ is the local orthonormal frame on $\S$ as introduced in Section \ref{Notation}. It is well known that $L^2(\S;\R^n)$ admits an orthonormal basis consisting of eigenfunctions of $-\Delta_{\S}$. In particular, for every $k \in \mathbb{N}$  there exists a finite number (denoted by $G_{n,k}$) of \textit{mutually $L^2$- $\mathrm{(}$and actually $W^{1,2}$-$\mathrm{)}$ orthogonal functions $(\psi_{n,k,j})_{j=1,\dots,G_{n,k}}\subset W^{1,2}(\S;\R^n)$}, which are called the \textit{vector valued $k$-th order spherical harmonics}, are restrictions on $\S$ of ($\R^n$-valued) homogeneous harmonic polynomials in $\mathbb{R}^n$ of degree $k$ respectively, and satisfy 
\begin{align}\label{properties}
-\Delta_{\S}\psi_{n,k,j}=\lambda_{n,k} \psi_{n,k,j} \quad \forall k \in \mathbb{N} \ \mathrm{and }\ j=1,2,\dots,G_{n,k}, \ \mathrm{where }\ \lambda_{n,k}:= k(k+n-2)\,.
\end{align}
In particular,
\begin{equation}\label{sharmonics}
\ds|\nabla_{T}\psi_{n,k,j}|^2=\lambda_{n,k} \ds|\psi_{n,k,j}|^2\quad \forall k\in \mathbb{N}\ \mathrm{and}\ j=1,2,\dots,G_{n,k}\,.
\end{equation}
In the scalar case, $G_{n,0}=1$ (with trivial eigenfunction the constant 1),\ $G_{n,1}=n$ (with the first order spherical harmonics being the coordinate functions $\psi_{n,1,j}(x):=\frac{x_j}{\sqrt{\omega_n}}$), and $G_{n,k}$ = ${n+k-1}\choose{k}$-${n+k-3}\choose{k-2}$ for $k\geq 2$. The reader can refer to \cite{notesspherical}, \cite{groemer1996geometric} for more information on spherical harmonics.

%
	

\begin{remark}\label{L2_orthogonality_of_spherical_harmonics}	
The following \textit{Parseval identities on $\S$} hold true: If $u \in W^{1,2}(\S;\R^n)$ with its Fourier expansion in spherical harmonics being $u=\sum_{k=0}^{\infty}\sum_{j=1}^{G_{n,k}}\alpha_{n,k,j}\psi_{n,k,j}$, then
\begin{equation}\label{Parseval}
\int_{\S}|u|^2=\sum_{k=0}^{\infty}\sum_{j=1}^{G_{n,k}}(\alpha_{n,k,j})^2 \ \mathrm{\ and\ } \ \int_{\S}|\nabla_{T} u|^2=\sum_{k=1}^{\infty}\sum_{j=1}^{G_{n,k}}\lambda_{n,k}(\alpha_{n,k,j})^2\,.  
\end{equation}
The \textit{sharp Poincare inequality} for maps $u\in W^{1,2}(\S;\R^n)$ is then easily deduced. Since $\lambda_{n,k}\geq n-1$ for every $k\geq 1$, from \eqref{Parseval} we obtain
\begin{equation}\label{Poincare}
\int_{\S}|\nabla_{T}u|^2 \geq(n-1)\sum_{k=1}^{\infty}\sum_{j=1}^{G_{n,k}}(\alpha_{n,k,j})^2= (n-1)\int_{\S}\Big|u-\ds u\Big|^2\,.
\end{equation}
\end{remark}
Of course, depending on the number of vanishing first \textit{Fourier modes} in the expansion of $u$, the constant in the above inequality can be improved in an obvious way. By expanding a function in spherical harmonics one can often obtain useful estimates. In the next lemma, we mention two of them that we have used earlier in the paper.
\begin{lemma}\label{harmest}
If $u\in W^{1,2}(\S;\R^n)$, and $u_h:\overline {B_1}\mapsto \R^n$ denotes its $\rm{(}$componentwise$\rm{)}$ harmonic extension, the following estimates hold true:
\begin{align}\label{basic_harmonic_estimates}
\dashint_{B_1} |\nabla u_h|^2\ &\leq \frac{n}{n-1}\ \ds |\t u|^2\,,\\
\frac{n}{n-1}\ \ds |\t u|^2 &\leq \ds |\nabla u_h|^2 \leq 2\ \ds |\t u|^2\,.
\end{align}
\end{lemma}
\begin{proof}
Let us give the proof of these two simple estimates in the case that $u$ is scalar-valued, the case of vector-valued $u$ being an immediate consequence. We write again 
\begin{equation*}
u=\sum_{k=0}^{\infty}\sum_{j=1}^{G_{n,k}}\alpha_{n,k,j}\psi_{n,k,j}\,,
\end{equation*}
and therefore, its harmonic extension can be written in polar coordinates $(r,\theta)\in [0,1]\times\S$ as
\begin{equation*}
u_h(r,\theta)=\sum_{k=0}^{\infty}\sum_{j=1}^{G_{n,k}}r^k\alpha_{n,k,j}\psi_{n,k,j}(\theta)\,.
\end{equation*}
For \ref{basic_harmonic_estimates}, we write
\begin{align*}
\dashint_{B_1} |\nabla u_h|^2 &=\dashint_{B_1} \mathrm{div}(u_h\nabla u_h) = n\ds u\ \partial_{\vec \nu}u_h=\sum_{k=0}^{\infty}\sum_{j=1}^{G_{n,k}}nk(\alpha_{n,k,j})^2=\sum_{k=1}^{\infty}\sum_{j=1}^{G_{n,k}}
\frac{n\lambda_{n,k}}{k+n-2}(\alpha_{n,k,j})^2\,,
\end{align*}
while for \hyperref[basic_harmonic_estimates]{(C.7)}, we write
\begin{align*}
\ds |\nabla u_h|^2= \ds |\t u|^2+\ds |\partial_{\vec \nu}u_h|^2&
=\ds |\t u|^2+\sum_{k=1}^{\infty}\sum_{j=1}^{G_{n,k}}\frac{k}{k+n-2}\lambda_{n,k}(\alpha_{n,k,j})^2\,.
\end{align*}
Since $\frac{1}{n-1}\leq \frac{k}{k+n-2}\leq 1$ for every $k\geq 1$, the desired estimates follows immediately by the above identities and \eqref{Parseval}.
\end{proof}

\EEE

\typeout{References}


\vspace{0.7em}

$^1$ Institut für Mathematik, Universität Leipzig, 
Germany\\ 
(Stephan.Luckhaus@math.uni-leipzig.de)
\\[-10pt] 

$^2$ Institut für Numerische und Angewandte Mathematik, Universität Münster, 
Germany\\
(konstantinos.zemas@uni-muenster.de 
)
\end{document}